\documentclass[reqno]{amsart}
\usepackage{amssymb,amsmath,hyperref}
\usepackage{amsrefs, float}
\usepackage[foot]{amsaddr}
\usepackage{bbold,stackrel}
 
\newcommand{\Z}{{\textsf{\textup{Z}}}}

\newtheorem{thm}{Theorem}

\newtheorem{cor}[thm]{Corollary}
\newtheorem{defi}[thm]{Definition}
\newtheorem{rem}[thm]{Remark}
\newtheorem{nota}[thm]{Notation}

\newtheorem{princ}[thm]{Principle}

\newtheorem{ack}[thm]{Acknowledgement}

\newtheorem*{tempo*}{Template}

\newcommand\be{\begin{equation}}
\newcommand\ee{\end{equation}} 

\usepackage{amsmath,amsfonts} 

\usepackage[applemac]{inputenc}

\def\bdefi{\begin{defi}\rm}
\def\edefi{\end{defi}}
\def\bnota{\begin{nota}\rm}
\def\enota{\end{nota}}
\def\FIVE{\Pi_{1}^{1}\text{-\textup{\textsf{CA}}}_{0}}
\def\SIX{\Pi_{2}^{1}\text{-\textsf{\textup{CA}}}_{0}}
\def\SIXk{\Pi_{k}^{1}\text{-\textsf{\textup{CA}}}_{0}}
\def\SIXko{\Pi_{k+1}^{1}\text{-\textsf{\textup{CA}}}_{0}}
\def\SIXK{\Pi_{k}^{1}\text{-\textsf{\textup{CA}}}_{0}^{\omega}}

\def\ATR{\textup{\textsf{ATR}}}
\def\STR{\Sigma_{}\text{-\textup{\textsf{TR}}}}
\def\ZFC{\textup{\textsf{ZFC}}}
\def\ZF{\textup{\textsf{ZF}}}

\def\L{\textsf{\textup{L}}}

 \def\r{\mathbb{r}}

\def\RCA{\textup{\textsf{RCA}}}
\def\({\textup{(}}
\def\){\textup{)}}
\def\WO{\textup{\textsf{WO}}}
\def\c{\textup{\textsf{c}}}

\def\RCAo{\textup{\textsf{RCA}}_{0}^{\omega}}
\def\ACAo{\textup{\textsf{ACA}}_{0}^{\omega}}
\def\WKL{\textup{\textsf{WKL}}}

\def\WWKL{\textup{\textsf{WWKL}}}
\def\bye{\end{document}}
\def\N{{\mathbb  N}}
\def\Q{{\mathbb  Q}}
\def\R{{\mathbb  R}}
\def\A{{\textsf{\textup{A}}}}

\def\SS{\textup{\textsf{S}}}

\def\MUC{\textup{\textsf{MUC}}}

\def\di{\rightarrow}

\def\asa{\leftrightarrow}

\def\ACA{\textup{\textsf{ACA}}}

\def\QFAC{\textup{\textsf{QF-AC}}}
\def\PRA{\textup{\textsf{PRA}}}
\def\NIN{\textup{\textsf{NIN}}}
\def\NBI{\textup{\textsf{NBI}}}
\def\HBC{\textup{\textsf{HBC}}}

\def\CLO{\textup{\textsf{CLO}}}
\def\OLC{\textup{\textsf{OLC}}}

\def\open{\textup{\textsf{open}}}

\def\CIT{\textup{\textsf{CIT}}}
\def\CBT{\textup{\textsf{CBT}}}

\def\u{\textup{\textsf{u}}}

\def\BOOT{\textup{\textsf{BOOT}}}

\def\MOT{\textup{\textsf{MOT}}}

\def\MON{\textup{\textsf{MON}}}
\def\AS{\textup{\textsf{AS}}}

\def\IND{\textup{\textsf{IND}}}
\def\NFP{\textup{\textsf{NFP}}}

\def\HBU{\textup{\textsf{HBU}}}

\def\RANGE{\textup{\textsf{RANGE}}}
\def\range{\textup{\textsf{range}}}
\def\SEP{\textup{\textsf{SEP}}}
\def\PIT{\textup{\textsf{PIT}}}

\def\CAU{\textup{\textsf{CAU}}}

\def\WHBU{\textup{\textsf{WHBU}}}

\def\net{\textup{\textsf{net}}}
\def\mod{\textup{\textsf{mod}}}
\def\lex{\textup{\textsf{lex}}}
\def\w{\textup{\textsf{w}}}
\def\BW{\textup{\textsf{BW}}}

\def\seq{\textup{\textsf{seq}}}

\def\LIND{\textup{\textsf{LIND}}}
\def\LIN{\textup{\textsf{LIN}}}

\def\RM{\textup{\textsf{rm}}}

\def\COH{\textup{\textsf{COH}}}

\def\MCT{\textup{\textsf{MCT}}}

\def\PST{\textup{\textsf{PST}}}
\def\eps{\varepsilon}

\def\X{\textup{\textsf{X}}}

\def\ECF{\textup{\textsf{ECF}}}
\def\PECF{\textup{\textsf{PECF}}}

\usepackage{graphicx}
\usepackage{tikz}
\usetikzlibrary{matrix, shapes.misc}

\numberwithin{equation}{section}
\numberwithin{thm}{section}

\usepackage{comment}

\begin{document}
\title{Plato and the foundations of mathematics}
\author{Sam Sanders}
\address{Department of Mathematics, TU Darmstadt}
\email{sasander@me.com}
\begin{abstract}
Plato is well-known in mathematics for the eponymous foundational philosophy \emph{Platonism} based on ideal objects.  
Plato's \emph{allegory of the cave} provides a powerful visual illustration of the idea that we only have access to shadows or
reflections of these ideal objects.   An inquisitive mind might then wonder what the current foundations of mathematics, like e.g.\ \emph{Reverse Mathematics} and the associated \emph{G\"odel hierarchy}, 
are reflections of.  In this paper, we identify a hierarchy in higher-order arithmetic that maps to the Big Five of Reverse Mathematics under the canonical embedding of higher-order into second-order 
arithmetic.  Conceptually pleasing, the latter mapping replaces uncountable objects by countable `codes', i.e.\ the very practise of formalising mathematics in second-order arithmetic. 
This higher-order hierarchy can be defined in Hilbert-Bernays' \emph{Grundlagen}, the spiritual ancestor of second-order arithmetic, while the associated embedding preserves equivalences.   
Also, in contrast to Kohlenbach's hierarchy based on \emph{discontinuity}, our hierarchy can be formulated in terms of (classically valid) continuity axioms from Brouwer's intuitionistic mathematics. 
Moreover, the higher-order counterpart of sequences is provided by nets, aka Moore-Smith sequences, while the gauge integral is the correct generalisation of the Riemann integral. 
For all these reasons, we baptise our higher-order hierarchy the \emph{Plato hierarchy}. 
\end{abstract}


\maketitle
\thispagestyle{empty}

\vspace{-0.2cm}
\section{Introduction}\label{schintro}
\subsection{Plato, Platonism, and the Plato hierarchy}
The Greek philosopher Plato is perhaps best known in mathematics and related fields for the eponymous philosophy of mathematics \emph{Platonism}. 
The OED entry for Platonism reads as follows.  
\begin{quote}
the theory that mathematical objects are objective, timeless entities, independent of the physical world and the symbols that represent them.
\end{quote}
Platonism postulates the existence of ideal or abstract objects, while Plato's \emph{allegory of the cave} provides a powerful illustration of the 
idea that we only have access to very limited reflections (or: shadows) of these ideal or abstract objects, as expressed by G\"odel in \cite{gcw3}*{p.\ 323}.  Taking this view seriously, we may ask 
 the following -perhaps uncomfortable- questions: \emph{what are the current foundations of mathematics reflections of? What is the nature of this reflection?}  
In this paper, we provide precise answers to these questions for a fragment of the foundations of mathematics, namely \emph{Reverse Mathematics} and the \emph{G\"odel hierarchy}. 
We hereafter assume familiarity with these italicised notions; an introduction is in Section~\ref{lolol}.  

\smallskip

In a nutshell, we identify a hierarchy in higher-order arithmetic that maps to the Big Five of Reverse Mathematics under the canonical `$\ECF$' (see Remark \ref{ECF}) embedding of higher-order into second-order 
arithmetic.  Conceptually pleasing, $\ECF$ replaces uncountable objects by countable `codes', i.e.\ the very practise of formalising mathematics in second-order arithmetic. 
Our higher-order hierarchy can be defined in Hilbert-Bernays' \emph{Grundlagen der Mathematik} (\cites{hillebilly, hillebilly2}), the spiritual ancestor of second-order arithmetic, while the associated $\ECF$ embedding preserves equivalences.   
Moreover, the higher-order counterpart of sequences is provided by nets (aka Moore-Smith sequences; see Section \ref{intronet}), while the gauge integral (\cite{dagsamIII}*{\S3}) is the correct generalisation of the Riemann integral. 
The correct notion of `open set' shall be seen to be uncountable unions of open balls (and not characteristic functions as in \cite{dagsamVII, samnetspilot}).   Finally,
our higher-order hierarchy can be formulated in terms of \emph{classically valid} continuity axioms of intuitionistic mathematics, called \emph{neighbourhood function principle}, in contrast\footnote{It should be noted that $\ECF$ converts the existence of a discontinuous function to `$0=1$', as discussed in Remark \ref{ECF}.  The (classically valid) intuitionistic axiom in question is $\NFP$ from \cite{troeleke1}.} to Kohlenbach's higher-order hierarchy based on \emph{discontinuity} from \cite{kohlenbach2}.  In this sense, our hierarchy constitutes a `return to Brouwer' and is `orthogonal' to the usual comprehension hierarchy.  
  
\smallskip

For all the above reasons, we baptise the aforementioned higher-order hierarchy the \emph{Plato hierarchy}. 
We discuss our results in more detail in Section \ref{bootstraps}, while Section \ref{lolol} provides an introduction to Reverse Mathematics and the G\"odel hierarchy. 
\subsection{Hilbert, G\"odel, and classification}\label{lolol} 
During his invited lecture at the second \emph{International Congress of Mathematicians} of 1900 in Paris, David Hilbert presented his famous list of 23 open problems (\cite{hilbertlist}) that would have a profound influence on modern mathematics.  
Hilbert's list contains a number of \emph{foundational/logical} problems.  For instance, Problem 2 pertains to the \emph{consistency} of mathematics, i.e.\ the fact that no contradiction can be proved in mathematics.  
Hilbert later elaborated on Problem 2 by formulating \emph{Hilbert's program for the foundations of mathematics} (\cite{hilbertendlich}); 
this program calls for a proof of consistency of \emph{all} of mathematics {using only methods from so-called finitistic\footnote{The system $\textsf{PRA}$ in Figure \ref{xxy} is believed to capture Hilbert's finitistic mathematics (\cite{tait1}).} mathematics}. 

\smallskip

However, G\"odel's famous \emph{incompleteness theorems} (\cite{goe1}) are generally believed to show that Hilbert's program is \emph{impossible}: G\"odel namely showed that any logical system rich enough to express arithmetic, \emph{cannot} even prove its \emph{own} consistency, let alone that of all of mathematics. 
Moreover, one can build stronger and stronger logical systems by consecutively appending the formula expressing the system's consistency (or inconsistency).  
This proliferation of logical systems has not led to chaos, but to remarkable order and surprising regularity, as follows:  as a \emph{positive} outcome of G\"odel's \emph{negative} solution to Hilbert's program, the notion of consistency gave rise to the \emph{G\"odel hierarchy} presented in Figure \ref{xxy}: a collection of logical systems \emph{linearly} ordered via increasing consistency strength.   

\smallskip

As to its import, the G\"odel hierarchy is claimed to capture \emph{all} systems that are natural or foundationally important .   
For instance, Simpson claims the following regarding the G\"odel hierarchy and the consistency strength ordering `$<$': 
\begin{quote}
{It is striking that a great many foundational theories are linearly ordered by $<$. Of course it is possible to construct pairs of artificial theories which are incomparable under $<$. 
However, this is not the case for the ``natural'' or non-artificial theories which are usually regarded as significant in the foundations of mathematics.} (\cite{sigohi})
\end{quote}
Burgess and Koellner corroborate Simpson's claims in \cite{dontfixwhatistoobroken}*{\S1.5} and \cite{peterpeter}*{\S1.1}; the former refers to the G\"odel hierarchy as the \emph{Fundamental Series}. 
Precursors to the G\"odel hierarchy may be found in the work of Wang (\cite{wangjoke}) and Bernays (\cite{theotherguy, puben}).  
Friedman (\cite{friedber}) has studied the linear nature of the G\"odel hierarchy in great detail, including many more systems than present in Figure \ref{xxy}.
The importance of the logical systems present in Figure \ref{xxy} is discussed below the latter.    
 \begin{figure}[h]
\[
\begin{array}{lll}
&\textup{\textbf{strong}} \hspace{1.5cm}& 
\left\{\begin{array}{l}

\textup{large cardinals}\\
\vdots\\
\ZFC \\
\textsf{\textup{ZC}} \\
\textup{simple type theory}
\end{array}\right.
\\
&&\\
&\textup{\textbf{medium}} & 
\left\{\begin{array}{l}
 {\Z}_{2}\equiv \cup_{k}\SIXk\\
\vdots\\
\textup{$\Pi_{2}^{1}\textsf{-CA}_{0}$}\\
\textup{$\FIVE$ }\\
\textup{$\ATR_{0}$}  \\
\textup{$\ACA_{0}$} \\
\end{array}\right.
\\
&
\\
&\begin{array}{c}\\\textup{\textbf{weak}}\\ \end{array}& 
\left\{\begin{array}{l}
\WKL_{0} \\
\textup{$\RCA_{0}$} \\
\textup{$\textsf{PRA}$} \\
\textup{bounded arithmetic} \\
\end{array}\right.
\\
\end{array}
\]
\caption{The G\"odel hierarchy (taken from \cite{sigohi}*{p.\ 111})}\label{xxy}
\begin{picture}(0,0)
\put(30,103){$\left.\begin{array}{cc} 
~\\
~\\
~\\
~\\
~\\
~\\
\end{array}\right.$}
\end{picture}
\end{figure}
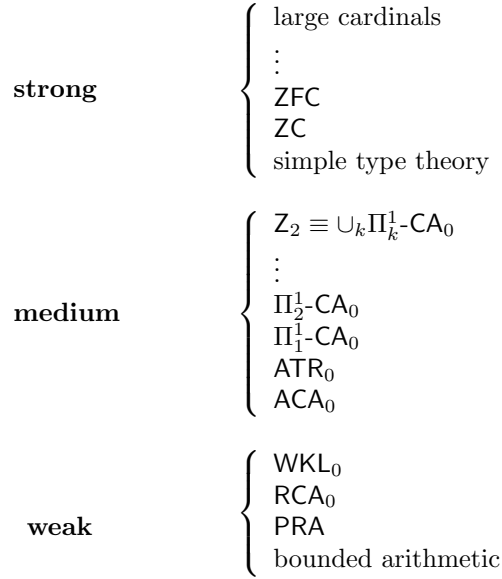~\\
We now discuss the systems in Figure \ref{xxy} and their role in mathematics and computer science.  In this light, the G\"odel hierarchy becomes a central object of study in logic to which all sub-fields contribute.
\renewcommand{\theenumi}{\roman{enumi}}
\begin{enumerate}
\item \emph{Bounded arithmetic} provides a logical framework for the study of polynomial time computation, and hence the `P versus NP' problem (\cite{buss}*{I, II}).  
\item The system $\RCA_{0}$ is the `base theory' of \emph{Reverse Mathematics} (RM hereafter; see Section \ref{prelim1}) and formalises `computable mathematics'.\label{RMKES}
\item The system $\WKL_{0}$ provides a \emph{partial} realisation of Hilbert's program (\cite{sigohi, simpsonpart}).  
The `finitistic' mathematics as in this program, is shown by Tait to be captured by the system $\textsf{PRA}$ (\cite{tait1}).
\item The system $\ATR_{0}$ is the upper limit of \emph{predicative mathematics} (\cite{sigohi,simpsonfreid}). 
\item The system $\Z_{2}$, called \emph{second-order arithmetic}, originates from the logical system $H$ used by Hilbert-Bernays in \emph{Grundlagen der Mathematik} (\cite{hillebilly,hillebilly2}).  
\item The system $\ZFC$ is \emph{Zermelo-Fraenkel set theory with the axiom of choice}, i.e.\ the standard/typical foundations of mathematics (\cite{jechh}). 
\item \emph{Large cardinal axioms} express regularities of the universe of sets and settle the truth of (certain) theorems independent of $\ZFC$ (\cite{jechh}).  
\end{enumerate}
We refer to \cites{simpson2, simpson1} for an overview of RM, and to \cite{stillebron} for an introduction.  
A brief introduction to Kohlenbach's \emph{higher-order} RM may be found in Section \ref{prelim1}.

\smallskip

Finally, the G\"odel hierarchy exhibits some remarkable \emph{robustness}: we can perform the following modifications and the hierarchy remains largely unchanged.
\begin{enumerate}
 \renewcommand{\theenumi}{\Roman{enumi}}
\item Instead of the consistency strength ordering, we can order via inclusion: Simpson claims that inclusion and consistency strength yield the same\footnote{Simpson mentions in \cite{sigohi} the caveat that e.g.\ $\PRA$ and $\WKL_{0}$ have the same first-order strength, but the latter is strictly stronger than the former.} G\"odel hierarchy as depicted in \cite{sigohi}*{Table 1} and Figure \ref{xxy}.  Some exceptional statements do fall outside of the inclusion-based G\"odel hierarchy.\label{gohi}  
\item We can replace systems with their higher-order counterparts (see e.g.\ \cite{kohlenbach2}) boasting a much richer language.  These higher-order systems generally prove the same sentences as their second-order counterpart.
\end{enumerate}
As suggested by item \eqref{gohi}, there are some examples of theorems that fall outside of the G\"odel hierarchy \emph{based on inclusion}, like \emph{special cases} of Ramsey's theorem and the axiom of determinacy from set theory (\cites{dsliceke, shoma}).
The latter axiom restricted to certain formula classes even yields a parallel hierarchy for the medium range of the G\"odel hierarchy based on inclusion.  
By the results in \cite{dagsamIII,dagsamIIIb, dagsamV, dagsamVI, dagsamVII}, basic compactness properties like the \emph{Heine-Borel theorem} for uncountable covers or \emph{Pincherle's theorem}, yield 
such parallel hierarchies \emph{in higher-order arithmetic}.  

\subsection{Plato, G\"odel, and their hierarchies}\label{bootstraps}
\subsubsection{Introduction}\label{pgintro}
We provide an overview of the results to be obtained in this paper, including the Plato hierarchy. 
The following figure provides a neat summary, while definitions may be found in Sections \ref{bootstraps1}, \ref{prelim2}, and \ref{bookstrap}.  
In this paper, we establish the hierarchy on the right-hand side of Figure \ref{kk} and associated results. 
\begin{figure}[H]
\includegraphics[width=1.05 \textwidth]{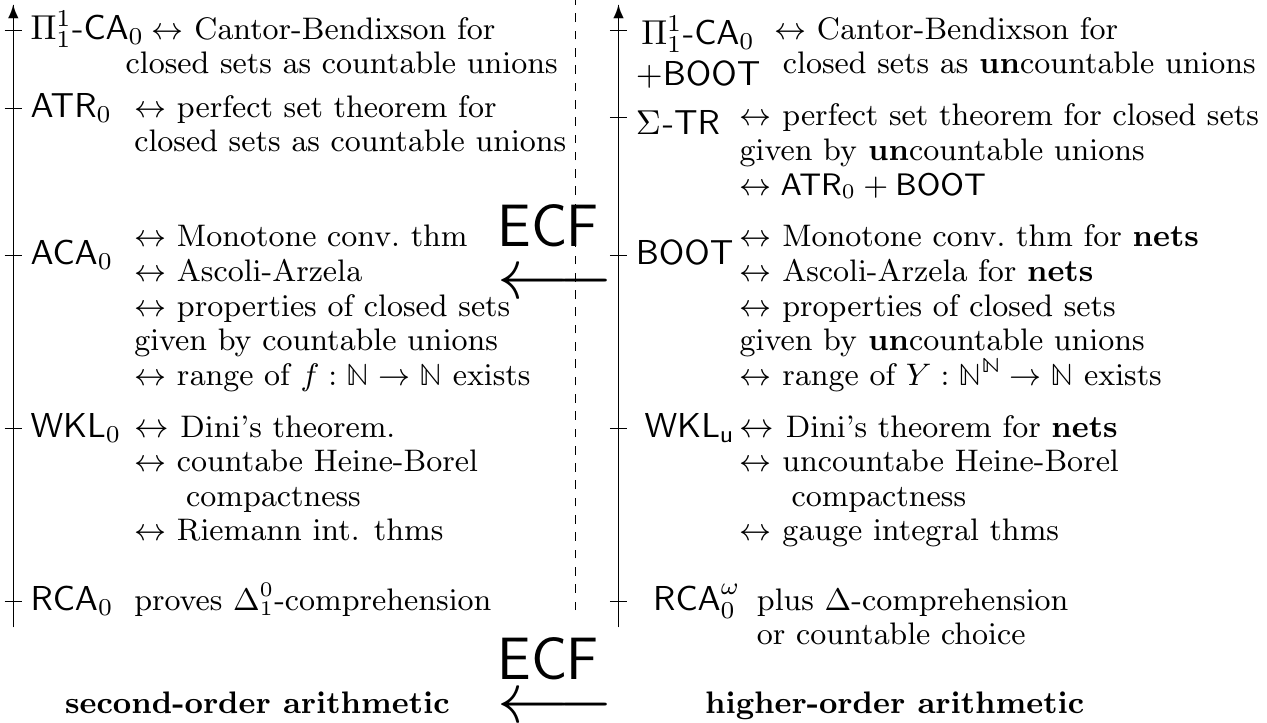}
\caption{The connection between the Plato and G\"odel hierarchies: $\ECF$ converts the right to the left hierarchy.}
\label{kk}
\end{figure}
\vspace{-0.7cm}
The systems at the same height in Figure \ref{kk} have the same first-order strength as the $\ECF$ translation converts the right-hand side into the left-hand side, taking into account the caveat in Remark \ref{unbeliever} regarding $\ECF$ (see Remark \ref{ECF} for the latter).    
In light of Figure \ref{kk}, it is no exaggeration to claim that the Big Five and the associated RM arise as special cases of higher-order 
RM via the lossy $\ECF$ translation. For this reason, the hierarchy formed by $\BOOT$ and its ilk is called the \emph{Plato hierarchy}, inspired by Plato's famous writings on ideal objects and their role in foundations of mathematics, the \emph{allegory of the cave} in particular.  

\smallskip

We note that the RM of the \emph{gauge integral} was studied in detail in \cite{dagsamIII}*{\S3}.  We briefly discuss this integral, and the associated RM-results, in Remark~\ref{woke}.
We note that nets and the gauge integral are well-known generalisations of sequences and the Riemann integral (see Section \ref{intronet} and Remark \ref{woke}).  
We also note that $\ECF$ translates the existence of discontinuous functions to `$0=1$'; since Kohlenbach's higher-order hierarchy (\cite{kohlenbach2}) makes essential use of discontinuous functions, the Plato hierarchy is seen to be markedly different.  In particular, as discussed below, the Plato hierarchy constitutes a `return to Brouwer' in a precise sense. 

\smallskip

Moreover, the axioms in the Plato hierarchy are \emph{explosive} in that combining them with comprehension axioms from Kohlenbach's hierarchy yields axioms \emph{much} stronger than the individual components.  The following remark is indispensable.  
  
\begin{rem}[The nature of $\ECF$]\label{unbeliever}\rm
We discuss the meaning of the words `$A$ is converted into $B$ by the $\ECF$-translation'.  Such statement is obviously not to be taken literally, as e.g.\ $[\BOOT]_{\ECF}$ is not verbatim $\ACA_{0}$.  
Nonetheless, $[\BOOT]_{\ECF}$ follows from $\ACA_{0}$ by noting that quantifiers over $\N^{\N}$ may be replaced by quantifiers over $\N$ in case all functionals on $\N^{\N}$ are continuous (see Theorem \ref{boef}).  Similarly, $[\HBU]_{\ECF}$ is not verbatim the Heine-Borel theorem for countable covers, but the latter does imply the former by noting that for uncountable covers represented by continuous functions, there is a trivial \emph{countable} sub-cover enumerated by $\Q$.  

\smallskip

In general, that (continuous) objects have countable representations is the very foundation of the formalisation of mathematics in $\L_{2}$, and identifying (continuous) objects and their countable representations is routinely done.  
Thus, when we say `$A$ is converted into $B$ by the $\ECF$-translation', we mean that $[A]_{\ECF}$ is about a class of continuous objects to which $B$ is immediately seen to apply, with a possible intermediate step involving representations.  
Since this kind of step forms the bedrock of (second-order) RM, it would therefore appear harmless in this context. 
\end{rem}
Taking into account the previous remark, the literature already boasts some results similar to the ones in Figure \ref{kk}.  
For instance, the RM of the Vitali covering theorem for uncountable covers, called $\WHBU$, is studied in \cite{dagsamVI}*{\S3}.  
Now, $\WHBU$ has the first-order strength of $\WWKL$ (see \cite{simpson2}*{X.1}) and the associated equivalences in measure theory fit between $\HBU/\WKL$ and $\RCAo/\RCA_{0}$ in Figure \ref{kk}.  

\smallskip

Next, it was noted above Remark \ref{unbeliever} that Kohlenbach's hierarchy from \cite{kohlenbach2} is based on \emph{discontinuity}, while the Plato hierarchy is markedly different.    
Indeed, $\BOOT$ and related principles map to quite fundamental axioms under $\ECF$, i.e.\ the replacement of higher-order objects by continuous-by-definition RM-codes.  In light of the previous, one
might expect that $\BOOT$ and related principles are somehow connected to continuity.  We shall establish that these axioms are indeed equivalent to fragments of a \emph{classically valid} continuity axiom from Brouwer's intuitionistic analysis, called \emph{neighbourhood function principle} ($\NFP$).  In particular, while higher-order comprehension does not capture the Plato hierarchy, fragments of $\NFP$ \emph{can} capture the latter.  
Thus, the Plato hierarchy is a `return to Brouwer' in the sense that we avoid discontinuous functions and work with (classically valid) axioms from intuitionistic mathematics.

\smallskip

Once the results of Figure \ref{kk} have sunk in, an obvious questions is: \emph{What is the Plato hierarchy a reflection of? What is the nature of this reflection?}
We provide a partial answer in this paper by generalising the equivalence between $\BOOT$ and the monotone convergence theorem for nets indexed by Baire space to larger index sets. 
We also provide a `translation' that reduces the new equivalence to the old one.  Thus, a more apt name perhaps would have been the \emph{Plato universe}.      

\smallskip

Finally, while $\ECF$ is clearly a `lossy' translation, results {can} also be `lifted' in the other direction in Figure \ref{kk}, i.e.\ from second-order to higher-order arithmetic: the proof of Theorem \ref{proofofconcept} establishes that the monotone convergence theorem for nets in the unit interval implies $\BOOT$ using so-called $\Delta$-comprehension.  
This proof is an almost verbatim copy of the associated second-order proof in \cite{simpson2}*{p.\ 107}, i.e.\ there is also a connection at the level of proofs. 
This is not an isolated case: many so-called recursive counterexamples give rise to reversals in RM, and these results can be lifted 
to obtain higher-order results in many cases, as studied in detail in \cite{samrecount, samFLO2}.  We caution the reader that these `lifted' proofs are not optimal, in that they generally
do not go through in the weakest possible base theory.  

\subsubsection{The inhabitants of the Plato hierarchy}\label{bootstraps1}
We discuss in detail the concepts and axioms involved with (part of) the Plato hierarchy as depicted in Figure \ref{kk}.
We shall introduce the notion of \emph{net} and the \emph{bootstrap axiom} $\BOOT$, starting from the former's historical roots.   
We also discuss our `uncountable' concept of open set to be used in the Plato hierarchy.

\smallskip

Abstraction is an integral part of mathematics, from Euclid's \emph{Elements} to the present day.  
In this spirit, E.\ H.\ Moore presented a framework called \emph{General Analysis} at the 1908 ICM in Rome (\cite{mooreICM}) that was to be a `unifying abstract theory' for various parts of analysis.  
Indeed, Moore's framework captures various limit notions in one abstract concept (\cites{moorelimit1}) and even includes a generalisation of the concept of \emph{sequence} to possibly uncountable index sets (called \emph{directed sets}), nowadays called \emph{nets} or \emph{Moore-Smith sequences}.  
These were first described in \cite{moorelimit2} and then formally introduced by Moore and Smith in \cites{moorsmidje}. 
They also established the generalisation from sequences to nets of various basic theorems due to Bolzano-Weierstrass, Dini, and Arzel\`a (\cite{moorsmidje}*{\S8-9}).
More recently, nets are central to the development of \emph{domain theory} (see \cites{gieren, gieren2,degou}), including a definition of the Scott and Lawson topologies in terms of nets.    
Moreover, sequences cannot be used in this context, as expressed in a number of places:
\begin{quote}
[\dots] clinging to ascending sequences would produce a mathematical theory that becomes rather bizarre, whence our move to directed families. (\cite{degou}*{p.\ 59})
\end{quote}
\begin{quote}
Turning to foundations, we feel that the necessity to choose chains where directed subsets are naturally available (such as in function spaces) and thus to rely on the Axiom of Choice without need, is a serious stain on this approach. (\cite{aju}*{\S2.2.4}).
\end{quote}
Thus, nets enjoy a rich history, as well as a mainstream (and essential) status in mathematics and computer science.  
Motivated by the above,  the study of nets in RM was undertaken in \cites{samcie19, samwollic19, samnetspilot}.  
We continue the RM study of nets in this paper, and the truly novel results in this paper are as follows. 
\begin{enumerate}
\item[(i)] basic convergence theorems for nets `bootstrap' themselves (or: explode) to higher levels of the hierarchy when combined with Kohlenbach's comprehension axioms from the medium range.   
\item[(ii)] basic convergence theorems for nets are equivalent to the following comprehension axiom $\BOOT$, plus potentially countable choice.  
\end{enumerate}
The axiom $\BOOT$ is defined as follows, and discussed in detail in Section \ref{bookstrap}. 
\bdefi[$\BOOT$]
$(\forall Y^{2})(\exists X^{1})(\forall n^{0})\big[ n\in X \asa (\exists f^{1})(Y(f, n)=0)    \big]. $
\edefi
Now, since uncountable index sets are first-class citizens in the theory of nets, we shall work in Kohlenbach's \emph{higher-order} RM (see Section \ref{prelim1}).  The exact formalisation of nets in higher-order RM is detailed in Definition \ref{strijker} and Section~\ref{intronet}.  
In Sections~{\ref{cauf} to \ref{powpow}}, we restrict ourselves to nets indexed by subsets of Baire space, i.e.\ part of third-order arithmetic, as such nets are already general enough to obtain our main results in Figure \ref{kk}.  
Our results for the \emph{monotone convergence theorem} $\MCT_{\net}^{C}$ for nets in Cantor space indexed by subsets of Baire space, are neatly summarised by Figure \ref{xxz}; the associated logical systems are defined in Section~\ref{prelim2}.
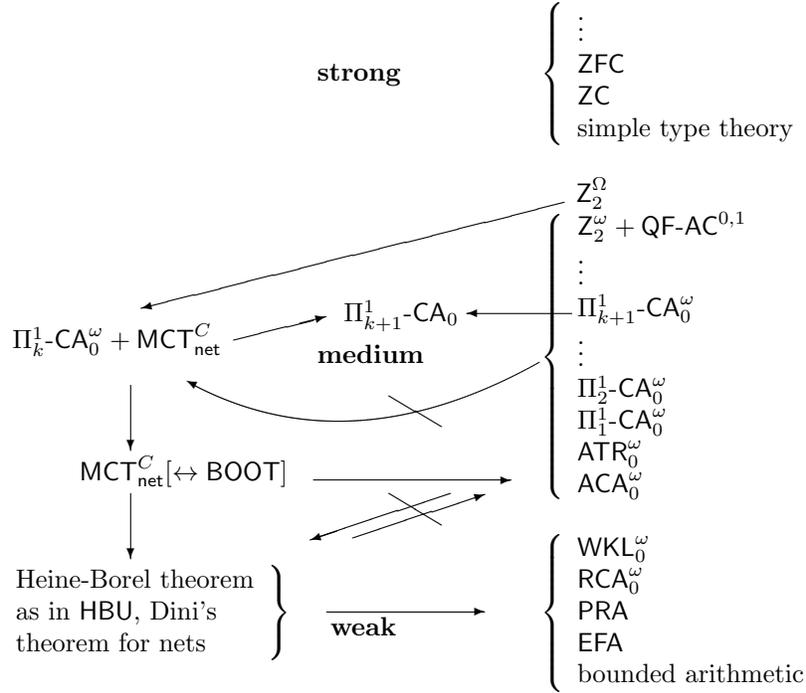
\begin{figure}[H]
\[
\begin{array}{lll}
&\textup{\textbf{strong}} \hspace{1.5cm}& 
\left\{\begin{array}{l}
\vdots\\
\ZFC \\
\textsf{\textup{ZC}} \\
\textup{simple type theory}
\end{array}\right.
\\
&&\\
  &&\quad{ {~\Z_{2}^{\Omega}}} \\
{ {\begin{array}{l}
\textup{$\SIXK +\MCT_{\net}^{C}$}\\
~\\
\end{array}}}
&\textup{\textbf{medium}} & 
\left\{\begin{array}{l}
 {\Z}_{2}^{\omega}+ \QFAC^{0,1}\\
\vdots\\
\SIXko^{\omega}\\
\vdots\\
\textup{$\Pi_{2}^{1}\textsf{-CA}_{0}^{ {\omega}}$}\\
\textup{$\FIVE^{ {\omega}}$ }\\
\textup{$\ATR_{0}^{ {\omega}}$}  \\
\textup{$\ACA_{0}^{ {\omega}}$} \\
\end{array}\right.
\\
~\\
{ {\left.\begin{array}{l}
\textup{Heine-Borel theorem}\\
\textup{as in $\HBU$, Dini's}\\
\textup{theorem for nets}\\
\end{array}\right\}}}
&\begin{array}{c}\\\textup{\textbf{weak}}\\ \end{array}& 
\left\{\begin{array}{l}
\WKL_{0}^{ {\omega}} \\
\textup{$\RCA_{0}^{ {\omega}}$} \\
\textup{$\textsf{PRA}$} \\
\textup{$\textsf{EFA}$ } \\
\textup{bounded arithmetic} \\
\end{array}\right.
\\
\end{array}
\]
\caption{The G\"odel hierarchy (based on inclusion and higher types) with a parallel branch for the medium range}\label{xxz}
\begin{picture}(250,0)
\put(16,166){ {\vector(0,-1){25}}}
\put(0, 130){$ \MCT_{\net}^{C}[\asa \BOOT]$}
\put(85,130){ {\vector(1,0){75}}}
\put(16,125){ {\vector(0,-1){25}}}
\put(100, 190){$\SIXko$}
\put(180,235){ {\vector(-4,-1){160}}}

\put(90,80){ {\vector(1,0){60}}}

\put(55,183){ {\vector(4,1){35}}}

\put(183,193){ {\vector(-1,0){40}}}

\put(43, 166){\setlength{\unitlength}{1cm}\qbezier(0,0)(2.1,-1.1)(4.6,0.3)}
\put(40,166){ {\vector(-2,1){3}}}
\put(137,150){{\line(-5,3){20}}}

\put(137,125){ {\vector(-3,-1){53}}}
\put(100,108){ {\vector(3,1){50}}}
\put(137,113){{\line(-5,3){20}}}
\end{picture}
\end{figure}~\\
Of course, Figure \ref{xxz} only provides one example and we shall obtain a number of such parallel hierarchies in Section \ref{karmichael}, based on the following theorems.
\begin{enumerate}
\item The Bolzano-Weierstrass theorem for nets (Section \ref{kawl}).
\item The existence of moduli of convergence for nets (Section \ref{cauf}). 
\item The Moore-Osgood theorem for nets (Section \ref{motkrijgen}).
\item Numerous variations including the anti-Specker property and the Arzel\`a and Ascoli-Arzel\`a theorems (Section \ref{kawl}) and Cauchy nets (Section \ref{cauf}).
\end{enumerate}
We refer to the hierarchy formed by $\SIXK+\MCT_{\net}^{C}$ for $k\geq 0$ as the \emph{bootstrap} hierarchy as the logical strength (at least $\SIXko$) is `bootstrapped' from two essential parts, 
namely $\SIXK$ and $\MCT_{\net}^{C}$ that are weak(er) in isolation.  
To obtain the aforementioned results, $\MCT_{\net}^{C}$ is shown to be equivalent to a new comprehension principle $\BOOT$, and similar results for the other convergence theorems.    

\smallskip

Next, we also study two `more general' convergence theorems, respectively for nets in function spaces and for nets involving index sets beyond Baire space.  
The former theorem `bootstraps itself', i.e.\ become stronger and stronger without the need for additional comprehension, as discussed in Section \ref{powpow}.  
The latter theorem carries us beyond second-order arithmetic, and shows that our proofs readily generalise to higher types.
Nonetheless, results associated to index sets beyond Baire space are still mapped into the lower regions of second-order arithmetic by $\ECF$, as discussed in Section \ref{moar}.
The results in the latter also provide an equivalence $\BOOT^{1}\asa \MCT_{\net}^{1}$ between two fourth-order principles; we define a lossy translation (but less lossy than $\ECF$) that converts this equivalence into $\BOOT\asa \MCT_{\net}^{C}$, i.e.\ an equivalence in third-order arithmetic.  This (partially) answers a question from the previous section, namely what the Plato hierarchy is a reflection of.  

\smallskip

After some contemplation, one observes that $\BOOT$ and $\HBU$ cannot be captured (well or at all) in terms of  the known comprehension axioms from \cites{kohlenbach2, simpson2} by Figures~\ref{kk} and~\ref{xxz}.  
However, the main question of RM dictates that we find a suitable class of set existence axioms that capture $\BOOT$ and $\HBU$.  

\smallskip

To this end, we show in Section~\ref{main3} that axioms from the Plato hierarchy are equivalent to fragments of a continuity axiom from intuitionistic analysis called \emph{special bar/Brouwer continuity} $\textsf{SBC}$ in \cite{KT} and \emph{neighbourhood function principle $\NFP$} in \cite{troeleke1}. 
Moreover, discontinuous functions are converted to `$0=1$' by $\ECF$, while the Plato hierarchy does have 
rather meaningful translations under $\ECF$.  
In this light, Kohlenbach's hierarchy from \cite{kohlenbach2} is based on \emph{discontinuity} and the Plato hierarchy `by contrast'
has a natural formulation in terms of \emph{continuity}.

\smallskip

Finally, a number of theorems in Figure \ref{kk} mentions open (and closed) sets.  
Open sets are represented in RM by countable unions of open balls and it is a natural question what the correct notion of open set in the Plato hierarchy is.  As studied in Section \ref{main2}, 
\emph{uncountable} unions of open balls are the correct notion (in contrast to open sets represented by characteristic functions as in \cites{dagsamVII, samnetspilot}), giving rise to nice equivalences and the original RM-equivalences under $\ECF$.

\smallskip

We shall study the Cantor-Bendixson theorem, the perfect set theorem, and located sets.  
We wish to point out that finding the aforementioned correct notion of open set is by no means obvious: we have previously studied (higher-order) open sets represented by characteristic functions in \cite{dagsamVII, samnetspilot}.  
Interesting results were definitely
obtained (see Remark \ref{nelta}), but the concept of open set from \cite{dagsamVII,samnetspilot} does not seem to yield nice RM-equivalences \emph{try as one might}.

\smallskip

It goes without saying that this paper constitutes a spin-off from our joint project with Dag Normann on the logical and computational properties of the uncountable. 
The interested reader may consult \cite{dagsamIII} as an introduction.  

\section{Preliminaries}\label{prelim}

We introduce \emph{Reverse Mathematics} in Section \ref{prelim1}, as well as its generalisation to \emph{higher-order arithmetic}, and the associated base theory $\RCAo$.  
We introduce some essential axioms in Section~\ref{prelim2}.  
We provide a brief introduction to nets and related concepts in Section \ref{intronet}. 
As noted in Section \ref{lolol}, we mostly study nets \emph{indexed by subsets of Baire space}, i.e.\ part of third-order arithmetic; the associated bit of set theory shall be represented in $\RCAo$ as in Definition \ref{strijker}.  

\subsection{Reverse Mathematics}\label{prelim1}
Reverse Mathematics is a program in the foundations of mathematics initiated around 1975 by Friedman (\cites{fried,fried2}) and developed extensively by Simpson (\cite{simpson2}).  
The aim of RM is to identify the minimal axioms needed to prove theorems of ordinary, i.e.\ non-set theoretical, mathematics. 

\smallskip

We refer to \cite{stillebron} for a basic introduction to RM and to \cite{simpson2, simpson1} for an overview of RM.  We expect basic familiarity with RM, but do sketch some aspects of Kohlenbach's \emph{higher-order} RM (\cite{kohlenbach2}) essential to this paper, including the base theory $\RCAo$ (Definition \ref{kase}).  
As will become clear, the latter is officially a type theory but can accommodate (enough) set theory via Definition \ref{strijker}. 

\smallskip

First of all, in contrast to `classical' RM based on \emph{second-order arithmetic} $\Z_{2}$, higher-order RM uses $\L_{\omega}$, the richer language of \emph{higher-order arithmetic}.  
Indeed, while the latter is restricted to natural numbers and sets of natural numbers, higher-order arithmetic can accommodate sets of sets of natural numbers, sets of sets of sets of natural numbers, et cetera.  
To formalise this idea, we introduce the collection of \emph{all finite types} $\mathbf{T}$, defined by the two clauses:
\begin{center}
(i) $0\in \mathbf{T}$   and   (ii)  If $\sigma, \tau\in \mathbf{T}$ then $( \sigma \di \tau) \in \mathbf{T}$,
\end{center}
where $0$ is the type of natural numbers, and $\sigma\di \tau$ is the type of mappings from objects of type $\sigma$ to objects of type $\tau$.
In this way, $1\equiv 0\di 0$ is the type of functions from numbers to numbers, and where  $n+1\equiv n\di 0$.  Viewing sets as given by characteristic functions, we note that $\Z_{2}$ only includes objects of type $0$ and $1$.    

\smallskip

Secondly, the language $\L_{\omega}$ includes variables $x^{\rho}, y^{\rho}, z^{\rho},\dots$ of any finite type $\rho\in \mathbf{T}$.  Types may be omitted when they can be inferred from context.  
The constants of $\L_{\omega}$ includes the type $0$ objects $0, 1$ and $ <_{0}, +_{0}, \times_{0},=_{0}$  which are intended to have their usual meaning as operations on $\N$.
Equality at higher types is defined in terms of `$=_{0}$' as follows: for any objects $x^{\tau}, y^{\tau}$, we have
\be\label{aparth}
[x=_{\tau}y] \equiv (\forall z_{1}^{\tau_{1}}\dots z_{k}^{\tau_{k}})[xz_{1}\dots z_{k}=_{0}yz_{1}\dots z_{k}],
\ee
if the type $\tau$ is composed as $\tau\equiv(\tau_{1}\di \dots\di \tau_{k}\di 0)$.  
Furthermore, $\L_{\omega}$ also includes the \emph{recursor constant} $\mathbf{R}_{\sigma}$ for any $\sigma\in \mathbf{T}$, which allows for iteration on type $\sigma$-objects as in the special case \eqref{special}.  Formulas and terms are defined as usual.  
One obtains the sub-language $\L_{n+2}$ by restricting the above type formation rule to produce only type $n+1$ objects (and related types of similar complexity).        
\bdefi\label{kase} 
The base theory $\RCAo$ consists of the following axioms.
\begin{enumerate}
 \renewcommand{\theenumi}{\alph{enumi}}
\item  Basic axioms expressing that $0, 1, <_{0}, +_{0}, \times_{0}$ form an ordered semi-ring with equality $=_{0}$.
\item Basic axioms defining the well-known $\Pi$ and $\Sigma$ combinators (aka $K$ and $S$ in \cite{avi2}), which allow for the definition of \emph{$\lambda$-abstraction}. 
\item The defining axiom of the recursor constant $\mathbf{R}_{0}$: For $m^{0}$ and $f^{1}$: 
\be\label{special}
\mathbf{R}_{0}(f, m, 0):= m \textup{ and } \mathbf{R}_{0}(f, m, n+1):= f(n, \mathbf{R}_{0}(f, m, n)).
\ee
\item The \emph{axiom of extensionality}: for all $\rho, \tau\in \mathbf{T}$, we have:
\be\label{EXT}\tag{$\textsf{\textup{E}}_{\rho, \tau}$}  
(\forall  x^{\rho},y^{\rho}, \varphi^{\rho\di \tau}) \big[x=_{\rho} y \di \varphi(x)=_{\tau}\varphi(y)   \big].
\ee 
\item The induction axiom for quantifier-free\footnote{To be absolutely clear, variables (of any finite type) are allowed in quantifier-free formulas of the language $\L_{\omega}$: only quantifiers are banned.} formulas of $\L_{\omega}$.
\item $\QFAC^{1,0}$: The quantifier-free Axiom of Choice as in Definition \ref{QFAC}.
\end{enumerate}
\edefi
\bdefi\label{QFAC} The axiom $\QFAC$ consists of the following for all $\sigma, \tau \in \textbf{T}$:
\be\tag{$\QFAC^{\sigma,\tau}$}
(\forall x^{\sigma})(\exists y^{\tau})A(x, y)\di (\exists Y^{\sigma\di \tau})(\forall x^{\sigma})A(x, Y(x)),
\ee
for any quantifier-free formula $A$ in the language of $\L_{\omega}$.
\edefi
We let $\IND$ be the induction axiom for all formulas in $\L_{\omega}$.  The system $\RCAo+\IND$ has the same first-order strength as Peano arithmetic.  

\smallskip

As discussed in \cite{kohlenbach2}*{\S2}, $\RCAo$ and $\RCA_{0}$ prove the same sentences `up to language' as the latter is set-based and the former function-based.  Recursion as in \eqref{special} is called \emph{primitive recursion}; the class of functionals obtained from $\mathbf{R}_{\rho}$ for all $\rho \in \mathbf{T}$ is called \emph{G\"odel's system $T$} of all (higher-order) primitive recursive functionals.  

\smallskip

We use the usual notations for natural, rational, and real numbers, and the associated functions, as introduced in \cite{kohlenbach2}*{p.\ 288-289}.  
\begin{defi}[Real numbers and related notions in $\RCAo$]\label{keepintireal}\rm~
\begin{enumerate}
 \renewcommand{\theenumi}{\alph{enumi}}
\item Natural numbers correspond to type zero objects, and we use `$n^{0}$' and `$n\in \N$' interchangeably.  Rational numbers are defined as signed quotients of natural numbers, and `$q\in \Q$' and `$<_{\Q}$' have their usual meaning.    
\item Real numbers are coded by fast-converging Cauchy sequences $q_{(\cdot)}:\N\di \Q$, i.e.\  such that $(\forall n^{0}, i^{0})(|q_{n}-q_{n+i}|<_{\Q} \frac{1}{2^{n}})$.  
We use Kohlenbach's `hat function' from \cite{kohlenbach2}*{p.\ 289} to guarantee that every $q^{1}$ defines a real number.  
\item We write `$x\in \R$' to express that $x^{1}:=(q^{1}_{(\cdot)})$ represents a real as in the previous item and write $[x](k):=q_{k}$ for the $k$-th approximation of $x$.    
\item Two reals $x, y$ represented by $q_{(\cdot)}$ and $r_{(\cdot)}$ are \emph{equal}, denoted $x=_{\R}y$, if $(\forall n^{0})(|q_{n}-r_{n}|\leq {2^{-n+1}})$. Inequality `$<_{\R}$' is defined similarly.  
We sometimes omit the subscript `$\R$' if it is clear from context.           
\item Functions $F:\R\di \R$ are represented by $\Phi^{1\di 1}$ mapping equal reals to equal reals, i.e. $(\forall x , y\in \R)(x=_{\R}y\di \Phi(x)=_{\R}\Phi(y))$.\label{EXTEN}
\item The relation `$x\leq_{\tau}y$' is defined as in \eqref{aparth} but with `$\leq_{0}$' instead of `$=_{0}$'.  Binary sequences are denoted `$f^{1}, g^{1}\leq_{1}1$', but also `$f,g\in C$' or `$f, g\in 2^{\N}$'.  Elements of Baire space are given by $f^{1}, g^{1}$, but also denoted `$f, g\in \N^{\N}$'.
\item For a binary sequence $f^{1}$, the associated real in $[0,1]$ is $\r(f):=\sum_{n=0}^{\infty}\frac{f(n)}{2^{n+1}}$.\label{detrippe}
\item Sets of type $\rho$ objects $X^{\rho\di 0}, Y^{\rho\di 0}, \dots$ are given by their characteristic functions $F^{\rho\di 0}_{X}\leq_{\rho\di 0}1$, i.e.\ we write `$x\in X$' for $ F_{X}(x)=_{0}1$. \label{koer} 
\end{enumerate}
\end{defi}
\noindent
The following special case of item \eqref{koer} is singled out, as it will be used frequently.
\bdefi[$\RCAo$]\label{strijker}
A `subset $D$ of $\N^{\N}$' is given by its characteristic function $F_{D}^{2}\leq_{2}1$, i.e.\ we write `$f\in D$' for $ F_{D}(f)=1$ for any $f\in \N^{\N}$.
A `binary relation $\preceq$ on a subset $D$ of $\N^{\N}$' is given by the associated characteristic function $G_{\preceq}^{(1\times 1)\di 0}$, i.e.\ we write `$f\preceq g$' for $G_{\preceq}(f, g)=1$ and any $f, g\in D$.
Assuming extensionality on the reals as in item \eqref{EXTEN}, we obtain characteristic functions that represent subsets of $\R$ and relations thereon.  
Using pairing functions, it is clear we can also represent sets of finite sequences (of reals), and relations thereon.  
\edefi
Next, we mention the highly useful $\ECF$-interpretation. 
\begin{rem}[The $\ECF$-interpretation]\label{ECF}\rm
The (rather) technical definition of $\ECF$ may be found in \cite{troelstra1}*{p.\ 138, \S2.6}.
Intuitively, the $\ECF$-interpretation $[A]_{\ECF}$ of a formula $A\in \L_{\omega}$ is just $A$ with all variables 
of type two and higher replaced by countable representations of continuous functionals.  Such representations are also (equivalently) called `associates' or `RM-codes' (see \cite{kohlenbach4}*{\S4}). 
The $\ECF$-interpretation connects $\RCAo$ and $\RCA_{0}$ (see \cite{kohlenbach2}*{Prop.\ 3.1}) in that if $\RCAo$ proves $A$, then $\RCA_{0}$ proves $[A]_{\ECF}$, again `up to language', as $\RCA_{0}$ is 
formulated using sets, and $[A]_{\ECF}$ is formulated using types, namely only using type zero and one objects.  
\end{rem}
In light of the widespread use of codes in RM and the common practise of identifying codes with the objects being coded, it is no exaggeration to refer to $\ECF$ as the \emph{canonical} embedding of higher-order into second-order RM.  
For completeness, we also list the following notational convention for finite sequences.  
\begin{nota}[Finite sequences]\label{skim}\rm
We assume a dedicated type for `finite sequences of objects of type $\rho$', namely $\rho^{*}$.  Since the usual coding of pairs of numbers goes through in $\RCAo$, we shall not always distinguish between $0$ and $0^{*}$. 
Similarly, we do not always distinguish between `$s^{\rho}$' and `$\langle s^{\rho}\rangle$', where the former is `the object $s$ of type $\rho$', and the latter is `the sequence of type $\rho^{*}$ with only element $s^{\rho}$'.  The empty sequence for the type $\rho^{*}$ is denoted by `$\langle \rangle_{\rho}$', usually with the typing omitted.  

\smallskip

Furthermore, we denote by `$|s|=n$' the length of the finite sequence $s^{\rho^{*}}=\langle s_{0}^{\rho},s_{1}^{\rho},\dots,s_{n-1}^{\rho}\rangle$, where $|\langle\rangle|=0$, i.e.\ the empty sequence has length zero.
For sequences $s^{\rho^{*}}, t^{\rho^{*}}$, we denote by `$s*t$' the concatenation of $s$ and $t$, i.e.\ $(s*t)(i)=s(i)$ for $i<|s|$ and $(s*t)(j)=t(|s|-j)$ for $|s|\leq j< |s|+|t|$. For a sequence $s^{\rho^{*}}$, we define $\overline{s}N:=\langle s(0), s(1), \dots,  s(N-1)\rangle $ for $N^{0}<|s|$.  
For a sequence $\alpha^{0\di \rho}$, we also write $\overline{\alpha}N=\langle \alpha(0), \alpha(1),\dots, \alpha(N-1)\rangle$ for \emph{any} $N^{0}$.  By way of shorthand, 
$(\forall q^{\rho}\in Q^{\rho^{*}})A(q)$ abbreviates $(\forall i^{0}<|Q|)A(Q(i))$, which is (equivalent to) quantifier-free if $A$ is.   
\end{nota}

\subsection{Some axioms of higher-order RM}\label{prelim2}
We introduce some functionals which constitute the counterparts of second-order arithmetic $\Z_{2}$, and some of the Big Five systems, in higher-order RM.
We use the formulation from \cite{kohlenbach2, dagsamIII}.  

\smallskip
\noindent
First of all, $\ACA_{0}$ is readily derived from:
\begin{align}\label{mu}\tag{$\mu^{2}$}
(\exists \mu^{2})(\forall f^{1})\big[ (\exists n)(f(n)=0) \di [(f(\mu(f))=0)&\wedge (\forall i<\mu(f))f(i)\ne 0 ]\\
& \wedge [ (\forall n)(f(n)\ne0)\di   \mu(f)=0]    \big], \notag
\end{align}
and $\ACA_{0}^{\omega}\equiv\RCAo+(\mu^{2})$ proves the same sentences as $\ACA_{0}$ by \cite{hunterphd}*{Theorem~2.5}.   The (unique) functional $\mu^{2}$ in $(\mu^{2})$ is also called \emph{Feferman's $\mu$} (\cite{avi2}), 
and is clearly \emph{discontinuous} at $f=_{1}11\dots$; in fact, $(\mu^{2})$ is equivalent to the existence of $F:\R\di\R$ such that $F(x)=1$ if $x>_{\R}0$, and $0$ otherwise (\cite{kohlenbach2}*{\S3}), and to 
\be\label{muk}\tag{$\exists^{2}$}
(\exists \varphi^{2}\leq_{2}1)(\forall f^{1})\big[(\exists n)(f(n)=0) \asa \varphi(f)=0    \big]. 
\ee
\noindent
Secondly, $\FIVE$ is readily derived from the following sentence:
\be\tag{$\SS^{2}$}
(\exists\SS^{2}\leq_{2}1)(\forall f^{1})\big[  (\exists g^{1})(\forall n^{0})(f(\overline{g}n)=0)\asa \SS(f)=0  \big], 
\ee
and $\FIVE^{\omega}\equiv \RCAo+(\SS^{2})$ proves the same $\Pi_{3}^{1}$-sentences as $\FIVE$ by \cite{yamayamaharehare}*{Theorem 2.2}.   The (unique) functional $\SS^{2}$ in $(\SS^{2})$ is also called \emph{the Suslin functional} (\cite{kohlenbach2}).
By definition, the Suslin functional $\SS^{2}$ can decide whether a $\Sigma_{1}^{1}$-formula as in the left-hand side of $(\SS^{2})$ is true or false.   We similarly define the functional $\SS_{k}^{2}$ which decides the truth or falsity of $\Sigma_{k}^{1}$-formulas; we also define 
the system $\SIXK$ as $\RCAo+(\SS_{k}^{2})$, where  $(\SS_{k}^{2})$ expresses that $\SS_{k}^{2}$ exists.  Note that we allow formulas with \emph{function} parameters, but \textbf{not} \emph{functionals} here.
In fact, Gandy's \emph{Superjump} (\cite{supergandy}) constitutes a way of extending $\FIVE^{\omega}$ to parameters of type two.  We identify the functionals $\exists^{2}$ and $\SS_{0}^{2}$ and the systems $\ACAo$ and $\SIXK$ for $k=0$.

\smallskip

\noindent
Thirdly, full second-order arithmetic $\Z_{2}$ is readily derived from $\cup_{k}\SIXK$, or from:
\be\tag{$\exists^{3}$}
(\exists E^{3}\leq_{3}1)(\forall Y^{2})\big[  (\exists f^{1})Y(f)=0\asa E(Y)=0  \big], 
\ee
and we therefore define $\Z_{2}^{\Omega}\equiv \RCAo+(\exists^{3})$ and $\Z_{2}^\omega\equiv \cup_{k}\SIXK$, which are conservative over $\Z_{2}$ by \cite{hunterphd}*{Cor.\ 2.6}. 
Despite this close connection, $\Z_{2}^{\omega}$ and $\Z_{2}^{\Omega}$ can behave quite differently, as discussed in e.g.\ \cite{dagsamIII}*{\S2.2}.   The functional from $(\exists^{3})$ is also called `$\exists^{3}$', and we use the same convention for other functionals.  

\smallskip

Finally, the Heine-Borel theorem states the existence of a finite sub-cover for an open cover of certain spaces. 
Now, a functional $\Psi:\R\di \R^{+}$ gives rise to the \emph{canonical cover} $\cup_{x\in I} I_{x}^{\Psi}$ for $I\equiv [0,1]$, where $I_{x}^{\Psi}$ is the open interval $(x-\Psi(x), x+\Psi(x))$.  
Hence, the uncountable cover $\cup_{x\in I} I_{x}^{\Psi}$ has a finite sub-cover by the Heine-Borel theorem; in symbols:
\be\tag{$\HBU$}
(\forall \Psi:\R\di \R^{+})(\exists  y_{1}, \dots, y_{k}\in I){(\forall x\in I)}(\exists i\leq k)(x\in I_{y_{i}}^{\Psi}).
\ee
Note that $\HBU$ is almost verbatim \emph{Cousin's lemma} (see \cite{cousin1}*{p.\ 22}), i.e.\ the Heine-Borel theorem restricted to canonical covers.  
The latter restriction does not make much of a big difference, as studied in \cite{sahotop}.
By \cite{dagsamIII, dagsamV}, $\Z_{2}^{\Omega}$ proves $\HBU$ but $\Z_{2}^{\omega}+\QFAC^{0,1}$ cannot, 
and many basic properties of the \emph{gauge integral} (\cite{zwette, mullingitover}) are equivalent to $\HBU$.  
Although strictly speaking incorrect, we sometimes use set-theoretic notation, like reference to the cover $\cup_{x\in I} I_{x}^{\Psi}$ inside $\RCAo$,  to make proofs more understandable.  
Such reference can in principle be removed in favour of formulas of higher-order arithmetic.     

\subsection{An introduction to nets}\label{intronet}
We introduce the notion of net and associated concepts.  
We first consider the following standard definition from \cite{ooskelly}*{Ch.\ 2}.
\bdefi[Nets]\label{nets}
A set $D\ne \emptyset$ with a binary relation `$\preceq$' is \emph{directed} if
\begin{enumerate}
 \renewcommand{\theenumi}{\alph{enumi}}
\item The relation $\preceq$ is transitive, i.e.\ $(\forall x, y, z\in D)([x\preceq y\wedge y\preceq z] \di x\preceq z )$.
\item For $x, y \in D$, there is $z\in D$ such that $x\preceq z\wedge y\preceq z$.\label{bulk}
\item The relation $\preceq$ is reflexive, i.e.\ $(\forall x\in D)(x \preceq x)$.  
\end{enumerate}
For such $(D, \preceq)$ and topological space $X$, any mapping $x:D\di X$ is a \emph{net} in $X$.  
We denote $\lambda d. x(d)$ as `$x_{d}$' or `$x_{d}:D\di X$' to suggest the connection to sequences.  
The directed set $(D, \preceq)$ is not always explicitly mentioned together with a net $x_{d}$.
\edefi 
Except for Section \ref{moar}, we only use directed sets that are subsets of Baire space, i.e.\ as given by Definition \ref{strijker}.   
Similarly, we only study nets $x_{d}:D\di \R$ where $D$ is a subset of Baire space.  Thus, a net $x_{d}$ in $\R$ is just a type $1\di 1$ functional with extra 
structure on its domain $D$ provided by `$\preceq$' as in Definition~\ref{strijker}, i.e.\ part of third-order arithmetic. 

\smallskip

The definitions of convergence and increasing net are of course familiar.  
\bdefi[Convergence of nets]\label{convnet}
If $x_{d}$ is a net in $X$, we say that $x_{d}$ \emph{converges} to the limit $\lim_{d} x_{d}=y\in X$ if for every neighbourhood $U$ of $y$, there is $d_{0}\in D$ such that for all $e\succeq d_{0}$, $x_{e}\in U$. 
\edefi
\bdefi[Increasing nets]
A net $x_{d}:D\di \R$ is \emph{increasing} if $a\preceq b$ implies $x_{a}\leq_{\R} x_{b} $ for all $a,b\in D$.
\edefi
\bdefi\label{clusterf}
A point $x\in X$ is a \emph{cluster point} for a net $x_{d}$ in $X$ if every neighbourhood $U$ of $x$ contains $x_{u}$ for some $u\in D$.
\edefi
The previous definition yields the following nice equivalence: a toplogical space is compact if and only if every net therein has a cluster point (\cite{zonderfilter}*{Prop.\ 3.4}).
All the below results can be formulated using cluster points \emph{only}, but such an approach does not address the question what the counterpart of `sub-sequence' for nets is. 
Indeed, an obvious next step following Definition \ref{clusterf} is to take smaller and smaller neighbourhoods around the cluster point $x$ and (somehow) say that the associated points $x_{u}$ net-converge to $x$.   
To this end, we consider the following definition, first introduced by Moore in \cite{moringpool}, and used by Kelley in \cite{ooskelly}.  
Alternative definitions involve extra requirements (see \cite{zot}*{\S7.14}), i.e.\ our definition is the weakest. 
\bdefi[Sub-nets]\label{demisti}
A \emph{sub-net} of a net $x_{d}$ with directed set $(D, \preceq_{D})$, is a net $y_{b}$ with directed set $(B, \preceq_{B})$ such that there is a function $\phi : B \di D$ such that:
\begin{enumerate}
 \renewcommand{\theenumi}{\alph{enumi}}
\item the function $\phi$ satisfies $ y_{b} = x_{\phi(b)},$
\item $(\forall d\in D)(\exists b_{0}\in B)(\forall b\succeq_{B} b_{0})(\phi(b)\succeq_{D} d)$.
\end{enumerate}
\edefi
We point out that the distinction between `$\preceq_{B}$' and `$\preceq_{D}$' is not always made in the literature (see e.g.\ \cite{zonderfilter, ooskelly}).  

\smallskip

Finally, we need to discuss the connection between nets and sequences. 
\begin{rem}[Nets and sequences]\label{memmen}\rm
First of all, $\N$ with its usual ordering yields a directed set, i.e.\ convergence results about 
nets do apply to sequences.  Of course, a \emph{sub-net} of a sequence is not necessarily a \emph{sub-sequence}, i.e.\ some care is advisable in these matters.  
Nonetheless, the Bolzano-Weierstrass theorem \emph{for nets} does for instance imply the {monotone convergence theorem} \emph{for sequences} (see \cite{samnetspilot}*{\S3.1.1}).

\smallskip

Secondly, the Bolzano-Weierstrass (or monotone convergence) theorem for \emph{countable} (or continuous on Baire space) nets can be formulated in the language of second-order arithmetic and constitutes a trivial
extension of the original.  Following Remark \ref{unbeliever}, we do not distinguish between them.   
 \end{rem}
On a historical note, Vietoris introduces the notion of \emph{oriented set} in \cite{kliet}*{p.~184}, which is exactly the notion of `directed set'.  He proceeds to prove (among others)
a version of the Bolzano-Weierstrass theorem for nets.  Vietoris also explains that these results are part of his dissertation, written in the period 1913-1919, i.e.\ during his army service for the Great War.
\section{Main results I: convergence of nets}\label{karmichael}
We introduce the axiom $\BOOT$ and related notions in Section~\ref{bookstrap}.  
In Sections~\ref{kawl}-\ref{motkrijgen}, we establish equivalences involving basic convergence theorems for nets and $\BOOT$, as laid out in Section \ref{bootstraps}.  
We point out Section \ref{liften} in which we re-obtain some of these implications by `lifting' well-known second-order results to higher-order arithmetic.  
The aforementioned results deal with nets \emph{in the unit interval} and \emph{indexed by Baire space}.  In Section \ref{wonker}, we show that interesting phenomena occur when
either of these restrictions is lifted.

\subsection{Introduction: the bootstrap hierarchy}\label{bookstrap}
The results in \cite{samwollic19, samcie19,samnetspilot} establish that basic convergence theorems for nets are extremely hard to prove, while the limits therein are similarly hard to compute.  
In this paper, we show that the first-order strength of such theorems can also `explode', i.e.\ increase dramatically when combined with certain comprehension axioms.
These results in turn give rise to the hierarchy described in Section~\ref{bootstraps}.   
To this end, we show in the next sections that various convergence theorems for nets imply, or are even equivalent to, the following higher-order comprehension axiom.  
\bdefi[$\BOOT$]
$(\forall Y^{2})(\exists X^{1})(\forall n^{0})\big[ n\in X \asa (\exists f^{1})(Y(f, n)=0)    \big]. $
\edefi
The formula in the right-hand side of $\BOOT$ is called a `$\Sigma$-formula'.
The name `$\BOOT$' derives from the word `bootstrap'.   
We refer to the hierarchy formed by $\SIXK+\BOOT$ as the \emph{bootstrap hierarchy} as the logical strength of the latter system (in casu at least $\SIXko$) is `bootstrapped' from two essential parts, namely $\SIXK$ and $\BOOT$ that are weak(er) in isolation.  
\begin{thm}\label{boef}
The system $\SIXK+\BOOT$ proves $\SIXko$.  The system $\RCAo+\BOOT$ proves the same second-order sentences as $\ACA_{0}$.  Moreover, $\RCA_{0}$ proves $\ACA_{0}\asa [\BOOT]_{\ECF}$.
\end{thm}
\begin{proof}
For the first part, a $\Pi_{k+1}^{1}$-formula from $\L_{2}$ is clearly equivalent to a formula of the form $(\forall f^{1})(Y(f, n)=0) $ given $\SS_{k}^{2}$. 
For the second part, $\RCAo+\BOOT$ readily proves $\ACA_{0}$, while the $\ECF$-translation establishes that $\BOOT$ proves the same second-order sentences as $\ACA_{0}$.  
Indeed, as discussed in Remark \ref{ECF}, the $\ECF$-translation replaces the functional $Y^{2}$ in $\BOOT$ by a total associate $\alpha^{1}$, 
i.e.\ the right-hand side of $[\BOOT]_{\ECF}$ is thus $(\exists f^{1})(\exists m^{0})(\alpha(\overline{f}m, n)=1)  $.
Given $\ACA_{0}$, there is clearly a set $X$ that collects all $n$ satisfying this formula.  
\end{proof}
The previous theorem is hardly surprising given the form of $\BOOT$.  By contrast, the equivalence between $\BOOT$ and the monotone convergence theorem $\MCT_{\net}^{C}$ for \emph{nets} in Cantor space indexed by Baire space from Section \ref{kawl} is rather surprising, in our opinion.  
Moreover, the addition of \emph{moduli of convergence} for nets gives rise to an equivalence involving $\BOOT$ and countable choice in Section \ref{cauf}.  
The Moore-Osgood theorem for nets is shown to exhibit similar behaviour in Section \ref{motkrijgen}.    
By Theorem \ref{boef}, these convergence theorems give rise to the `bootstrap hierarchy' and variations.  
We note in passing that the usual `excluded middle' trick yields the cute disjunction $\ACA_{0}\asa [\BOOT\vee(\exists^{2})]$, which is converted into a tautology by $\ECF$.  

\smallskip

Following Remark \ref{unbeliever}, $\ECF$ maps equivalences like $\MCT_{\net}^{[0,1]}\asa \BOOT$, to well-known RM-equivalences, like the equivalence between arithmetical comprehension and the monotone convergence theorem \emph{for sequences} (\cite{simpson2}*{III.2}).     
We stress that the $\ECF$-translation is \emph{the} canonical embedding of higher-order into second-order arithmetic, replacing as it does higher-order objects by the codes typical of the practise of RM and second-order arithmetic.  In the other direction, Theorems \ref{proofofconcept} and \ref{nerode} show that certain second-order proofs, namely involving Specker sequences, almost verbatim translate to proofs of $\MCT_{\net}^{[0,1]}\di \BOOT$ and generalisations.  The latter proofs are however not `optimal' as they use a non-trivial extension of $\RCAo$.   

\smallskip

Finally, we study two `more complicated' convergence theorems: for nets in the function space $[0,1]\di [0,1]$ (Section \ref{powpow}) and for nets with index sets beyond Baire space (Section \ref{moar}).
Section \ref{powpow} is interesting as we obtain a convergence theorem for nets in functions spaces -still in the language of third-order arithmetic- that `bootstraps itself', i.e.\ does not need additional 
comprehension axioms (like $\SS_{k}^{2}$ or even $\exists^{2}$) to become stronger and stronger.   Section \ref{moar} shows that our proofs easily generalise to higher types, while the general case is perhaps best treated in a set-theoretic framework.  Moreover, Section \ref{moar} provides (partial) answers to the questions: \emph{What is the Plato hierarchy a reflection of? What is the nature of this reflection?}
Indeed, we provide a translation that yields the equivalence $\MCT_{\net}^{C}\asa \BOOT$ from a similar equivalence $\MCT_{\net}^{1}\asa \BOOT^{1}$ involving index sets beyond Baire space.    

\smallskip

We finish this section with some historical remarks pertaining to $\BOOT$.
\begin{rem}[Historical notes]\label{hist}\rm
First of all, the bootstrap principle $\BOOT$ is definable in Hilbert-Bernays' system $H$ from the \emph{Grundlagen der Mathematik}; see \cite{hillebilly2}*{Supplement IV}.  In particular, the functional $\nu$ from \cite{hillebilly2}*{p.\ 479} immediately\footnote{The functional $\nu$ from \cite{hillebilly2}*{p.\ 479} is such that if $(\exists f^{1})A(f)$, the function $(\nu f)A(f)$ is the lexicographically least such $f^{1}$.  The formula $A$ may contain type two parameters, as is clear from e.g.\ \cite{hillebilly2}*{p.\ 481} and other definitions.} yields the set $X$ from $\BOOT$, viewing the type two functional $Y^{2}$ as a parameter; 
the use of `unspoken higher-order parameters' is common throughout \cite{hillebilly2}*{Supplement~IV}.  Thus, the Plato and G\"odel hierarchies have the same historical roots.  

\smallskip

Secondly, Feferman's axiom \textsf{(Proj1)} from \cite{littlefef} is similar to $\BOOT$.  The former is however formulated using sets, which makes it more `explosive' than $\BOOT$ in that full $\Z_{2}$ follows when combined with $(\mu^{2})$, as noted in \cite{littlefef}*{I-12}.  The axiom \textsf{(Proj1)} only became known to us after the results in this section were finished. 
\end{rem}

\subsection{Convergence theorems for nets}\label{kawl}
We show that a number of convergence theorems for nets gives rise to $\SIXko$ in combination with $\SIXK$.  
This is done by establishing the connection between these theorems and $\BOOT$.
\subsubsection{Bolzano-Weierstrass and related theorems}\label{BWS}
In this section, we study the Bolzano-Weierstrass theorem for nets and related theorems. 
\bdefi[$\BW_{\net}^{C}$]
A net in Cantor space indexed by a subset of Baire space has a convergent sub-net.
\edefi
\begin{thm}\label{stovokor}
The system $\ACAo+\BW^{C}_{\net}$ proves $\FIVE$.
\end{thm}
\begin{proof}
A $\Sigma_{1}^{1}$-formula $\varphi(n)\in \L_{2}$ is readily seen to be equivalent to a formula $(\exists g^{1})(Y(g, n)=0)$ for $Y^{2}$ defined in terms of $\exists^{2}$.
Let $D$ be the set of finite sequences in Baire space and let $\preceq_{D}$ be the inclusion ordering, i.e.\ $w\preceq_{D}v$ if $(\forall i<|w|)(\exists j<|v|)(w(i)=_{1}v(j))$.  Now define the net $f_{w}:D\di C $ as $f_{w}:=\lambda k.F(w, k)$ where $F(w, k)$ is $1$ if $(\exists i<|w|)(Y(w(i), k)=0)$, and zero otherwise. 
Using $\BW_{\net}^{C}$, let $\phi:B\di D$ be such that $\lim_{b}f_{\phi(b)}=f$.
We now establish this equivalence:
\be\label{nogisnekeer}
(\forall n^{0})\big[(\exists g^{1})(Y(g,n)=0)\asa f(n)=1\big].
\ee
For the reverse direction, note that for fixed $n_{0}$, if $Y(g, n_{0})>0$ for all $g^{1}$, then $f_{w}(n_{0})=0$ for any $w\in D$.  The definition of limit then implies $f(n_{0})=0$, i.e.\ we have established (the contraposition of) the reverse direction.  
For the forward direction in \eqref{nogisnekeer}, suppose there is some $n_{0}$ such that $(\exists g^{1})(Y(g,n_{0})=0)\wedge f(n_{0})=0$.
Now, $\lim_{b}f_{\phi(b)}=f$ implies that there is $b_{0}\in B$ such that for $b\succeq_{B} b_{0}$, we have $\overline{f_{\phi(b)}}n_{0}=\overline{f}n_{0}$, i.e.\ $f_{\phi(b)}(n_{0})=0$ for $b\succeq_{B} b_{0}$.  
Let $g_{0}^{1}$ be such that $Y(g_{0}, n_{0})=0$, and use the second item in Definition \ref{demisti} for $d=\langle g_{0}\rangle$, i.e.\ there is $b_{1}\in B$ such that $\phi(b)\succeq_{D} \langle g_{0}\rangle$ for any $b\succeq_{B} b_{1}$.
Now let $b_{2}\in B$ be such that $b_{2}\succeq_{B} b_{0}, b_{1}$ as provided by Definition \ref{nets}.  On one hand, $b_{2}\succeq_{B}b_{1}$ implies that $\phi(b_{2})\succeq_{D}\langle g_{0}\rangle$, and hence $f_{\phi(b_{2})}(n_{0})=F(\phi(b_{2}), n_{0})=1$, as $g_{0} $ is in the finite sequence $ \phi(b_{2})$ by the definition of $\preceq_{D}$.  On the other land, $b_{2}\succeq b_{0}$ implies that $f_{\phi(b_{2})}(n_{0})=f(n_{0})=0$, a contradiction.  
Hence, \eqref{nogisnekeer} follows, yielding $\{n: \varphi(n)   \}$, as required by $\FIVE$.
\end{proof}
The previous theorem is elegant, but hides an important result involving the \emph{monotone convergence theorem for nets}. 
 As to its provenance, the latter theorem can be found in e.g.\ \cite{obro}*{p.\ 103}, but is also implicit in domain theory (\cites{gieren, gieren2}).  
Indeed, the main objects of study of domain theory are \emph{dcpos}, i.e.\ directed-complete posets, and an increasing net converges to its supremum in a dcpo.
\bdefi[$\MCT_{\net}^{C}$]
Any increasing net in $C$ indexed by a subset of $\N^{\N}$ converges in $C$.
\edefi
Note that we use the \emph{lexicographic ordering} $\leq_{\lex}$ on $C$ in the previous definition, i.e.\ $f\leq_{\lex}g$ if either $f=_{1}g$ or there is $n^{0}$ such that $\overline{f}n=\overline{g}n$ and $f(n+1)<g(n+1)$.
\begin{thm}\label{bongra}
The system $\RCAo$ proves that $\MCT_{\net}^{C}\asa \BOOT$.
\end{thm}
\begin{proof}
We first prove the equivalence assuming $(\exists^{2})$.
For the forward direction, fix some $Y^{2}$ and consider $f_{w}$ from the proof of the theorem.    Note that $v\preceq_{D}w \di f_{v}\leq_{\lex}f_{w}$, i.e.\ this net is indeed increasing.  
Let $f=\lim_{w}f_{w}$ be the limit provided by $\MCT_{\net}^{C}$ and verify that \eqref{nogisnekeer} also holds in this case.  In this way, we obtain the equivalence required by $\BOOT$.  
Note that $\exists^{2}$ is necessary for defining `$\preceq_{D}$'.  

\smallskip

For the reverse direction, let $x_{d}:D\di C$ be an increasing net in $C$ and consider the formula $(\exists d\in D)(x_{d}\geq_{\lex} \sigma*00\dots)$, where $\sigma^{0^{*}}$ is a finite binary sequence. 
The latter formula is equivalent to a formula of the form $(\exists g^{1})(Y(g, n)=0)$ where $Y^{2}$ is defined in terms of $\exists^{2}$ and $n$ codes a finite binary sequence.  
To define the limit $f$ required by $\MCT_{\net}^{C}$, $f(0)$ is $1$ if  $(\exists d\in D)(x_{d}\geq_{\lex} 100\dots)$ and zero otherwise.  
One then defines $f(n+1)$ in terms of $\overline{f}n$ in the same way, yielding the equivalence $\MCT_{\net}^{C}\asa \BOOT$ given $(\exists^{2})$.

\smallskip

Next, we establish the theorem assuming $\neg(\exists^{2})$, which implies that all functionals on Baire space are continuous (see \cite{kohlenbach2}*{\S3}).  
In this light, $\BOOT$ reduces to (essentially) $\ACA_{0}$ by the proof of Theorem \ref{boef}.  
Similarly, any formula involving a type one quantifier $(\exists d\in D)(\dots x_{d}\dots)$ may be equivalently replaced by $(\exists \sigma^{0^{*}})(\sigma*00\dots \in D \wedge \dots x_{\sigma*00\dots}\dots)$, which now involves a type zero quantifier (modulo coding).
Thus, $\MCT_{\net}^{C}$ also reduces to (essentially) the 
monotone convergence theorem for sequences, and the latter is equivalent to $\ACA_{0}$ by \cite{simpson2}*{III.2}. 
Hence, we have proved the theorem in both cases and the law of excluded middle $(\exists^{2})\vee \neg(\exists^{2})$ finishes the proof.  
\end{proof}
We can formulate the previous theorem in terms of classical computability theory as follows; 
let `$\leq_{T}$' be the usual Turing reducibility relation and let $J(Y)$ be the set $\{n: (\exists f^{1})(Y(f, n)=0)\}$.  
The forward direction of Theorem \ref{bongra} becomes:
\[\textstyle
\textup{for any $Y^{2}$, there is a net $x_{d}:D\di I $ such that $x=\lim_{d}x_{d}\di  J(Y)\leq_{T}x$.}
\]
Note that the net $x_{d}$ can be defined in terms of $Y^{2}$ via a term of G\"odel's $T$.
Moreover, $\ECF$ converts this statement into actual classical computability theory. 
%
\smallskip

Let $\MCT_{\seq}^{C}$ be the monotone convergence theorem for sequences in $C$, which is equivalent to $\ACA_{0}$ by \cite{simpson2}*{III.2}.  
The $\ECF$-translation converts $\MCT_{\net}^{C}\asa \BOOT$ into $\MCT_{\seq}^{C}\asa \ACA_{0}$ following Remark \ref{unbeliever}.  
Indeed, if a net $x_{d}$ is continuous in $d$, then $(\exists d\in D)(x_{d}> y)$ is equivalent to a $\Sigma_{1}^{0}$-formula 
and the `usual' interval halving proof goes through for $[\MCT_{\net}^{C}]_{\ECF}$ given $\ACA_{0}$. 
\begin{cor}\label{stovokor2}
The systems $\SIXK+\BW^{C}_{\net}$ and $\SIXK+\MCT^{C}_{\net}$ prove $\SIXko$ \($k\geq 0$\).  The system $\Z_{2}^{\Omega}$ proves $\MCT_{\net}^{C}$.  
\end{cor}
\begin{proof}
By Theorem \ref{boef} and the fact that $(\exists^{3})$ trivially proves $\BOOT$.
\end{proof}
By the second part of the corollary, the power, strength, and hardness of $\MCT_{\net}^{C}$ have nothing to do with the Axiom of Choice.  
We actually study the connection between the latter and the convergence of nets in Section \ref{cauf}.  

\smallskip

Of course, there is nothing special about Cantor space in the previous results.  Let $\BW_{\net}^{[0,1]}$ and $\MCT_{\net}^{[0,1]}$ be respectively the Bolzano-Weierstrass and monotone convergence theorem for nets in the unit interval indexed by subsets of Baire space.  
\begin{cor}\label{corkorcor}
The system $\RCAo+\IND+\X$ proves $\BOOT$, for $\X$ equal to either $\BW_{\net}^{[0,1]}$ or $\MCT_{\net}^{[0,1]}$.  
\end{cor}
\begin{proof}
It is well-known that $\exists^{2}$ defines a functional $\eta^{1\di 1}$ that converts real numbers in $[0,1]$ into binary representation, choosing a tail of zeros whenever there are two possibilities.   
Now consider the following alternative version of \eqref{nogisnekeer}: 
\be\label{romane}
(\forall n^{0})\big[(\exists g^{1})(Y(g,n)=0)\asa \eta(x)(n)=1\big], 
\ee
where $x$ is the limit provided by $\BW_{\net}^{[0,1]}$ for the sub-net of the net $x_{w}:=\r(\lambda k.F(w, k))$. 
Note that \eqref{romane} only holds in case $x$ has a unique binary representation.  
In the case of non-unique binary representation of $x$, there is $n_{0}$ such that $(\exists g^{1})(Y(g,n)=0$ has the same truth value for $n\geq n_{0}$. 
Now use $\IND$ to establish that for every $m^{0}\geq 1$, there is $w_{0}$ of length $m$ such that $(\forall i<m)\big[(\exists g^{1})(Y(f, i)=0)  \di Y(w(i), i)=0 \big]$.
Hence, the `non-unique' case has been handled too.  Finally, the net $x_{w}$ is increasing (in the sense of $\leq_{\R}$), i.e.\ $\MCT_{\net}^{[0,1]}$ also establishes the corollary.
\end{proof}
The \emph{anti-Specker property for nets}, denoted $\AS_{\net}$, is studied in \cite{samnetspilot}*{\S3.1.3}.  Now, $\AS_{\net}$ essentially expresses that if a net converges to an isolated point, it is eventually constant.   Since $\AS_{\net}$ readily implies $\MCT_{\net}^{[0,1]}$ using classical logic, the former also implies $\BOOT$ by the previous corollary.  The same holds for the Arzel\`a and Ascoli-Arzel\`a theorems for nets studied in \cite{samnetspilot}*{\S3.2.2}.  
As it turns out, the index sets used in this section, essentially consisting of finite sets ordered by inclusion, are called \emph{phalanxes} by Tukey (\cites{niettukop}), a martial term that has not caught on. 

\subsubsection{Moduli of convergence}\label{cauf}
In this section, we study the additional power provided by \emph{modulus functions} for convergence theorems pertaining to nets.  
We first discuss our motivation for this study. 

\smallskip

First of all, given an `epsilon-delta' definition, a \emph{modulus} is a functional that provides the `delta' in terms of the `epsilon' and other data.  
Bolzano already made use of moduli of continuity (see \cite{russje}), while they are implicit in RM-codes for continuous functions by \cite{kohlenbach4}*{Prop.\ 4.4}.
E.H.\ Moore also suggests using moduli in \cite{moorelimit2}*{p.\ 632} in the context of `general limits', a predecessor to nets and \cite{moorsmidje}. 
In the case of convergent sequences in the unit interval, the existence of a modulus is readily provable in $\ACA_{0}$; thus
the extra information provided by a modulus (or rate) of convergence does not change the associated RM-results for convergence theorems as in \cite{simpson2}*{III.2}.  
By contrast, we show that enriching some of the above theorems with a modulus gives rise to an equivalence involving countable choice.  

\smallskip

Secondly, we need the notion of \emph{Cauchy net} (see e.g.\ \cite{ooskelly}*{p.\ 190}), defined as follows for $\R$.
 It goes without saying that such nets are the generalisation of the notion of Cauchy sequence to directed sets.  
\bdefi[Cauchy net]\label{caucau}
 A net $x_{d}:D\di \R$ is \emph{Cauchy} if $(\forall \eps>0)(\exists d\in D)(\forall e,f\succeq_{D} d)(|x_{e}-x_{f}|<\eps)$.
\edefi
\bdefi[Cauchy modulus]\label{caucau2}
 A net $x_{d}:D\di \R$ is \emph{Cauchy with a modulus} if there is $\Phi:\R\di D$ such that $(\forall \eps>0)(\forall e,f\succeq_{D} \Phi(\eps))(|x_{e}-x_{f}|<\eps)$.
\edefi
On one hand, the convergence of Cauchy \emph{sequences} in the unit interval is equivalent to $\ACA_{0}$ by \cite{simpson2}*{III.2.2}, i.e.\ we expect the generalisation to Cauchy nets 
to exhibit similar behaviour to $\MCT_{\net}^{[0,1]}$ 
One the other hand, $\MCT_{\net}^{[0,1]}$ obviously follows from the two following facts:
\begin{enumerate}
\item An increasing net in $[0,1]$ indexed by a subset of $\N^{\N}$ is Cauchy.\label{facile}
\item A Cauchy net in $[0,1]$ indexed by a subset of $\N^{\N}$ converges.  \label{burj}
\end{enumerate}
One readily shows that item \eqref{burj} gives rise to hierarchies as in Corollary \ref{corkorcor}, while item \eqref{facile} is provable in $\RCAo+\IND$.  
Item \eqref{facile} is therefore quite weak and we shall enrich it with a Cauchy modulus, as follows.  
\bdefi[$\CAU_{\mod}$]
An increasing net in $[0,1]$ is Cauchy {with a modulus}.
\edefi
\begin{thm}\label{weirdoooo}
The system $\ACAo+\CAU_{\mod}$ proves $\FIVE$.
\end{thm}
\begin{proof}
A $\Sigma_{1}^{1}$-formula $\varphi(n)\in \L_{2}$ is readily seen to be equivalent to a formula $(\exists f^{1})(Y(f, n)=0)$ for $Y^{2}$ defined in terms of $\exists^{2}$.
Let $D$ be the set of finite sequences in Baire space and let $\preceq_{D}$ be the inclusion ordering, i.e.\ $w\preceq_{D}v$ if $(\forall i<|w|)(\exists j<|v|)(w(i)=_{1}v(j))$.  Now define the net $x_{w}:D\di \R $ as $x_{w}:=\r(\lambda k.F(w, k))$ where $F(w, k)$ is $1$ if $(\exists i<|w|)(Y(w(i), k)=0)$, and zero otherwise. 
Note that $x_{w}$ is increasing by definition.  Let $\Phi:\N\di D$ be such that $(\forall k^{0})(\forall w,v\succeq_{D} \Phi(k))(|x_{w}-x_{v}|<\frac{1}{2^{k}})$. 
We now establish this equivalence:
\be\label{nogis}
(\forall n^{0})\big[(\exists f^{1})(Y(f,n)=0)\asa (\exists g^{1}\in \Phi(n))(Y(g,n)=0)\big].
\ee
The reverse direction in \eqref{nogis} is trivial.  For the forward direction, suppose there is some $n_{0}$ such that $(\exists f^{1})(Y(f,n_{0})=0)\wedge (\forall g^{1}\in \Phi(n_{0}))(Y(g,n_{0})>0)$.
Let $f_{0}^{1}$ be such that $Y(f_{0}, n_{0})=0$, implying $F(\Phi(n_{0}), n_{0})=0$ and $F(w_{0},n_{0})=1$ for $w_{0}:=\Phi(n_{0})*\langle f_{0}\rangle$.  Hence $|x_{\Phi(n_{0})}-x_{w_{0}}|\geq \frac{1}{2^{n_{0}}}$ and $w_{0}\succeq_{D} \Phi(n_{0})$, a contradiction. 
Thus, \eqref{nogis} holds and yields the set $\{n: \varphi(n)   \}$, as required by $\FIVE$.
\end{proof}
The proof of the theorem also yields a nice splitting as follows. 
\begin{cor}\label{floopy}
The system $\RCAo$ proves $\CAU_{\mod}\asa [\BOOT+\QFAC^{0,1}]$.
\end{cor}
\begin{proof}
For the reverse implication, the proof of Theorem \ref{bongra} yields $\BOOT\di\MCT_{\net}^{[0,1]}$ with minimal adaptation.
Let $x_{d}:D\di [0,1]$ be an increasing net and let $x\in [0,1]$ be the limit provided by $\MCT_{\net}^{[0,1]}$.  Now apply $\QFAC^{0,1}$ to the formula $(\forall k^{0})(\exists d\in D)(|x_{d}-x|<\frac{1}{2^{k}})$ and note that the 
resulting functional is a Cauchy modulus since $x_{d}$ is an increasing net. 
  
\smallskip  
  
For the forward implication, we again use $(\exists^{2})\vee \neg(\exists^{2})$.
In case $\neg(\exists^{2})$, all functions on Baire space are continuous by \cite{kohlenbach2}*{\S3}.  In this case, $\QFAC^{0,1}$ is immediate from $\QFAC^{0,0}$ (included in $\RCAo$) and $\BOOT$ 
reduces to $\ACA_{0}$ as noted in the proof of Theorem \ref{bongra}.  
In case of $(\exists^{2})$, the proof of the theorem yields \eqref{nogis}; $\BOOT$ and $\QFAC^{0,1}$ are now immediate as the right-hand side of \eqref{nogis} is decidable. 
\end{proof}
The definition of a `modulus of net convergence' is now obvious following Definition~\ref{caucau2}.  
Let $\MCT_{\mod}^{[0,1]}$ and $\BW_{\mod}^{[0,1]}$ be resp.\ $\MCT_{\net}^{[0,1]}$ and $\BW_{\net}^{[0,1]}$ with the addition of a modulus of convergence. 
\begin{cor}\label{lopsided}
The system $\ACAo+\BW_{\mod}^{[0,1]}$ proves $\BOOT+\QFAC^{0,1}$. 
\end{cor}
\begin{proof}
Immediate by the proof of the theorem and the observation that for an increasing net, a modulus of convergence of a sub-net is also a Cauchy modulus for the (original) net.  
\end{proof}
\begin{cor}
The system $\RCAo$ proves $\MCT_{\mod}^{[0,1]}\asa [\BOOT+\QFAC^{0,1}]$.
\end{cor}
\begin{proof}
By Theorem \ref{bongra} and Corollary \ref{floopy}.
\end{proof}
A similar result can now be obtained for the Arzel\`a and Ascoli-Arzel\`a theorems for nets studied in \cite{samnetspilot}*{\S3.2.2}.  
Moreover, to derive $\BW_{\net}^{[0,1]}$ from item \eqref{burj} at the beginning of this section, one requires $\COH_{\net}$, i.e.\ the statement \emph{any net in the unit interval contains a Cauchy sub-net}.  The associated property for \emph{sequences} is equivalent to $\COH$ from the RM zoo (see \cites{keuzer}).  Clearly, $\COH_{\net}$ upgraded with a modulus would also give rise to e.g.\ a version of Corollary \ref{lopsided}.
\subsubsection{Lifting second-order results}\label{liften}
We have obtained the equivalence $\MCT_{\net}^{[0,1]}\asa \BOOT$ in Section \ref{BWS}.  In this section, we show that 
the forward implication can also be obtained by `lifting' the second-order proof of $\MCT_{\seq}^{[0,1]}\di \ACA_{0}$ to higher-order arithmetic; $\MCT_{\seq}^{[0,1]}$ is the monotone convergence theorem \emph{for sequences}.  
On one hand, this result suggest that second-order and higher-order arithmetic are not as fundamentally different as often claimed (the author is guilty of some such claims). 
On the other hand, the `lifted' proofs are not optimal as they need a non-trivial extension of the base theory.  

\smallskip

First of all, the crux of numerous reversals $T\di \ACA_{0}$ is that the theorem $T$ (somehow) allows for the reduction of (certain) $\Sigma_{1}^{0}$-formulas to $\Delta_{1}^{0}$-formulas.  
Since $\Delta_{1}^{0}$-comprehension is included in $\RCA_{0}$, one then obtains $\Sigma_{1}^{0}$-comprehension or the existence of the range of arbitrary functions, and $\ACA_{0}$ follows.  
We now show that this technique elegantly extends to $\BOOT$, which in turn allows us to lift proofs from the second-order to the higher-order framework \emph{with minimal adaptation}.    

\smallskip

Secondly, $\ACA_{0}$ is equivalent 
to $\range$, i.e.\ the existence of the range of any one-to-one $f:\N\di \N$, by \cite{simpson2}*{III.1.3}; $\BOOT$ satisfies a similar equivalence involving the
existence of the range of any type two functional, as follows. 
\begin{thm}\label{rage}
The system $\RCAo$ proves that $\BOOT$ is equivalent to 
\be\label{myhunt}\tag{$\RANGE$}
(\forall G^{2})(\exists X^{1})(\forall n^{0})\big[n\in X\asa (\exists f^{1})(G(f)=n)  ].
\ee
\end{thm}
\begin{proof}
The forward direction is immediate.  For the reverse direction, define $G^{2}$ as follows for $n^{0}$ and $g^{1}$: put $G(\langle n\rangle *g)=n+1$ if $Y(g, n)=0$, and $0$ otherwise. 
Let $X\subseteq \N$ be as in $\RANGE$ and note that 
\[
(\forall m^{0}\geq 1 )( m\in X \asa (\exists f^{1})(G(f)=m)\asa (\exists g^{1})(Y(g, m-1)=0)  ).
\]
which is as required for $\BOOT$ after trivial modification. 
\end{proof}
It goes without saying that $[\RANGE]_{\ECF}$ is essentially $\range$, i.e.\ the existence of the range of any one-to-one $f:\N\di \N$, following Remark~\ref{unbeliever}.

\smallskip

Thirdly, our base theory plus countable choice proves the following higher-order version of $\Delta_{1}^{0}$-comprehension, by Theorem \ref{DELTA}.
\begin{align}
(\forall Y^{2}, Z^{2})\big[ (\forall n^{0})( (\exists f^{1})&(Y(f, n)=0) \asa (\forall g^{1})(Z(g, n)=0) )\tag{$\Delta$-comprehension} \\
& \di (\exists X^{1})(\forall n^{0})(n\in X\asa (\exists f^{1})(Y(f, n)=0)\big]\notag
\end{align}
Note that the $\ECF$-translation converts $\Delta$-comprehension into $\Delta_{1}^{0}$-comprehension, while $\QFAC^{0,1}$ becomes $\QFAC^{0,0}$, following Remark \ref{unbeliever}.
As shown in \cite{dagsamX}, $\Delta$-comprehension is perhaps the weakest comprehension principle that still implies that there is no bijection from $[0,1]$ to $\N$ (using the usual definition from set theory).
\begin{thm}\label{DELTA}
The system $\RCAo+\QFAC^{0,1}$ proves $\Delta$-comprehension.
\end{thm}
\begin{proof}
The antecedent of $\Delta$-comprehension implies the following
\be\label{hani}
(\forall n^{0})(\exists g^{1}, f^{1})( Z(g, n)=0\di Y(f, n)=0 ).
\ee
Applying $\QFAC^{0,1}$ to \eqref{hani} yields $\Phi^{0\di 1}$ such that 
\be\label{fok}
(\forall n^{0})\big( (\forall g^{1})(Z(g, n)=0)\di Y(\Phi(n), n)=0 \big),
\ee
and by assumption an equivalence holds in \eqref{fok}, and we are done. 
\end{proof}
The previous theorem demonstrates its importance in the following proof.  
Indeed, the very first reversal in Simpson's monograph can be found in \cite{simpson2}*{III.2.2}, which is the implication $\MCT_{\seq}^{[0,1]}\di\ACA_{0}$ via an intermediate step involving $\range$;
the (second part of the) following proof is exactly Simpson's proof of $\MCT_{\seq}^{[0,1]}\di \range$, save for the replacement of sequences by nets. 
\begin{thm}\label{proofofconcept}
The system $\RCAo+\QFAC^{0,1}$ proves $\MCT_{\net}^{[0,1]}\di \BOOT$.
\end{thm}
\begin{proof}
In case $\neg(\exists^{2})$, note that $\MCT_{\net}^{[0,1]}$ also implies $\MCT_{\seq}^{[0,1]}$ as sequences are nets with directed set $(\N, \leq_{\N})$.  By \cite{simpson2}*{III.2}, $\ACA_{0}$ is available, which readily implies $\BOOT$ for continuous $Y^{2}$, but all functions on Baire space are continuous by \cite{kohlenbach2}*{\S3}.  

\smallskip

In case $(\exists^{2})$, we shall establish $\RANGE$ and obtain $\BOOT$ by Theorem \ref{rage}.  
Now fix some $Y^{2}$ and let $(D, \preceq_{D})$ be a directed set with $D$ consisting of the finite sequences $w^{1^{*}}$ in $\N^{\N}$ such that $(\forall i, j<|w|)(Y(w(i)=Y(w(j)))\di i=j)$ and $v\preceq_{D} w $ if $ (\forall i<|v|)(\exists j<|w|)(v(i)=_{1}w(j))$. 
 Define the net $c_{w}:D\di [0,1]$ as $c_{w}:= \sum_{i=0}^{|w|-1}2^{-Y(w(i))}$.  
Clearly, $c_{w}$ is increasing and let $c$ be the limit provided by $\MCT_{\net}^{[0,1]}$.  Now consider the following equivalence:
\be\label{kikop}
(\exists f^{1})(Y(f)=k)\asa (\forall w^{1^{*}})\big( |c_{w}-c|<2^{-k}\di (\exists g\in w)(Y(g)=k)     \big), 
\ee
for which the reverse direction is trivial thanks to $\lim_{w}c_{w}=c$.  For the forward direction in \eqref{kikop}, assume the left-hand side holds for $f=f_{1}^{1}$ and fix some $w_{0}^{1^{*}}$ such that $|c-c_{w_{0}}|<\frac{1}{2^{k}}$.  
Since $c_{w}$ is increasing, we also have $|c-c_{w}|<\frac{1}{2^{k}}$ for $w\succeq_{D} w_{0}$.  
Now there must be $f_{0}$ in $w_{0}$ such that $Y(f_{0})=k$, as otherwise $w_{1}=w_{0}*\langle f_{1}\rangle$ satisfies $ w_{1}\succeq_{D}w_{0}$ but also $c_{w_{1}}>c$, which is impossible.  

\smallskip

Note that \eqref{kikop} has the right form to apply $\Delta$-comprehension (modulo some coding), and the latter provides the set required by $\RANGE$.
\end{proof}
The net $c_{w}$ from the proof should be called a \emph{Specker net}, similar to \emph{Specker sequences}, pioneered in \cite{specker}. 
In light of the previous (and \cite{samrecount, samFLO2}), proofs from classical RM can be `recycled' as proofs related to the Plato hierarchy.  
The aforementioned `reuse' comes at a cost however: the proof of $\MCT_{\net}^{[0,1]}\di \BOOT$ in Theorem~\ref{bongra} does not make use of countable choice.  
The previous is not an isolated case: many so-called recursive counterexamples give rise to reversals in RM, and these results can often be lifted 
to obtain higher-order results, as studied in \cite{samrecount, samFLO2} for a variety of topics in RM.     
We list another example of the reuse of recursive counterexamples (to even higher types) in Section~\ref{moar}.

\subsection{The Moore-Osgood theorem for nets}\label{motkrijgen}
We study the \emph{Moore-Osgood theorem} which provides a sufficient criterion for the existence of double limits.  
We show that this theorem \emph{for nets} is explosive in the same way as in the previous sections. 
Our motivation is that the above proofs can be viewed as a kind of double limit construction involving nets and sequences.

\smallskip

As to history, E.\ H.\ Moore's version of the Moore-Osgood theorem apparently goes back to 1900 (see \cite{graf}*{p.\ 100}), while Osgood's version goes back to 1907 (see \cite{osgoed}).  
As expected, Moore-Smith deal with double (net) limits in \cite{moorsmidje}*{\S7}.
We use the following version of the Moore-Osgood theorem, similar to \cite{twoapp}*{Lemma 2.3}, where $D$ is assumed to be a subset of Baire space. 
\bdefi[$\MOT$] Let $(D, \preceq_{D})$ be a directed set with $D\subseteq\N^{\N}$.  
For a sequence of nets $x_{d, n}:(D\times \N)\di [0,1]$, if $\lim_{n\di \infty}x_{d, n}=y_{d}$ for some net $y_{d}:D\di [0,1]$ and if the net $\lambda{d}.x_{d,n}$ is uniformly Cauchy, then $\lim_{d}y_{d}=z$ for $z\in [0,1]$. 
\edefi
A sequence of nets $x_{d, n}$ is \emph{uniformly Cauchy} if the $d$ claimed to exist by Definition~\ref{caucau} does not depend on the sequence parameter $n$.  This definition is equivalent to uniform convergence in $\Z_{2}^{\Omega}+\QFAC^{0,1}$.
We use uniform Cauchyness because one generally needs non-trivial comprehension and choice to obtain a \emph{sequence} of limits from the existence of the individual limits $\lim_{d}x_{d,n}$ for all $n$.
\begin{thm}
The system $\ACAo+\IND+\MOT$ proves $\FIVE$.
\end{thm}
\begin{proof}
A $\Sigma_{1}^{1}$-formula $\varphi(n)\in \L_{2}$ is readily seen to be equivalent to a formula $(\exists f^{1})(Y(f, n)=0)$ for $Y^{2}$ defined in terms of $\exists^{2}$.
Let $D$ be the set of finite sequences in Baire space and let $\preceq_{D}$ be the inclusion ordering, i.e.\ $w\preceq_{D}v$ if $(\forall i<|w|)(\exists j<|v|)(w(i)=_{1}v(j))$.  Now define $F(w, k)$ as $1$ if $(\exists i<|w|)(Y(w(i), k)=0)$, 
and zero otherwise, and define the sequence of nets  $x_{w,k}:=\sum_{i=0}^{k}\frac{F(w, i)}{2^{i+1}}$.  By definition, we have $\lim_{k\di \infty} x_{w, k}= y_{w}$, where $y_{w}:=\sum_{i=0}^{\infty}\frac{F(w,i)}{2^{i+1}}$.
To prove that $x_{w,k}$ is uniformly Cauchy, use $\IND$ to establish that for every $m^{0}\geq 1$, there is $w$ of length $m$ such that $(\forall i<m)\big[(\exists g^{1})(Y(g, i)=0)  \di Y(w(i), i)=0 \big]$.
For $m\geq1$ and such $w$, note that $x_{v, k}$ is below $x_{w, k}+\frac{1}{2^{m}}$ for any $ k$ and $v\succeq_{D} w$, i.e.\ uniform Cauchyness. 

\smallskip

Let $z$ be the limit provided by $\MOT$, i.e.\ $\lim_{w}y_{w}=z$.
One now readily establishes the following equivalence for $\eta$ as in the proof of Corollary \ref{corkorcor}:
\be\label{nogisnekeerkes}
(\forall n^{0})\big[(\exists g^{1})(Y(g,n)=0)\asa \eta(z)(n)=1\big].
\ee
Clearly, \eqref{nogisnekeerkes} yields $\{n: \varphi(n)   \}$, as required by $\FIVE$.
\end{proof}
Finally, one can obtain $\BOOT$ from $\MOT$ in the same way as in the previous sections, while introducing moduli would similarly yield $\QFAC^{0,1}$.
To establish $\BOOT\di \MOT$, note that $y_{d}$ is a Cauchy net due to the assumptions in $\MOT$.

\subsection{Stronger convergence theorems}\label{wonker}
We have previously studied the convergence of nets in the unit interval indexed by Baire space.  
In this section, we show that interesting phenomena occur when lifting some of these restrictions.
In particular, we study the strength of convergence of nets in function spaces indexed by Baire space (Section \ref{powpow}) and of nets in the unit interval with `larger' index sets beyond Baire space (Section \ref{moar})

\subsubsection{Convergence in function spaces}\label{powpow}
In the previous sections, we have studied a number of convergence theorems for nets that give rise to parallel hierarchies as sketched in Figure \ref{xxz}.
Of course, these theorems do not involve formula classes, but the associated hierarchies are still based on formula classes via $\SIXK$.  
In this section, we formulate $\MON$, a (third-order) convergence theorem for nets that does not need $\SIXK$ to bootstrap to the next level, but rather `bootstraps itself', i.e.\ 
$\RCAo+\MON$ can prove $\SIXK$ for any $k$, via longer and longer proofs. 
 
 \smallskip

Now, we have previously considered nets in basic spaces like $2^{\N}$ and $[0,1]$.  
While Moore-Smith in \cite{moorsmidje} limited themselves to nets in $\R$, Vietoris already studied nets in (much) more general spaces in \cite{kliet}, even in the early days of nets. 
Hence, it is a natural question how strong $\MCT_{\net}^{[0,1]}$ becomes for nets in e.g.\ function spaces.  Note that this generalisation still is part of the language of third-order arithmetic.  

\smallskip

In this section, we show that for nets in the function space $[0,1]\di [0,1]$, the associated monotone convergence theorem $\MON$ becomes extremely powerful, in that it implies $\SIXK$ for any $k$ \emph{without additional axioms}.
\bdefi[$\MON$]
Let $(D, \preceq_{D})$ be a directed set where $D\subseteq\N^{\N}$.
Any increasing net $F_{d}:D \di (I\di I)$ converges to some $H:I\di I$.
\edefi
\noindent
Recall that a net $F_{d}:D \di (I\di I)$ is \emph{increasing} if we have that:
\[
(\forall x\in I)(\forall d, e\in D)(d\preceq_{D} e\di F_{d}(x)\leq_{\R}F_{e}(x)).
\]
Due to the boundedness property of $F_{d}$, for fixed $x\in I$, the net $F_{d}(x)$ converges to some limit, 
and the limit function from $\MON$ is obtained by putting all these individual limits together.
Note that $\MON$ implies $\BOOT$ by Corollary~\ref{corkorcor}.  
However, $\MON$ is much more `explosive' than the latter by the following theorem. 
\begin{thm}\label{labelfree}
The system $\RCAo+\MON$ proves $(\SS^{2})$.
\end{thm}
\begin{proof}
First of all, we prove $\MON\di (\exists^{2})$.  Let $F_{n}$ be the piecewise linear function that is zero for $x= 0$ and $1$ for $x\geq \frac{1}{2^{n}}$.
Consider the directed set $(\N, \leq)$ and the net $F_{n}$.
The latter is increasing in that $(\forall n, m\in \N)(\forall x\in [0,1])(n\leq m \di F_{n}(x)\leq F_{m}(x))$, and hence $F_{n}$ has a limit $H:I\di I$ by $\MON$.  Clearly, $H(0)=0$ and $H(x)=1$ for $x\in (0,1]$, i.e.\ $H$ is discontinuous, and \cite{kohlenbach2}*{\S3} yields $(\exists^{2})$.  

\smallskip

Secondly, note that the variable `$f$' in the definition of the Suslin funtional $(\SS^{2})$ can be restricted to Cantor space without loss of generality. 
Moreover, if $f\in C$ is eventually constant $0$ (resp.\ constant $1$), then $(\exists g^{1})(\forall n^{0})(f(\overline{g}n)=0)$ clearly holds (resp.\ does not hold).
Given $\exists^{2}$, we can decide whether $f\in C$ is eventually constant, i.e.\ we may restrict ourselves to $f\in C$ that are not eventually constant when defining the Suslin functional.   
Recall that $\exists^{2}$ defines a functional $\eta^{1\di 1}$ that converts real numbers in $[0,1]$ into binary representation, choosing a tail of zeros whenever there are two possibilities.   

\smallskip

Now, let $D$ be the set of finite sequences in Baire space and let $\preceq_{D}$ be the inclusion ordering, i.e.\ $w\preceq_{D}v$ if $(\forall i<|w|)(\exists j<|v|)(w(i)=_{1}v(j))$.  For $w^{1^{*}}\in D$, define the net $F_{w}(f)$ as $1$ if $(\exists g^{1}\in w)(\forall n^{0})(f(\overline{g}n)=0)$, and $0$ otherwise. 
Define $G_{w}:D\di (I \di I)$ as $G_{w}(x):=F_{w}(\eta(x))$.
Note that for $w\preceq_{D}v$, we have $G_{w}(x)\leq G_{w}(x)$ for all $x\in I$, i.e.\ $G_{w}$ is increasing in the sense of nets.  Let $H:I\di I$ be the limit $\lim_{w}G_{w}$ and consider:
\be\label{tatters}
(\forall f^{1}\in C)\big[ H_{0}(f)=1 \asa (\exists g^{1})(\forall n^{0})(f(\overline{g}n)=0)  \big], 
\ee
where $H_{0}(f)$ is $H(\r(f))$ if $\r(f)$ has a unique binary representation, and otherwise $0$ or $1$ depending on whether $f$ is eventually constant $0$ or eventually constant $1$.  
For any $f\in C$, \eqref{tatters} is immediate in the `otherwise' case in $H_{0}(f)$, by the above.  In the unique representation case, if $H_{0}(f)=H(\r(f))=1$ then the definition of limit implies that there is $w\in D$ such that for all $v\succeq_{D}w$, we have $G_{v}(\r(f))=F_{v}(f)=1$, which immediately yields the right-hand side of \eqref{tatters}.  Now let $g_{0}^{1}$ be such that $(\forall n^{0})(f(\overline{g_{0}}n)=0)$ in the unique representation case and suppose $H_{0}(f)=H(\r(f))=0$.  
Again by the definition of limit, there is $w\in D$ such that for all $v\succeq_{D}w$, we have $G_{v}(\r(f))=F_{v}(f)=0$.  This yields a contradiction for $v=w*\langle g_{0}\rangle$, and \eqref{tatters} follows. 
Clearly, the latter defines $(\SS^{2})$. 
\end{proof}
\begin{cor}
For any $k$, the system $\RCAo+\MON$ proves $(\SS_{k}^{2})$.  
\end{cor}
\begin{proof}
To obtain $(\SS^{2}_{2})$, $(\exists g^{1})(\forall h^{1})(\exists n^{0})(f(\overline{g}n, \overline{h}n)=0)$ is equivalent to the formula $(\exists g^{1})(Y(f,g )=0)$, where $Y^{2}$ is defined in terms of $\SS^{2}$.
Now repeat the proof of the theorem step with `$(\forall n^{0})(f(\overline{g}n)=0)$' replaced by `$Y(f, g)=0$'.
\end{proof}
Finally, $\MON$ is not that much more `exotic' than e.g.\ $\MCT_{\net}^{[0,1]}$ by the following. 
\begin{thm}\label{dyrk}
The system $\RCAo$ proves $[\MCT_{\net}^{[0,1]}+\QFAC^{1,1}+(\exists^{2})]\di \MON$.
\end{thm}
\begin{proof}
Let $F_{d}$ be as in $\MON$.  By $\MCT_{\net}^{[0,1]}$, for fixed $x\in I$, the net $F_{d}(x)$ converges to some limit $y\in I$, implying the following formula:
\[\textstyle
(\forall x\in I)(\exists y\in I)\underline{(\forall k^{0})(\exists d\in D)(|F_{d}(x)-y|<\frac{1}{2^{k}})}.
\]
Apply $\QFAC^{0,1}$ to the underlined formula to obtain
\[\textstyle
(\forall x\in I)(\exists y\in I)(\exists d_{n}^{0\di 1}){(\forall k^{0})(|F_{d_{k}}(x)-y|<\frac{1}{2^{k}})},
\]
which qualifies for $\QFAC^{1,1}$ in the presence of $(\exists^{2})$ and coding of the second existential quantifier as a type one object. 
The resulting functional is the limit as required for $\MON$.
\end{proof}
The previous proof actually provides a modulus of convergence for the limit process $\lim_{d}F_{d}=H$.  
Moreover, introducing a modulus of convergence in $\MON$, one obtains \emph{mutatis mutandis} that the enriched principle implies $\QFAC^{1,1}$, and hence an equivalence in the previous theorem.
One can also prove that $\MON$ is equivalent to the following straightforward generalisation of $\BOOT$:
\[
(\forall Y^{2})(\exists G^{2})(\forall f^{1})(G(f)=0\asa (\exists g^{1})(Y(f,g)=0)).
\]
The proof is similar to that of Theorem \ref{labelfree}, and we therefore omit it. 

\subsubsection{Index sets beyond Baire space}\label{moar}
In this section, we study the Bolzano-Weierstrass theorem for nets with index sets beyond Baire space, namely subsets of $\N^{\N}\di \N$. 
Such index sets are also studied in \cite{samnetspilot}*{Appendix A} in the context of computability theory and RM, but we stress that these results  
are only given (here and in \cite{samnetspilot}) by way of illustration: the general study of nets is perhaps best undertaken in a suitable set theoretic framework.  
That is not to say this section should be dismissed as \emph{spielerei}; our results come with conceptual motivation as follows:
\begin{enumerate}
\item Index sets beyond Baire space do occur `in the wild', namely in e.g.\ \emph{fuzzy mathematics} and \emph{gauge integration}, by Remark \ref{fuzzytop}. 
\item It is a natural question whether the above proofs generalise to higher types.  
\item In light of Corollary \ref{stovokor2}, it is a natural question whether nets with index sets beyond Baire space take us beyond second-order arithmetic.  
\item It is a natural question whether $\ECF$ maps results pertaining to index sets beyond Baire space into second-order arithmetic.  
\item Nets with index sets beyond $\N^{\N}$ provide a partial answer to a question from Section \ref{pgintro}, namely what the Plato hierarchy is a reflection of.  
\end{enumerate}
As we will see below, the answer is positive for each of these questions. 
Thus, similar to Definition \ref{strijker}, we introduce the following.  
\bdefi[$\RCAo$]\label{strijker2}
A `subset $E$ of $\N^{\N}\di \N$' is given by its characteristic function $F_{E}^{3}\leq_{3}1$, i.e.\ we write `$Y\in E$' for $ F_{E}(Y)=1$ for any $Y^{2}$.
A `binary relation $\preceq$ on the subset $E$ of $ \N^{\N}\di \N$' is given by the associated characteristic function $G_{\preceq}^{(2\times 2)\di 0}$, i.e.\ we write `$Y\preceq Z$' for $G_{\preceq}(Y, Z)=1$ and any $Y, Z\in E$.
\edefi
\bdefi[$\BW_{\net}^{1}$]
Any net in Cantor space indexed by a subset of $\N^{\N}\di \N$ has a convergent sub-net.
\edefi
\begin{thm}\label{koonfin}
The system $\Z_{2}^{\Omega}+\BW^{1}_{\net}$ proves $\Pi_{1}^{2}\textup{-\textsf{CA}}_{0}$.
\end{thm}
\begin{proof}
A $\Sigma_{1}^{2}$-formula $\varphi(n)\in \L_{3}$ is readily seen to be equivalent to a formula $(\exists Y^{2})(Z(Y, n)=0)$ for $Z^{3}$ defined in terms of $\exists^{3}$.
Let $E$ be the set of finite sequences in $\N^{\N}\di \N$ and let $\preceq_{E}$ be the inclusion ordering, i.e.\ $w\preceq_{E}v$ if $(\forall i<|w|)(\exists j<|v|)(w(i)=_{2}v(j))$.  Define the net $f_{w}:E\di C $ as $f_{w}:=\lambda k.F(w, k)$ where $F(w, k)$ is $1$ if $(\exists i<|w|)(Z(w(i), k)=0)$, and zero otherwise. 
Using $\BW_{\net}^{1}$, let $\phi:B\di E$ and $f^{1}$ be such that $\lim_{b}f_{\phi(b)}=f$.
We now establish that
\be\label{nogisnekeer2}
(\forall n^{0})\big[(\exists Y^{2})(Z(Y,n)=0)\asa f(n)=1\big].
\ee
For the reverse direction, note that for fixed $n_{0}$, if $Z(Y, n_{0})=0$ for all $Y^{2}$, then $f_{w}(n_{0})=0$ for any $w\in E$.  The definition of limit then implies $f(n_{0})=0$, i.e.\ we have established (the contraposition of) the reverse direction.  
For the forward direction in \eqref{nogisnekeer2}, suppose there is some $n_{0}$ such that $(\exists Y^{2})(Z(Y,n_{0})=0)\wedge f(n_{0})=0$.
Now, $\lim_{b}f_{\phi(b)}=f$ implies that there is $b_{0}\in B$ such that for $b\succeq_{B} b_{0}$, we have $\overline{f_{\phi(b)}}n_{0}=\overline{f}n_{0}$, i.e.\ $f_{\phi(b)}(n_{0})=0$ for $b\succeq_{B} b_{0}$.  
Let $Y_{0}^{2}$ be such that $Z(Y_{0}, n_{0})=0$, and use the second item in Definition \ref{demisti} for $d=\langle Y_{0}\rangle$, i.e.\ there is $b_{1}\in B$ such that $\phi(b)\succeq_{E} \langle Y_{0}\rangle$ for any $b\succeq_{B} b_{1}$.
Now let $b_{2}\in B$ be such that $b_{2}\succeq_{B} b_{0}, b_{1}$ as provided by Definition \ref{nets}.  On one hand, $b_{2}\succeq_{B}b_{1}$ implies that $\phi(b_{2})\succeq_{E}\langle Y_{0}\rangle$, and hence $f_{\phi(b_{2})}(n_{0})=F(\phi(b_{2}), n_{0})=1$, as $Y_{0} $ is in the finite sequence $ \phi(b_{2})$ by the definition of $\preceq_{E}$.  On the other land, $b_{2}\succeq_{B} b_{0}$ implies that $f_{\phi(b_{2})}(n_{0})=f(n_{0})=0$, a contradiction.  
Hence the forward direction follows and so does \eqref{nogisnekeer2}, yielding the set $\{n: \varphi(n)   \}$, as required by $\Pi_{1}^{2}\textup{-\textsf{CA}}_{0}$.
\end{proof}
We now generalise Theorem \ref{bongra} to higher types.  To this end, inspired by \eqref{nogisnekeer2}, we generalise $\BOOT$ to $\N^{\N}\di \N$ as follows:  
\be\label{woot}\tag{$\BOOT^{1}$}
(\forall Z^{3})(\exists X^{1})(\forall n^{0})(n\in X\asa (\exists Y^{2})(Z(Y,n)=0)).
\ee
Similarly, let $\MCT_{\net}^{1}$ be the monotone convergence theorem based on $\BW_{\net}^{1}$.  
\begin{cor}\label{kook}
The system $\Z_{2}^{\Omega}$ proves $\BOOT^{1}\asa \MCT^{1}_{\net}$.
The $\L_{2}$-sentence $[\BOOT^{1}]_{\ECF}$ is provable in $\SIX$.    
\end{cor}
\begin{proof}
For the second part, let $\gamma^{1}$ be a total associate for $Z^{3}$ in $\BOOT^{1}$.  
The right-hand side of $[\BOOT^{1}]_{\ECF}$ is 
\be\label{Gprt}
(\exists \alpha^{1})\big((\forall \beta^{1})(\exists m^{0})(\alpha(\overline{\beta}m)>0)\wedge (\exists k^{0})(\gamma(\overline{\alpha}k,n)=1)  \big), 
\ee
and the set consisting of such $n^{0}$ is clearly definable in $\SIX$.

\smallskip

For the first part, the reverse direction follows in the same way as the proof of the theorem, i.e.\ \eqref{nogisnekeer2} also goes through for the limit provided by $\MCT_{\net}^{1}$.  The forward direction follows by the usual interval halving technique based on $\BOOT^{1}$, i.e.\ as
in the proof of Theorem \ref{bongra}.
\end{proof}
A problem with the previous results is that $(\exists^{3})$ seems needed, but $\ECF$ converts this axiom to `$0=1$', and the same for $(\exists^{2})$. 
We now introduce a `weaker' lossy translation that behaves better in this regard. 
For any $A\in \L_{\omega}$, let $[A]_{\PECF}$ be $A$ with any variable $Y^{2}$ restricted to $Y^{2}\in C$, i.e.\ we replace type two functionals by \emph{continuous} type two functionals (essentially as in $\ECF$), but do not modify higher types.  
We have the following result that suggests
that $\PECF$ converts `$\Z_{2}^{\Omega}\vdash[\BOOT^{1}\asa \MCT^{1}_{\net}]$' to `$\ACAo\vdash [\BOOT\asa \MCT_{\net}^{C}]$'.  
\begin{thm}\label{diak}
The system $\RCAo$ proves $[(\exists^{3})]_{\PECF}\asa (\exists^{2})$, while $\ACAo$ proves $[\BOOT^{1}]_{\PECF}\asa \BOOT$ and $[\MCT_{\net}^{1}]_{\PECF}\asa \MCT_{\net}^{C}$.
\end{thm}
\begin{proof}
First of all, any $Y^{2}\in C$ has a type one associate given $\ACA_{0}$ by \cite{kohlenbach4}*{\S4}.  
Thus, $(\exists Y^{2}\in C)(Z(Y,n)=0)$ is equivalent to $(\exists f^{1})(Z(F(f), n)=0)$, where $F^{1\di 2}$ is defined as $F(f)(g):=f\big(\overline{g}(\mu n)(f(\overline{g}n)>0)\big)-1$.  
Similarly modify $[\MCT_{\net}^{1}]_{\PECF}$ and $[(\exists^{3})]_{\PECF}$ to obtain principles provable from resp.\ $\MCT_{\net}^{C}$ and $(\exists^{2})$. 
\end{proof}
The previous provides a partial answer to a question from Section \ref{pgintro}, namely what the Plato hierarchy could be a reflection of.  
Our answer is only partial as $\PECF$ does not have as nice properties as $\ECF$: the former converts trivialities like $(\exists^{3})\di (\exists^{2})$ into $(\exists^{2})\di 0=1$.  
Perhaps a refinement of $\PECF$ will be seen to have better properties.  

\smallskip

Next, Specker nets are used in the proof of Theorem \ref{proofofconcept} to establish $\MCT_{\net}^{[0,1]}\di \RANGE$.
We show that this proof also readily generalises as follows.    
\begin{thm}\label{nerode}
The system $\Z_{2}^{\Omega}+\QFAC^{0,2}+\MCT^{1}_{\net}$ proves the following:
\be\label{myhunt1}\tag{$\RANGE^{1}$}
(\forall G^{3})(\exists X^{1})(\forall n^{0})\big[n\in X\asa (\exists Y^{2})(G(Y)=n)  ].
\ee
\end{thm}
\begin{proof}
A slight modification of the proof of Theorem \ref{proofofconcept} goes through as follows: let $E$ be the set of finite sequences in $\N^{\N}\di \N$ and let $\preceq_{E}$ be the inclusion relation, for which $\exists^{3}$ is needed (instead of $\exists^{2}$).  The Specker net $c_{w}:E\di [0,1]$ is defined in exactly the same way as in 
Theorem \ref{proofofconcept}, namely as $c_{w}:= \sum_{i=0}^{|w|-1}2^{-Z(w(i))}$, where $Z^{3}$ is given.  The associated version of \eqref{kikop} is:
\be\label{hugs}
(\exists Y^{2})(Z(Y)=k)\asa (\forall w^{2^{*}})\big( |c_{w}-c|<2^{-k}\di (\exists V\in w)(Z(V)=k)     \big), 
 \ee
 where $c=\lim_{w}c_{w}$ is provided by $\MCT_{\net}^{1}$.  Applying $\QFAC^{0,2}$ to \eqref{hugs} as in the proof of Theorem~\ref{DELTA}
yields the set $X\subset \N$ required for $\RANGE^{1}$.
\end{proof}
In light of the proofs of Theorem \ref{koonfin} and \ref{nerode}, it is now be clear that the above proofs readily generalise to higher types.  
To avoid repetition, we do not study further generalisations of convergence theorems for nets in this paper.  
We do list some nice results: let $\BW_{\net}^{\sigma}$ be the obvious generalisation of $\BW_{\net}^{1}$ to index sets of type $\sigma+1$ objects.  
A straightforward modification of Theorem \ref{koonfin} implies that $\RCAo+(\exists^{k+2})+\BW_{\net}^{k}$ proves $\Pi_{1}^{k+1}$-comprehension for $k\geq 1$. 
Hence, the general Bolzano-Weierstrass theorem for nets is extremely hard to prove.

\smallskip

Recall Corollary \ref{floopy} which implies $\CAU_{\mod}\asa \QFAC^{0,1}$ over $\Z_{2}^{\Omega}$.
Let $\CAU_{\mod}^{2}$ be the generalisation of $\CAU_{\mod}$ to index sets that are subsets of $\N^{\N}\di \N$.
\begin{cor}\label{higherandhigher2}
The system $\RCAo+(\exists^{4})$ proves $\QFAC^{0,2}\asa \CAU^{2}_{\mod}$.
\end{cor}
\begin{proof}
Generalise the proof of Theorem \ref{weirdoooo} in the same way as Theorem \ref{koonfin}. 
\end{proof}
Let $\CAU^{\sigma}_{\mod}$ be the obvious generalisation of $\CAU_{\mod}^{2}$ to sets of type $\sigma+1$ objects.
One then readily proves $\QFAC^{0,k}\asa \CAU_{\net}^{k}$ over $\RCAo+(\exists^{k+2})$.

\smallskip

We finish this section with a conceptual remark on `large' index sets and their occurrence in mathematics and logic.  
\begin{rem}[Large index sets]\label{fuzzytop}\rm
First of all, Zadeh founded the field of \emph{fuzzy mathematics} in \cite{zadeh65}.  
The core notion of \emph{fuzzy set} is a mapping that assigns values in $[0,1]$, i.e.\ a `level' of membership, rather than the binary relation from usual set theory.  
The first two chapters of Kelley's \emph{General Topology} (\cite{ooskelly}) are generalised to the setting of fuzzy mathematics in \cite{pupu}.  
As an example, \cite{pupu}*{Theorem 11.1} is the fuzzy generalisation of the classical statement that a point is in the closure of a set if and only if there is a net that converges to this point.  
However, as is clear from the proof of this theorem, to accommodate fuzzy points in $X$, the net is indexed by the space $X\di [0,1]$.  

\smallskip

Secondly, the \emph{iterated limit theorem} (both the fuzzy and classical versions: \cite{pupu}*{Theorem 12.2} and \cite{ooskelly}) involves 
an index set $E_{m}$ \emph{indexed by $m\in D$}, where $D$ is an index set.  Thus, `large' index sets are found in the wild.       

\smallskip

Thirdly, by way of an exercise, the reader should generalise the well-known formulation of the Riemann integral in terms of nets (see e.g.\ \cite{ooskelly}*{p.\ 79}) to the gauge integral as studied in \cite{dagsamIII}*{\S3.3}. 
As will become clear, this generalisation involves nets indexed by $\R\di \R$-functions, and this very definition can also be found in the literature, namely \cite{leider}*{\S1.3}.

\smallskip

Fourth, the results in \cite{samnetspilot}*{\S4.3-4.5} connect continuity and open sets to nets, all in $\R$, while avoiding the Axiom of Choice.  As is clear from the proofs (esp.\ the use of the net $x_{d}:=d$), replacing $\R$ by a larger space requires the 
introduction of nets with a similarly large index set.  In particular, to show that a net-closed\footnote{A set $C$ is \emph{net-closed} if for any net in $C$ that converges to $x$, we also have $x\in C$ (\cite{ooskelly}*{p.\ 66}).} set $C$ is closed (see \cite{samnetspilot}*{Theorem 4.15} for $C\subseteq \R$), one seems to need nets with an index set the same cardinality as $C$.    
\end{rem}

\section{Main results II: open sets and Heine-Borel compactness}\label{main2}
\subsection{Introduction}\label{clonets}
In this section, we establish the results sketched in Section~\ref{bootstraps} pertaining to open sets and the axiom $\BOOT$, as well as the connection to Heine-Borel compactness.  
In particular, the latter connection is studied in Section~\ref{plato}, while we identify the `correct' notion of open set to be used in the Plato hierarchy and obtain interesting RM-results in Section~\ref{locatesect}. 
As will become clear, some of our results are straightforward generalisations of second-order equivalences, while others yield genuine surprises, like the Cantor-Bendixson theorem.  
In particular, the study of open sets in the Plato hierarchy directly inspires the higher-order counterparts of $\ATR_{0}$ and $\FIVE$, as will become clear in Section \ref{locatesect}.

\smallskip

We first discuss the intended meaning of `correct' notion of open set.
While such judgements are inherently subjective, we shall use the following two (more or less) objective criteria to 
judge whether a new notion of open set is acceptable.  
\begin{enumerate}
\item[(I)] The new notion of open set reduces to RM-codes of open sets under $\ECF$.
\item[(II)] The new notion of open set yields (lots of) equivalences that reduce to known (interesting) equivalences under $\ECF$.
\end{enumerate}
The first criterion is a basic requirement that merits no further discussion, while the second criterion is based on the so-called main theme of RM, expressed as follows:
\begin{quote}
very often, if a theorem of ordinary mathematics is proved from
the ``right'' set existence axioms, the statement of that theorem will be provably
equivalent to those axioms over some weak base system.
\end{quote}
This opinion may be found in e.g.\ \cites{browner, fried, simpson2} and many other places.  
In Section~\ref{locatesect}, we introduce a notion of open set consistent with the above items (I) and (II).  
We shall obtain a number of equivalences involving \emph{nets} rather than sequences.  
We stress that finding the `correct' generalisation of open set, namely uncountable unions as in Definition~\ref{opensset}, is non-trivial as follows.

\smallskip

Our initial motivation for the new notion of open set as in Definition \ref{opensset}, stems from \cite{dagsamVII, samnetspilot}; in the latter, open sets in $\R$ are given by (possibly discontinuous) characteristic functionals $Y:\R\di \R$ where `$x\in Y$' is short for $Y(x)>_{\R}0$.  
While this definition begets plenty of interesting results, it does not yield the expected reversals; 
Definition~\ref{opensset} is better this way in light of Theorem~\ref{yield}.
In other words, the concept of open set from \cite{dagsamVII, samnetspilot} satisfies (I) and not (II), but yields interesting results as follows.
\begin{rem}\label{nelta}\rm
First of all, nets obviate the (otherwise necessary) use of the Axiom of Choice in \cite{samnetspilot} as part of the study of open and closed sets via sequences/nets.

\smallskip

Secondly, the \emph{$\Delta$-functional} from \cite{dagsamVII}*{\S5} converts 
between two notions of open set based on characteristic functions, namely from a realiser for the usual definition of open set to a distance function for the complement.  It is proved in \cite{dagsamVII} that:
\begin{enumerate}
\item[(P1)] $\Delta$ is not computable in any type $2$ functional, but computable in any Pincherle realiser (see \cite{dagsamV}), a class weaker than $\Theta$-functionals (see \cite{dagsam, dagsamII}).
\item[(P2)] $\Delta$ is unique, genuinely type $3$, and adds no computational strength to $\exists^2$ in terms of computing functions from functions.
\end{enumerate}  
It was previously believed that functionals with the above properties would be ad hoc and could only be obtained via some complicated forcing construction. 
\end{rem}
We finish this section by noting that while our concept of open set is \emph{uncountable unions of basic opens} (see Definition \ref{opensset}), 
we could obtain all the below results working solely with countable unions of basic opens \emph{assuming the mainstream definition of `countable'}, as discussed in Remark \ref{cuntable}. 
\subsection{Open sets via uncountable unions}\label{locatesect}
\subsubsection{Open sets as uncountable unions}
In this section, we introduce a notion of open set consistent with items (I) and (II) from Section \ref{clonets}.  
In particular, we obtain some elegant equivalences involving locatedness and nets at the level of $\ACA_{0}$ (Section~\ref{blif}), and perfect sets at the level of $\ATR_{0}$ and $\FIVE$ (Section~\ref{blef}). 

\smallskip

First of all, we shall make use of the following notion of open set. 
Hereafter, `open' refers to the below definition, while `RM-open' refers to the well-known RM-definition from \cite{simpson2}*{II.5} involving countable unions of basic open balls.
\bdefi[Open sets]\label{opensset}
An open set $O$ in $\R$ is represented by a functional $\psi:\R\di \R^{2}$.  We write `$x\in O$' for $(\exists y\in \R)(x\in I_{y}^{\psi})$, where $I_{y}^{\psi}$ is the open interval $\big(\psi(y)(1), \psi(y)(1)+|\psi(y)(2)|\big)$ in case the end-points are different, and $\emptyset$ otherwise.  We write $O=\cup_{y\in \R}I_{y}^{\psi}$ to emphasise the connection to uncountable unions.  
A closed set is represented by the complement of an open set.       
\edefi
Intuitively, open sets are given by \emph{uncountable} unions $\cup_{y\in \R}I_{y}^{\psi}$, just like RM-open sets are given by countable such unions.  
Hence, our notion of open set reduces to the notion RM-open set when applying $\ECF$ or when all functions on $\R$ are continuous.  
Moreover, writing down the definition of elementhood in an RM-open set, one observes that such sets are also open (in our sense). 
Finally, closed sets are readily seen to be sequentially closed, and the same for nets instead of sequences.  

\smallskip

The following `coding principle' turns out to have nice properties.  Note that $\open$, a weaker version of $\open^{+}$, was introduced and studied in \cite{dagsamVII}.
We fix an enumeration of all basic open balls $B(q_{n}, r_{n})\subset \R$ for rational $q_{n}, r_{n}$ with $r_{n}>_{\Q}0$.
\bdefi[$\open^{+}$]
For every open set  $Z\subseteq \R$, there is $X\subset \N$ such that $(\forall n\in \N)(n\in X\asa B(q_{n}, r_{n})\subseteq Z)$.
\edefi
In the next section, we prove equivalences at the level of $\ACA_{0}$ involving $\BOOT$ and $\open^{+}$.  
Equivalences at the level of $\ATR_{0}$ and $\FIVE$ are in Section~\ref{blef}. 
\subsubsection{At the level of $\ACA_{0}$}\label{blif}
A number of theorems regarding RM-closed sets are equivalent to $\ACA_{0}$; we now generalise some of these results, based on Definition~\ref{opensset}.  
Recall that a closed set $C$ is called \emph{located} (see \cite{simpson2}*{IV.2.17} for the RM-notion) if the distance function $d(x, C):=\inf_{y\in C}d(x, y)$ exists as a continuous real-valued function.  
To be absolutely clear, `continuous' refers to the usual `epsilon-delta' definition, while `RM-continuous' refers to the RM-definition as in \cite{simpson2}*{II.6.1}.
\begin{thm}\label{yield}
The following are equivalent over $\RCAo+\QFAC^{0,1}$:
\begin{enumerate}
 \renewcommand{\theenumi}{\alph{enumi}}
\item $\open^{+}+\ACA_{0}$, \label{kopen}
\item Every non-empty closed set in $[0,1]$ is located, \label{b1asefq}
\item For every non-empty closed set $C\subseteq[0,1]$, the supremum $\sup C$ exists,\label{basef}
\item Monotone convergence theorem for nets in $[0,1]$ indexed by subsets of $\N^{\N}$,\label{mct}
\item For closed $C\subseteq[0,1]$ and $f:\R\di \R$ continuous on $C$, $\sup_{x\in C}f(x)$ exists, \label{baseflp}
\item For closed $C\subseteq[0,1]$ and $f:\R\di \R$ cont.\ on $C$, $f$ attains its maximum,\label{baseflpr}
\item $\BOOT$. \label{tropen}
\end{enumerate}
The axiom $\QFAC^{0,1}$ is only used for $\BOOT\di \open^{+}$.
\end{thm}
\begin{proof}
We first prove $\eqref{kopen}\di \eqref{b1asefq}\di \eqref{basef}\di \eqref{mct}\di \eqref{tropen}\di \eqref{kopen}$.
The implication \eqref{kopen}$\di$\eqref{b1asefq} follows from the usual second-order equivalence between $\ACA_{0}$ 
and the fact that any RM-closed set in the unit interval is located by \cite{withgusto}*{Theorem~3.8}, since $\open^{+}$ reduces open sets to RM-open sets.
Indeed, an RM-code for $Z$ as in $\open^{+}$ is given by $\cup_{n\in \N}(a_{n}, b_{n})$, where $a_{n}=b_{n}$ if $n\not\in  X$ and $(a_{n}, b_{n})=B(q_{n}, r_{n})$ otherwise. 
The implication $\eqref{b1asefq}\di \eqref{basef}$ is immediate as either $1$ is the supremum of $C$, or $1-d(1, C)$ is, where the locatedness of $C$ begets the distance function $d$.

\smallskip

For the implication \eqref{basef}$\di$\eqref{mct}, fix an increasing net $x_{d}:D\di [0,1]$.  
In case this net comes arbitrarily close to $1$, we are done.  If not, define the \emph{non-empty} closed set $C$ by putting $x\in C$ if and only if $(\forall d\in D)(x\geq_{\R} x_{d})$ for $x\in [0,1]$.
Indeed, the complement of $C$ is open in $[0,1]$, as it is given by $\cup_{d\in D}[0, x_{d})$.
Since $C$ is closed, $\sup C$ exists and the latter real is readily seen to be the limit of the net $x_{d}$.    
Note that $C$ is not \emph{exactly} as in Definition \ref{opensset}, but this does not matter: in case $\neg(\exists^{2})$, the implication \eqref{basef}$\di$\eqref{mct} reduces to a known second-order result; in case $(\exists^{2})$, we can use $\exists^{2}$ to freely convert between reals and elements of Cantor and Baire space, modifying $C$ to conform to Definition \ref{opensset}. 
The implication $\eqref{mct}\di \eqref{tropen}$ is immediate by Theorem \ref{bongra}. 

\smallskip

We now prove the `crux' implication $\BOOT\di \open^{+}$.  In case $\neg(\exists^{2})$, all functionals on $\R$ or $\N^{\N}$ are continuous by \cite{kohlenbach2}*{\S3}.
Thus, an open set $\cup_{y\in \R}I_{y}^{\psi}$ reduces to the \emph{countable} union $\cup_{q\in \Q}I_{q}^{\psi}$, yielding $\open^{+}$ in this case.     
In case $(\exists^{2})$, let $O$ be an open set given by $\psi:\R\di \R^{2}$ as in Definition \ref{opensset}.  
Now use $\BOOT$ (and $(\exists^{2})$) to define the following set $X\subset \N\times \Q$:
\be\label{flim}\textstyle
(\forall n\in \N, q\in \Q)\big( (n,q)\in X\asa (\exists y\in \R)\big(B(q, \frac{1}{2^{n}})\subset I_{y}^{\psi} \big)  \big).
\ee
Apply $\QFAC^{0, 1}$ to the forward direction in \eqref{flim} to obtain $\Phi$ such that:
\be\label{flim2}\textstyle
(\forall n\in \N, q\in \Q)\big( (n,q)\in X\di \big(B(q, \frac{1}{2^{n}})\subset I_{\Phi(n, q)}^{\psi} \big)  \big).
\ee
The following formula \eqref{keind} provides a representation of $O$ as a countable union of open balls, and of course gives rise to $\open^{+}$:
\be\label{keind}
x\in O\asa (\exists n\in \N, q\in \Q)( (n, q)\in X\wedge x\in I_{\Phi(n, q)}^{\psi}).
\ee 
For the reverse implication in \eqref{keind}, $x\in O$ follows by definition from the right-hand side of \eqref{keind}. 
For the forward implication, $x_{0}\in O$ implies $B(x_{0},\frac{1}{2^{n_{0}}} )\subset I_{y_{0}}^{\psi}$ for some $y_{0}\in \R$ and $n_{0}\in \N$ by definition.
For $n_{1}$ large enough, the rational $q_{0}:=[x_{0}](n_{1})$ is inside $B(x_{0},\frac{1}{2^{n_{0}+1}} )$.   
Hence, $(q_{0}, n_{0}+1)\in X$ by \eqref{flim} for $y=y_{0}$.  Applying \eqref{flim2} then yields $B(q_{0}, \frac{1}{2^{n_{0}+1}})\subset I_{\Phi(n_{0}+1, q_{0})}^{\psi}  $.
By assumption, we also have $x_{0}\in B(q_{0}, \frac{1}{2^{n_{0}+1}})\subset I_{\Phi(n_{0}+1, q_{0})}^{\psi}  $, and the right-hand side of \eqref{keind} follows. 

\smallskip

What remains to prove is $\eqref{kopen}\di \eqref{baseflpr}\di \eqref{baseflp}\di \eqref{mct}$.  
The implication $\eqref{kopen}\di \eqref{baseflpr}$ follows as in the first paragraph of this proof.  Indeed, $\ACA_{0}$ is equivalent to item \eqref{baseflpr} for RM-closed sets by \cite{simpson2}*{IV.2.11} and $\open^{+}$ converts closed sets into RM-closed sets. 
Clearly, $\eqref{baseflpr}\di \eqref{baseflp}$ is trivial, while $\eqref{baseflp}\di \eqref{mct}$ follows as in the second paragraph of this proof for the net $x_{d}$. 
Indeed, consider the closed set defined by $x\in C$ if and only if $(\forall d\in D)(x\geq_{\R} x_{d})$ and the function $f(x):=-x+1$.  
The real $\sup_{x\in C}f(x)$ readily provides the limit of the net $x_{d}$, and we are done. 
\end{proof}
It should be noted that $\eqref{basef}\di \eqref{mct}$ in the proof is proved based on the proof of \cite{withgusto}*{Theorem 3.8}, but with sequences replaced by nets (indexed by $\N^{\N}$). 
Moreover, in light of the previous proof, we could restrict items \eqref{baseflpr} and \eqref{baseflp} to RM-continuous functions (or other notions).
Since $\ECF$ converts $\open^{+}$ to a triviality, we do need $\ACA_{0}$ in item \eqref{kopen}.
Moreover, it seems that $\QFAC^{0,1}$ is essential in the previous theorem, but a reversal is not possible: $\Z_{2}^{\Omega}$ proves $\open^{+}$ by \cite{dagsamVII}*{Thm 3.22}.

\smallskip

Finally, a \emph{separably RM-closed} set $\overline{S}$ in a metric space is given in RM by a sequence $\lambda n.x_{n}$ and `$x\in \overline{S}$' is then $(\forall k^{0})(\exists n^{0})(d(x,x_{n})<\frac{1}{2^{k}})$, 
where $d$ is the metric of the space.  Intuitively, a separably RM-closed set is represented by a countable dense subset given by a sequence.   
We shall study this concept for \emph{sequences replaced by nets} as in Definition \ref{dirf}.  

\smallskip

What follows is not just \emph{spielerei} for the following reason:
 it is well-known that $\ZF$ cannot prove that `$\R$ is a sequential space', i.e.\ the equivalence between the definition of closed and sequentially closed set; countable choice however suffices (see \cite{heerlijkheid}*{p.\ 73}).   
On the other hand, we can avoid the Axiom of Choice by replacing sequences with nets everywhere, as shown in \cite{samnetspilot}*{\S4.4}.  
In this light, the following definition make sense.   
\bdefi\label{dirf}
A separably closed set $\overline{S}$ in $\R$ is given by a net $x_{d}:D\di \Q$ with $D\subseteq\N^{\N}$ and where $x\in \overline{S}$ is given by  $(\forall k^{0})(\exists d\in D)(|x-x_{d}|<\frac{1}{2^{k}})$.
\edefi
\begin{princ}[$\CLO$]
A separably closed set in $\R$ is closed.  
\end{princ}
Note that $\ACA_{0}$ is equivalent to the RM-version of $\CLO$ by \cite{browner2}*{Theorem 2.9}.  
\begin{thm}\label{yielddouble}
The system $\RCAo+\QFAC^{0,1}$ proves $\CLO\asa \BOOT$.
\end{thm}
\begin{proof}
The forward direction is immediate by the proof of \cite{samrecount}*{Theorem 3.19}, in light of Theorem \ref{DELTA}.
For the reverse direction, in case $\neg(\exists^{2})$, the implication reduces to the known second-order result, following Remark \ref{unbeliever}.
In case $(\exists^{2})$, let $\overline{S}$ and $x_{d}$ be as in $\CLO$.  Now use $(\exists^{2})$ and $\BOOT$ to obtain a set $X\subset \Q$ such that $(\forall q\in \Q)(q\in X\asa q\in \overline{S})$.  
By definition, for any $x\in \R$, we have $x\not\in \overline{S}\asa (\exists k\in \N)([x](k)\not \in \overline{S})$, and the latter is decidable thanks to $\exists^{2}$ and the aforementioned set $X$.
Following this observation,  for any $x\not \in\overline{S}$, we can find $k_{0}\in \N$ using Feferman's $\mu$ such that $B(x, \frac{1}{2^{k_{0}}})$ does not intersect $\overline{S}$.  
Thus, the complement of $\overline{S}$ is an open set as in Definition \ref{opensset}, and we are done. 
\end{proof}
As it happens, the converse of $\CLO$ for RM-codes is equivalent to $\FIVE$ by \cite{browner2}*{Theorem 2.18}, and we study systems at the level of the latter in Section \ref{blef}. 
We finish this section with a conceptual remark regarding the above results. 
\begin{rem}[The power of nets]\label{memmen2}\rm
As noted in Remark \ref{memmen}, nets with \emph{countable} index sets do not yield a stronger monotone convergence theorem, while uncountable index sets like $\N^{\N}$ of course do, by the above. 
Thus, `larger' index sets would seem to yield stronger versions of the monotone convergence theorem.  Moreover, the latter seems intrinsically tied to arithmetical comprehension, as $\ECF$ translates $\BOOT$ to $\ACA_{0}$.  
Both of the aforementioned suggestions are incorrect as follows: one can show that $\MCT_{\net}^{-}$, i.e.\ the monotone convergence theorem for nets in $[0,1]$ indexed by $2^{\N}$, is provable from the existence of the intuitionistic fan functional as follows:
\be\tag{$\MUC$}
(\exists \Omega^{3})(\forall Y^{2}, f, g\in 2^{\N})(\overline{f}\Omega(Y)=\overline{g}\Omega(Y)\di Y(f)=Y(g)).
\ee
Hence, $\MCT_{\net}^{-}$ has the same first-order strength as $\WKL_{0}$, as $\ECF$ converts $\MUC$ into $\WKL$ by \cite{longmann}*{p.~497}.
Moreover, the same holds for the items from Theorem~\ref{yield} for open sets represented by $\cup_{y\in [0,1]}I_{y}^{\Psi}$ and $\Psi:\R\di \Q^{2}$, and for many theorems pertaining to nets from \cite{samnetspilot}; this is a sizable contribution to \emph{Hilbert's program} as in \cite{simpson2}*{IX.3.18}.
Moreover, over $\RCAo$, we have $[\ACA_{0}+\MCT_{\net}^{-}]\asa \MCT_{\net}^{C}$, while $\MUC\di \MCT_{\net}^{-}$ is also provable using intuitionistic logic, i.e.\ convergence 
theorems for nets are not necessarily non-constructive, but can be (at least) intuitionistic.  
In conclusion, the \emph{structure} of the index set matters as much as its size, and these results should be contrasted with \cite{simpson2}*{V.5.8}.
\end{rem}

\subsubsection{At the level of $\ATR_{0}$ and $\FIVE$}\label{blef}
We study theorems pertaining to perfect sets based on our notion of open set from Definition \ref{opensset}.  
This will give rise to the counterparts of $\ATR_{0}$ and $\FIVE$ in the Plato hierarchy.  

\smallskip

First of all, the Cantor-Bendixson theorem for RM-closed sets is equivalent to $\FIVE$ by \cite{simpson2}*{VI.1.6}.  We study this theorem for closed sets as in Definition \ref{opensset}.
\begin{princ}[$\CBT$]
For any closed set $C\subseteq [0,1]$, there exist $P, S\subset C$ such that $C=P\cup S$, $P$ is perfect and closed, and $S^{0\di 1}$ is a sequence of reals. 
\end{princ}
To be absolutely clear, the countable set $S$ is given as a sequence of real numbers $S^{0\di 1}$, just like in second-order RM.  
We also study the following variation of $\CBT$ involving the `usual' definition of countable set, i.e.\ the existence of an injective function from the set to $\N$.
\begin{princ}[$\CBT'$]
For closed $C\subseteq [0,1]$, there is $P, S\subseteq C$ such that $C=P\cup S$, $P$ is perfect and closed, and $S$ is a countable set of points of $C$.
\end{princ}
\noindent
By Theorem \ref{keslich}, the exact notion of countable set in $\CBT$ does not matter. 

\smallskip

On one hand, theorems like e.g.\ item \eqref{b1asefq} from Theorem \ref{yield} only mention closed sets in the outermost universal quantifier, i.e.\ we are dealing with a straightforward
generalisation of the associated second-order theorem. 
On the other hand, the Cantor-Bendixson theorem as in $\CBT$ additionally states the existence of a (perfect) closed set, i.e.\ it is not clear whether $\CBT$ is in fact a generalisation of 
the second-order version in the absence of $\open^{+}$.
Nonetheless, we have the following splitting.  
\begin{thm}\label{keslich}
Over $\RCAo+\QFAC^{0,1}$, we have $[\ACA_{0}+\CBT]\asa [\ACA_{0}+\CBT']\asa [\FIVE+\BOOT]$.
\end{thm}
\begin{proof}
Recall that $\FIVE$ is equivalent to the second-order version of $\CBT$ over $\ACA_{0}$ (\cite{simpson2}*{VI.1.6}).  
Hence, the second reverse implication is immediate from $\open^{+}$ provided by Theorem \ref{yield}.
Moreover, it suffices to prove $\CBT'\di \BOOT$ for the forward implications as $\CBT\di \CBT'$.  
Since the implication reduces to the second-order result in case $\neg(\exists^{2})$, we may assume $(\exists^{2})$.
Fix $Y^{2}$ and consider the following:  the formula $(\exists f^{1})(Y(f, n)=0)$ is equivalent to $(\exists X \subset \N^{2})(Y(F(X), n)=0)$, where $F(X)(n):=(\mu m)((n,m)\in X)$.
Hence, $(\exists f^{1})(Y(f, n)=0)$ is equivalent to a formula $(\exists f\in 2^{\N})(\tilde{Y}(f, n)=0)$, where $\tilde{Y}$ is defined explicitly in terms of $Y$ and $\exists^{2}$. 
Now define the functional $Z:\R\di \R$ as:  
\be\label{honker2}
Z(x):=
\begin{cases}
\qquad\emptyset & \text{if} \begin{array}{c} n<_{\R}|x|\leq_{\R} n+1~ \wedge \\ \tilde{Y}(\eta(x)(0),n)\times \tilde{Y}(\eta(x)(1),n)=0\end{array}\\
\begin{array}{c}(n, n+\frac{1}{2})~\cup  \\
 ( n+\frac{1}{2}, n+1) \end{array} & \textup{otherwise}
\end{cases},
\ee
where $\eta(x)$ provides a pair consisting of the binary expansions of $x-\lfloor x \rfloor$; the pair consists of identical elements if there is a unique such expansion.  
Note that $\exists^{2}$ can define such functionals $Z$ and $\eta^{1\di (1\times 1)}$.
One readily converts $Z$ into an open set $O$ as in Definition \ref{opensset}.  Let $C=P\cup S$ be the complement of $O$, where $P, S$ are provided by $\CBT'$, i.e.\ $S$ is just a countable set of points. 
Then for all $n\in \N$: 
\be\label{frong2}\textstyle
(\exists f^{1})(Y(f, n)=0)\asa[ (n+\frac12)\in P], 
\ee
and note that $P$ is a closed set and hence `$x\in P$' has the form `$(\forall y\in \R)A(x, y)$' for arithmetical $A(x, y)$ by Definition \ref{opensset}.  Hence, $\BOOT$ follows from \eqref{frong2} as $\Delta$-comprehension is available by Theorem \ref{DELTA}
\end{proof}
The attentive reader has noted that the open set $O$ defined by $Z$ in the previous proof is actually a countable union of intervals \emph{in the usual sense of `countable' from mainstream mathematics}.  
We discuss this point in Remark \ref{cuntable}.
We also note that \eqref{frong2} only holds because $P$ is the \emph{largest} perfect subset of $C$, i.e.\ it would not necessarily work for other perfect subsets. 

\smallskip

Next, we formulate another variation of $\CBT$ involving a characteristic function for the countable set.
\begin{princ}[$\CBT''$]
For any closed set $C\subseteq [0,1]$, there exist $P, S\subset C$ such that $C=P\cup S$, $P$ is perfect and closed, and there is a characteristic function for the countable set $S$ of points of $C$.
\end{princ}
We have the following nice equivalence. 
\begin{cor}
Over $\RCAo+\QFAC^{0,1}$, $[\ACA_{0}+\CBT'']\asa [(\exists^{2})+\BOOT]$.
\end{cor}
\begin{proof}
In the light of the proof of the theorem and the fact that $[\BOOT+(\exists^{2})]\di \FIVE$, we only need to prove $\CBT''\di (\exists^{2})$, which is immediate by \cite{kohlenbach2}*{\S3}.  
\end{proof}
Next, we study the converse of $\CLO$ from the previous section, as follows.
\begin{princ}[$\OLC$]
A closed set in $\R$ is separably closed.  
\end{princ}
Note that the RM-version of $\OLC$ is equivalent to $\FIVE$ over $\RCA_{0}$ by (the proof of) \cite{browner2}*{Theorem 2.18}.
The same caveats as for $\CBT$ apply to $\OLC$, and we have the following splitting.  
\begin{cor}
Over $\RCAo+\QFAC^{0,1}$, we have $\OLC\asa [\BOOT+\FIVE]$.
\end{cor}
\begin{proof}
The reverse direction is immediate from the known second-order results, in light of Theorem \ref{yield} and \ref{yielddouble}, and the fact that sequences are nets. 
For the forward direction, consider the closed set $C$ from the proof of Theorem \ref{keslich}.  Then $\OLC$ provides a net $x_{d}$ generating a set $\overline{S}$ equalling $C$.  
Note that we have for all $n^{0}$:
\begin{align*}\textstyle
(\exists f^{1})(Y(f,n)=0)
&\textstyle\asa (\textup{$n+\frac{1}{2}$ is an isolated point of $C$})\\
&\textstyle\asa (\forall d\in D)(x_{d} \in (n, n+1)\di x_{d}=n+\frac{1}{2}),
\end{align*}
and $\Delta$-comprehension (together with $(\exists^{2})$ as usual) yields $\BOOT$.  Note that in case $\neg(\exists^{2})$, the implication reduces to the known second-order results. 
\end{proof}
We included the previous result as it gives rise to the following conceptual remark:  in spaces `larger' than $\R$, it is natural to define open sets given by uncountable unions indexed by $\N^{\N}\di \N$, while separably closed
sets are given by nets indexed by $\N^{\N}\di \N$.  The associated generalisations of $\CLO$ and $\OLC$ then imply $\BOOT^{1}$.  
To put it more bluntly, even if the reader does not share the author's sense of wonder about these results, \emph{that} the latter generalise to all finite 
types \emph{with little effort}, should at least come as a surprise.  

\smallskip

Next, we study the \emph{perfect set theorem} for closed sets as in Definition \ref{opensset}.  
This theorem for RM-codes is equivalent to $\ATR_{0}$ by \cite{simpson2}*{V.5.5 and VI.1.5}.  A subset $C$ of $\R$ is \emph{uncountable} if for every sequence of reals $\lambda n.x_{n}$, there is $y\in C$ such that $(\forall n\in \N)(x_{n}\ne_{\R}y)$; the same concept is used in RM, namely in \cite{simpson2}*{p.\ 193}. 
\begin{princ}[$\PST$]
For any closed and uncountable set $C\subseteq [0,1]$, there exist $P\subseteq C$ such that $P$ is perfect and closed.  
\end{princ}
The same caveats as for $\CBT$ apply to $\PST$, and we have the following splitting.  
\begin{thm}\label{dontlabel}
Over $\RCAo+\QFAC^{0,1}$, $[\ACA_{0}+\PST]\asa [\ATR_{0}+\BOOT]$.
\end{thm}
\begin{proof}
Recall that $\ATR_{0}$ is equivalent to the second-order version of $\PST$ over $\ACA_{0}$ by \cite{simpson2}*{V.5.5}.  
The reverse implication is immediate from $\open^{+}$ provided by Theorem \ref{yield}.
Moreover, it suffices to prove $\PST\di \BOOT$ for the forward implication.  
Since the implication reduces to $\ATR_{0}\di \ACA_{0}$ in case $\neg(\exists^{2})$, we may assume $(\exists^{2})$.
Fix $Y^{2}$ and define the following functional: 
\be\label{honker3}
Z(x):=
\begin{cases}
\begin{array}{c}(n, n+\frac{1}{2})~\cup  \\
 ( n+\frac{1}{2}, n+1) \end{array} & \text{if} \begin{array}{c} n<_{\R}|x|\leq_{\R} n+1 ~\wedge \\ \tilde{Y}(\eta(x)(0),n)\times \tilde{Y}(\eta(x)(1),n)=0\end{array}\\
\quad(n, n+1) & \textup{otherwise}
\end{cases}, 
\ee
One readily converts $Z$ into an open set $O$ as in Definition \ref{opensset}.  Let $C$ be the complement of $O$ and note the former only consists of isolated points, i.e.\ $C$ cannot have a perfect subset.
Hence, the contraposition of $\PST$ provides a sequence $\lambda n.x_{n}$ that includes all the elements of $C$.    
We now have, for all $n\in \N$, that 
\be\label{frong3}\textstyle
(\exists f^{1})(Y(f, n)=0)\asa (\exists m\in \N)(n+\frac{1}{2}=_{\R} x_{m} \wedge x_{m}\in C).
\ee
Given $\QFAC^{0,1}$, a formula of the form $(\exists m^{0})(\forall f^{1})(Y(f, n)=0)$ is equivalent to $(\forall g^{(0\times 0)\di 1})(\exists m^{0})(Y(\lambda m.g(n,m), n)=0)$.
Since $(\exists^{2})$ is given and since the right-hand side of \eqref{frong3} has the aforementioned form, we observe that $\Delta$-comprehension applies to the latter, and 
and $\BOOT$ follows. 
\end{proof}
By Remark \ref{memmen2}, open sets represented by $\cup_{x\in [0,1]}I_{x}^{\Psi}$ have a lot more `constructive' properties than open sets represented by $\cup_{x\in \R}I_{x}^{\Psi}$.  
In fact, one readily shows that $\MUC$ implies $\CBT$ and $\PST$ formulated using the former notion of open set indexed by the unit interval.  As noted in Remark \ref{memmen2}, 
this means that these theorems have the same first-order strength as $\WKL_{0}$.

\smallskip

Inspired by the previous, $\ATR_{0}$ and $\FIVE$ now boast higher-order counterparts. 
\bdefi[$\BOOT_{2}$]\label{Xz} For $Y^{2}$ such that $\lambda g^{1}.Y(f, g, n)$ is continuous for all $f^{1}, n^{0}$, we have $(\exists X^{1})(\forall n^{0})(n\in X\asa (\exists f^{1})(\forall g^{1})(Y(f,g,n)=0))$.
\edefi
\noindent It is straightforward to show that $\BOOT_{2}\asa [\BOOT+\FIVE]$ over $\RCAo$, which combines nicely with Theorem \ref{keslich} and similar equivalences. 
\bdefi[$\STR$]\label{Xw} For $\theta(n, g)\equiv (\exists f^{1})(Z(f, g, n)=0)$ where $\lambda g^{1}.Z(f, g, n)$ is continuous for any $f^{1}, n^{0}$, we have:
\[
(\forall X^{1})(\WO(X)\di (\exists Y^{1})H_{\theta}(X, Y)).
\]
\edefi
It is straightforward to show that the $\ECF$-translation of $\STR$ is $\ATR_{0}$, while $[\ATR_{0}+\BOOT]\asa \STR$ is immediate, which combines nicely with Theorem \ref{dontlabel}.
A related result is mentioned below Theorem \ref{PI}.

\smallskip

Moreover, let $\textsf{\textup{T}}$-$\SEP$ be the usual separation schema (see e.g.\ \cite{simpson2}*{I.11.7}) for formulas $\varphi_{i}(n)\equiv (\exists f_{i}^{1})(\forall g_{i}^{1})(Y_{i}(f_{i},g_{i}, n)=0)$.
Imitating the proof that $\ATR_{0}$ follows from $\Sigma_{1}^{1}$-separation in \cite{simpson2}*{V.5.1}, one readily obtains $\textsf{\textup{T}}$-$\SEP\di \STR$.
The crucial part is that given countable choice as in $\QFAC^{0,1}$, $(\exists Y^{1})H_{\theta}(X, Y)$ has the same form as the $\varphi_{i}$ in $\textsf{\textup{T}}$-$\SEP$.
Restricting to a continuous parameter $g_{i}$ seems essential for a reversal. 

\smallskip

Finally, with the gift of hindsight, we can now generalise Definition \ref{opensset} and Theorem~\ref{yield} \emph{to any higher type}. 
By way of an example, one can consider nets indexed by subsets of $\N^{\N}\di \N$, while the quantifier `$(\exists y\in \R)$' in the definition of open sets
is similarly `bumped up one type', namely from ranging over $\R$ to $\R\di \R$.  The associated comprehension axiom is of course $\BOOT^{1}$.  
The equivalences in the above theorems then go through over a suitable base theory.  
We leave the details to be worked out.  We finished this section with an important conceptual remark.  
\begin{rem}[A cardinality by any other name]\label{cuntable}\rm
Let us begin by recalling that if $\cup_{n\in \N}(a_{n}, b_{n})$ is a countable union of basic open balls, then so is $\cup_{f\in \N^{\N}}(a_{Y(f)}, b_{Y(f)})$ for any $Y^{2}$ and using the mainstream definition of `countable set'.
Now note that the open set $O$ defined by $Z$ in \eqref{honker2} can be expressed in the latter form, i.e.\ it is also a countable union of basic open balls.  
Thus, all the results in this section also hold for $\CBT$ restricted to open sets given by countable unions, i.e.\ the generalisation to uncountable unions is (technically) superfluous.  

\smallskip

For the above reason, countable unions from RM like $\cup_{n\in \N}(a_{n}, b_{n})$ should be referred to as `sequential' or `searchable' or a similar term that captures the fact that we are dealing with a sequence that one can search through `one by one' in a weak system. 
By contrast, the countable union $\cup_{f\in \N^{\N}}(a_{Y(f)}, b_{Y(f)})$ is not searchable in any reasonably sense.  
\emph{In conclusion}, the lack of structure of $O$ defined by \eqref{honker2} is what gives rise to the strength of $\CBT$, \emph{not} the cardinality of the index set.    
More palatable examples based on countable fields can be found in \cite{samFLO2, samrecount}.
\end{rem}
We recall that a similar situation for nets exists, as discussed in Remark \ref{memmen2}.  Moreover, defining `$w\approx_{D}v$' as $c_{w}=_{\R}c_{v}$ in the proof of Theorem \ref{proofofconcept}, we observe that   
the index set $D$ only involves countably many equivalence classes modulo $\approx_{D}$.  In this sense, the index set $D$ of $c_{w}$ is also countable.  

\subsection{Heine-Borel compactness}\label{plato}
In this section, we connect $\BOOT$ to $\HBU$ and other higher-order axioms as in Figure \ref{kk}.  

\smallskip

We first show that $\HBU$ follows from $\BOOT$, in contrast to the known comprehension axioms of third-order arithmetic provided by $\SIXK$.
\begin{thm}\label{honor}
The system $\RCAo+\IND$ or $\RCAo+\QFAC^{0,1}$ proves $\BOOT\di \HBU$ while $\Z_{2}^{\omega}+\QFAC^{0,1}$ does not prove $\BOOT$ or $\HBU$. 
\end{thm}
\begin{proof}
The first negative result follows directly from Theorem \ref{boef}, while  $\Z_{2}^{\omega}+\QFAC^{0,1}\not\vdash \HBU$ has been established in \cites{dagsamIII, dagsamV}.  
For the positive result, we prove $\HBU_{\c}$, i.e.\ the Heine-Borel compactness of Cantor space, as follows
\be\tag{$\HBU_{\c}$}
(\forall G^{2})(\exists  f_{1}, \dots, f_{k} \in C ){(\forall f^{1}\in C)}(\exists i\leq k)(f\in [\overline{f_{i}}G(f_{i})]).
\ee
Note that $\HBU\asa \HBU_{\c}$ over $\RCAo$ by the proof of \cite{dagsamIII}*{Theorem 3.3}.
Fix $G^{2}$ and let $A(\sigma)$ be the following formula
\be\label{bongka}
(\exists g\in C)\big[G(g)\leq |\sigma| \wedge \sigma*00\dots \in [\overline{g}G(g)]\big],
\ee
where $\sigma^{0^{*}}$ is a finite sequence of natural numbers.  Note that the formula in \eqref{bongka} in square brackets is quantifier-free.  Thus, $\BOOT$ provides a set $X\subseteq \N$ such that $(\forall \sigma^{0^{*}})(\sigma  \in X\asa A(\sigma) )$, with minimal coding.  Now, we have $(\forall f\in C)(\exists n^{0})A(\overline{f}n)$ since we may take $g=f$ and $n=G(f)$.  Hence, we have $(\forall f\in C)(\exists n^{0})(\overline{f}n\in X)$ and applying $\QFAC^{1,0}$, there is $H^{2}$ such that $(\forall f\in C)(\overline{f}H(f)\in X)$ and $H(f)$ is the least such number.  Obviously $H^{2}$ is continuous on $C$ and hence bounded above on $C$ by \cite{kohlenbach4}*{\S4}.  Hence, there is $N_{0}^{0}$ such that $(\forall f\in C)(\exists n\leq N_{0})A(\overline{f}n)$.
Let $\sigma_1 , \dots ,\sigma_{2^{N_0 }+ 1}$ enumerate all binary sequences of length $N_0 + 1$ and define  $f_i := \sigma_i*00\dots$ for $i\leq 2^{N_{0}+1}$.
Intuitively speaking, we now apply \eqref{bongka} for $f_i$ and obtain $g_i$ for each $i\leq 2^{N_{0}+1}$.  Then $\langle g_1 , \ldots , g_{2^{N_0 + 1}}\rangle$ provides the finite sub-cover for $G$.
Formally, it is well-known that $\ZF$ proves the `finite' axiom of choice via mathematical induction (see e.g.\ \cite{tournedous}*{Ch.\ IV}).  Similarly, one readily uses $\textsf{IND}$ to prove the existence of the
aforementioned finite sequence based on \eqref{bongka}.
We can replace $\IND$ by $\QFAC^{0,1}$, which is applied to \eqref{bongka} to yield the finite sub-cover.
\end{proof}
The final part of the proof was first used in \cite{sahotop} to prove \emph{without using the Axiom of Choice} the equivalence between $\HBU$ and a version involving more general covers.  
Note that $\BOOT\di \HBU$ becomes $\ACA_{0}\di \WKL_{0}$ when applying $\ECF$.  

\smallskip

Secondly, $\WKL_{0}$ is equivalent to the separation axiom $\Sigma_{1}^{0}$-$\SEP$, i.e.\ the schema \eqref{SIG} for $\L_{2}$-formulas $\varphi_{i}\in \Sigma_{1}^{0}$, by \cite{simpson2}*{IV.4.4}.  
We consider the separation axiom $\Sigma$-$\SEP$ and note that $\HBU\di \Sigma$-$\SEP$ becomes $\WKL_{0}\di\Sigma_{^{1}}^{0}$-$\SEP$ under $\ECF$.
\bdefi[$\Sigma$-\SEP]
For $\varphi_{i}(n)\equiv (\exists f_{i}^{1})(Y_{i}(f_{i}, n)=0)$, we have 
\be\label{SIG}
(\forall n^{0})(\neg\varphi_{1}(n)\vee \neg\varphi_{2}(n))\di (\exists Z^{1})(\forall n^{0})\big[\varphi_{1}(n)\di n\in Z\wedge \varphi_{2}(n)\di n\not\in Z\big].
\ee
\edefi
\begin{thm}\label{klato}
The system $\RCAo+\IND+\QFAC^{1,1}$ proves $\HBU\di\Sigma$-$\SEP$. 
\end{thm}
\begin{proof}
Suppose $\varphi_{i}$ is as in $\Sigma$-$\SEP$ and satisfies the antecedent of \eqref{SIG}.  
Note that using $\IND$, it is straightforward to prove that for every $m^{0}$, there is a finite binary sequence $\sigma^{0^{*}}$ such that $|\sigma|=m$ and 
\be\label{pbc}
(\forall n< m)\big[\varphi_{1}(n)\di (\sigma(n)=1)\wedge \varphi_{2}(n)\di( \sigma(n)=0)\big].
\ee
Now let $A(n, Z)$ be the formula in square brackets in \eqref{SIG} and suppose we have $(\forall Z^{1})(\exists n^{0})\neg A(n, Z)$.  
Note that $\neg A(n,Z)$ hides two existential quantifiers involving $f_{1}, f_{2}$.  Applying $\QFAC^{1,1}$, we obtain $G:C\di \N$ such that  $(\forall Z^{1})(\exists n\leq G(Z))\neg A(n, Z)$.
Apply $\HBU_{\c}$ to the canonical cover $\cup_{f\in C}[\overline{f}G(f)]$ and obtain a finite sub-cover $f_{0}, \dots, f_{k}$, i.e.\ $\cup_{i\leq k}[\overline{f_{i}}G(f_{i})]$ also covers $C$. 
Let $k_{0}$ be $\max_{i\leq k} G(f_{i})$ and consider binary $\sigma_{0}$ of length $k_{0}+2$ satisfying \eqref{pbc}.
Then $g_{0}:= \sigma_{0}*00\dots$ is in some neighbourhood of the finite sub-cover, say $g_{0}\in [\overline{f_{j}}G(f_{j})]$.  
By definition, $k_{0}\geq G(f_{j})$, i.e.\ $\overline{g_{0}}G(f_{j})=\overline{\sigma_{0}}G_(f_{j})=\overline{f_{j}}G(f_{j})$.  However, \eqref{pbc} is false for $m=G(f_{j})$ and $\sigma=\overline{f_{j}}G(f_{j})$, a contradiction. 
\end{proof}
The usual `interval halving' proof (going back to Cousin in \cite{cousin1}) establishes the reversal, also using countable choice, in the theorem.  
We have the following corollary, variations of which are published in \cites{dagsam, dagsamII, dagsamIII}, all involving different proofs.  
\begin{cor}\label{shortie}
The system $\ACAo+\IND+\QFAC^{1,1}+\HBU$ proves $\ATR_{0}$.
\end{cor}
\begin{proof}
The schema \eqref{SIG} for $\L_{2}$-formulas $\varphi_{i}\in \Sigma_{1}^{1}$ is called \emph{$\Sigma_{1}^{1}$-separation} and equivalent to $\ATR_{0}$ by \cite{simpson2}*{V.5.1}. 
This separation axiom immediately follows from $(\exists^{2})$ and $\Sigma$-$\SEP$, and hence the theorem finishes the proof. 
\end{proof}
Thirdly, there is a straightforward generalisation of $\WKL$, equivalent to $\HBU$. 
\begin{rem}[Uniform theorems]\label{flurki}\rm
Dag Normann and the author study the RM and computability theory of \emph{uniform theorems} in \cite{dagsamV}.  
A theorem is \emph{uniform} if the objects claimed to exist by the theorem depend on few of its parameters.  
For instance, the contraposition of $\WKL_{0}$, aka the \emph{fan theorem}, expresses that a binary tree with no paths must be finite.  
It is readily seen that the latter is equivalent to the following sentence with the underlined quantifiers swapped: 
\be\tag{$\WKL_{\u}$}
(\forall G^{2})\underline{(\exists m^{0})(\forall T\leq_{1}1)}\big[(\forall \alpha \in C)(\overline{\alpha}G(\alpha)\not\in T)\di (\forall \beta \in C)( \overline{\beta}m\not\in T )     \big].
\ee
Note that $\WKL_{\u}$ expresses that a binary tree $T$ is finite if it has no paths, \emph{and} the upper bound $m$ only depends on a realiser $G$ of `$T$ has no paths'.  
For this reason, $\WKL_{\u}$ is called \emph{uniform} weak K\"onig's lemma.  
It is easy to show that $\WKL_{\u}\asa \HBU$ by adapting the proof of \cite{dagsam}*{Theorem 4.6}.  It goes without saying that most theorems from the RM of $\WKL_{0}$ have uniform versions that are equivalent to $\HBU$.
For instance, uniform versions of the \emph{Pincherle, Heine, and Fej\`er theorems} are studied in \cite{dagsamV}.  Moreover, as documented in \cite{dagsamV}*{Appendix A}, many proofs from the literature actually establish
the uniform version of the theorem at hand, including the first proof of Heine's theorem in Stillwell's introduction to RM (\cite{stillebron}).  
Finally, the original \emph{K\"onig's lemma} (see e.g.\ \cite{simpson2}*{III.7}) can be given a similar `uniform' treatment, something worthy of future study.   
\end{rem}
Fourth, $\WKL$ is equivalent to the Cantor intersection theorem for RM codes of closed sets, even constructively (see e.g.\ \cite{ishi1}).
As a further litmus test for our notion of closed set, we show in Theorem \ref{sosimple} that the Cantor intersection theorem for closed sets is equivalent to $\HBU$. 
Note that $\ECF$ yields the original equivalence as $\QFAC^{1,1}$ is translated to a triviality.   
\bdefi[$\CIT$]
Let $C_{n}$ be a sequence of closed sets such that $\emptyset \ne C_{n+1}\subseteq C_{n}\subseteq [0,1]$.  Then $\cap_{n\in \N}C_{n}\ne \emptyset$. 
\edefi
Note that the contraposition of $\CIT$ is a version of the Heine-Borel theorem for countable covers consisting of open sets.
\begin{thm}\label{sosimple}
The system $\RCAo+\QFAC^{1,1}$ proves $\HBU\asa \CIT$.
\end{thm}
\begin{proof}
In case $\neg(\exists^{2})$, the usual second-order proofs go through.  Indeed, all functions on $\R$ are continuous by \cite{kohlenbach2}*{\S3} and $\HBU$ reduces to the Heine-Borel theorem 
for the unit interval and \emph{countable covers}, which is just $\WKL$ by \cite{simpson2}*{IV.1}.  Similarly, closed sets become RM-closed sets.
We shall now prove the equivalence assuming $(\exists^{2})$, and the law of excluded middle finishes the proof. 

\smallskip

For the reverse direction, fix $\Psi:I\di \R^{+}$ and apply $\QFAC^{1,0}$ to the formula $(\forall x\in I)(\exists n^{0})( |I_{x}^{\Psi} |>\frac{1}{2^{n+1}} )$, we obtain $\Phi:I\di \Q^{^{+}}$ such that 
a finite sub-cover of $\cup_{x\in I}I_{x}^{\Phi}$ is also a finite sub-cover of $\cup_{x\in I}I_{x}^{\Psi}$.  In other words, we may restrict $\HBU$ to functionals $I\di \Q^{^{+}}$.
Now suppose $\HBU$ is false for $\Psi:I\di \Q^{+}$ and define the open set $O_{n}$ as follows using $\exists^{2}$: $x\in O_{n}$ if and only if 
\[\textstyle
(\exists y\in \R)(x\in I\wedge x\in I_{y}^{\Psi}\wedge |I_{y}^{\Psi}|>_{\R} \frac{1}{2^{n}}).
\]  
One readily obtains a definition of $O_{n}$ as in Definition \ref{opensset}.
Note that $O_{n}\subseteq O_{n+1}$ and define the closed set $C_{n}$ as the complement of $O_{n}$.  
By our assumption $\neg\HBU$, $C_{n}\ne \emptyset$ for any $n$.  Applying $\CIT$, there is $x_{0}\in \cap_{n\in \N} C_{n}$.
However, since $x_{0}\in I_{x_{0}}^{\Psi}$, we have $x_{0}\in O_{n_{0}}$ in case $|I_{x_{0}}^{\Psi}|\geq \frac{1}{2^{n_{0}}}$, a contradiction.  Hence $\HBU$ must hold for $\Psi$, in this case, and the reverse direction is done.

\smallskip

For the forward direction, let $C_{n}$ be as in $\CIT$, i.e.\ $C_{n}$ is the complement of $O_{n}=\cup_{y\in \R}I_{y}^{\Psi_{n}}$ for some sequence $\Psi_{n}:(\R\times \N)\di \R$.  
Now suppose $(\forall x\in I)(\exists n^{0}){(x\not \in C_{n})}$, i.e.\ $(\forall x\in I)(\exists n^{0}, y\in \R)(x\in I_{y}^{\Psi_{n}})$.
We may apply $\QFAC^{1,1}$ to obtain $\Phi$ such that $(\forall x\in I)(x \in I^{\Psi_{\Phi(x)(1)}}_{\Phi(x)(2)}) $.
Thus, $\cup_{x\in I}I^{\Psi_{\Phi(x)(1)}}_{\Phi(x)(2)}$ covers the unit interval and apply $\HBU$ to find a finite sub-cover, 
i.e.\ $y_{0}, \dots y_{k}\in I$ such that $[0,1]\subset\cup_{i\leq k}I^{\Psi_{\Phi(y_{i})(1)}}_{\Phi(y_{i})(2)}$.  
However, this implies $[0,1]\subset\cup_{i\leq k_{0}}O_{i}$ for $k_{0}:=\max_{i\leq k}\Phi(y_{i})(1)$ and $C_{k_{0}+1}$ must be empty, a contradiction.  Hence, $\HBU\di \CIT$ follows, and we are done. 
\end{proof}
The use of the axiom of choice in Theorems \ref{klato} and \ref{sosimple} is somewhat unsatisfactory.  
This shall be remedied in Section \ref{main3}.    

\smallskip

Sixth, we recall some results from \cite{samnetspilot, dagsamIII, dagsamVI} that complete Figure \ref{kk}.
\begin{rem}[Gauge integral]\label{woke}\rm
The gauge integral is a generalisation of the Lebesgue and (improper) Riemann integral (\cite{zwette}); it was introduced by Denjoy (in a different from) around 1912 
and studied by Lusin, Perron, Henstock, and Kurzweil.  The latter two pioneered the modern formulation of the gauge integral as the Riemann integral with the 
constant `delta' from the usual `epsilon-delta' definition replaced by a \emph{function}.  The gauge integral boasts the most general version of the \emph{fundamental theorem of calculus} and is `maximally' closed under improper integrals as in \emph{Hake's theorem} (see \cite{bartle, bartleol}).  The gauge integral also provides a unique and direct formalisation of Feyman's path integral (see \cite{mullingitover, pouly,nopouly,nopouly2,nopouly3}).  Many basic properties of the gauge integral, including the aforementioned theorems, are equivalent to $\HBU$, as shown in \cite{dagsamIII, dagsamIIIb}.
Applying $\ECF$ to these results, one obtains equivalences between $\WKL$ and theorems pertaining to the Riemann integral, as gauge integrals are just Riemann integrals if all functions are continuous.   
\end{rem}
\begin{rem}[Dini's theorem]\rm
Dini's theorem is equivalent to $\WKL$, as shown in \cite{diniberg, diniberg2}.  
Dini's theorem \emph{for nets} is verbatim the same theorem except for the replacement of `sequence' by `net'.  
Dini's theorem for nets is equivalent to $\HBU$, as shown in \cite{samnetspilot}*{\S3.2.1}. 
\end{rem}

Finally, $\ECF$ maps the Plato hierarchy from Figure \ref{kk} to the G\"odel hierarchy.  
Now, $\ECF$ replaces higher-order objects by RM codes, continuous by definition.  
For this reason, the existence of \emph{discontinuous} functions as in $(\exists^{2})$ is mapped to $0=1$ by $\ECF$.  
By contrast, $\ECF$ interprets $\BOOT$ and $\HBU$ as quite meaningful theorems.  
For this reason, it seemed obvious to us that $\BOOT$ and $\HBU$ should be equivalent to certain 
continuity axioms.   We explore this idea in the next section.   

\section{Main results III: continuity and neighourhoods}\label{main3}
\subsection{Introduction}
In this section, we provide a formulation of the Plato hierarchy based on \emph{continuity}.  
In particular, we show that $\BOOT$, $\HBU$, and related principles are equivalent to fragments of a certain continuity axiom schema stemming from intuitionistic analysis, called \emph{special bar/Brouwer continuity} $\textsf{SBC}$ in \cite{KT} and \emph{neighbourhood function principle} $\NFP$ in \cite{troeleke1}.  
The latter is classically true and connects axioms central to Brouwer's intuitionistic mathematics (see \cites{KT, troeleke1, keuzet}).  

\smallskip

Our results should be contrasted with Kohlenbach's approach from \cite{kohlenbach2} based on \emph{discontinuous} functions like $\exists^{2}$ and the Suslin functional. 
Indeed, $\ECF$ converts the existence of a discontinuous function like $\exists^{2}$ to $0=1$, while $\BOOT$ is converted to $\ACA_{0}$; in other words, it is to almost expected 
that $\BOOT$ has a formulation in terms of continuity.  
In this light, Kohlenbach approach yields a \emph{discontinuity} hierarchy, while the Plato hierarchy is a \emph{continuity} hierarchy and can be said to be a `return to Brouwer' in the aforementioned sense.  

\smallskip

A conceptual motivation for the study of $\NFP$ is provided by the very aim of RM itself, namely to find the minimal (set existence) axioms needed to prove theorems of ordinary mathematics.  Now, Heine-Borel compactness (and related principles like the Lindel\"of property) cannot be captured (well or at all) by higher-order comprehension.  
Indeed, one of the main results in \cite{dagsamIII, dagsamV} is that $\Z_{2}^{\omega}+\QFAC^{0,1}$ cannot prove $\HBU$ (and related principles like the Lindel\"of lemma), while $\Z_{2}^{\Omega}$ of course can;
the first-order strength of these systems is however massive compared to $\HBU$, i.e.\ anything remotely related to an equivalence is off the table.
By contrast, the continuity schema $\NFP$ will be seen to yield elegant equivalences.  

\smallskip

Finally, as part of this study, we suggest new axioms to be added to the base theory of higher-order RM, as discussed in Section \ref{nintro}.
One advantage is that these new axioms readily equip continuous functionals on Baire space with RM-codes, a topic studied by Kohlenbach in \cite{kohlenbach4}*{\S4}.
It should be noted that the base theory $\RCA_{0}$ contains weak comprehension axioms, i.e.\ it is only natural that the RM-study of $\NFP$ also requires weak fragments of the latter in the base theory. 
\subsection{New axioms and some motivation}\label{nintro}
We introduce the new axioms $\textsf{A}_{i}$ in Section~\ref{twaxioms}; we show in Section \ref{XxX} that these axioms yield many elegant equivalences, e.g.\ involving $\NFP$.  
In particular, these new axioms obviate the use of the Axiom of Choice in some of our above proofs.  
An overview of the arguments for the extension of $\RCAo$ with these axioms is found in Section \ref{baketheory}. 
\subsubsection{The new axioms $\A_{i}$}\label{twaxioms}
The development of RM starts with the definition of a good base theory.  So far, we have mostly used Kohlenbach's $\RCAo$ plus countable choice.  
Nonetheless, Theorems \ref{klato} and \ref{sosimple} seem to need more choice, namely $\QFAC^{1,1}$, and it is a natural question whether these results also go through in $\ZF$, or even 
a suitable weak extension of $\RCAo$ not involving (countable) choice.  

\smallskip

In this section, we formulate such a weak extension and show that it yields numerous elegant equivalences involving fragments of $\NFP$, including the promised `choice-free' improvement of Theorems \ref{klato} and \ref{sosimple}.   Other arguments in favour of our new axioms $\A_{i}$ are in Section~\ref{baketheory}.
We first introduce the axiom schema $\NFP$, which intuitively speaking expresses that if there \emph{could} be a continuous choice functional (the antecedent of \eqref{durgo}), then there is one given by an associate (the consequent of \eqref{durgo}).  
\bdefi[$\NFP$]
For any $A(\sigma^{0^{*}})$ in $\L_{\omega}$, we have
\be\label{durgo}
(\forall f^{1})(\exists n^{0})A(\overline{f}n)\di (\exists \gamma\in K_{0})(\forall f^{1})A(\overline{f}\gamma(f)), 
\ee
where `$\gamma\in K_{0}$' means that $\gamma^{1}$ is a total associate on Baire space.  
\edefi
\noindent 
The schema $\NFP$ was used in \cite{dagsamIII} to obtain the Lindel\"of lemma inside $\Z_{2}^{\Omega}+\QFAC^{0,1}$; 
it was also proved in \cites{samnetspilot} that $\NFP\di \MCT_{\net}^{[0,1]}\di \HBU$ over $\RCAo$.

\smallskip

Intuitively speaking, our new axioms $\textsf{A}_{i}$ shall be a generalisation of $\QFAC^{1,0}$ to the following formula classes.  
\bdefi[$C_{i}$-formulas]
\begin{itemize}
\item A $C_{0}$-formula $A$ has the following form: $A(n)\equiv(\exists f\in 2^{\N})(Y(f, n^{0})=0)$.  
\item A $C_{1}$-formula $A$ has the following form: $A(n)\equiv(\forall f\in 2^{\N})(Y(f, n^{0})=0)$.
\item A $C_{2}$-formula $A$ has the following form: \\ $A(n)\equiv(\exists f\in 2^{\N})(Y(f, n^{0})=0)\vee (\forall g\in 2^{\N})(Z(g, n^{0})=0)$.  
\end{itemize}
\edefi
\noindent
Note that $C_{i}$-formulas can have parameters besides the number variable.   
Our new axioms $\A_{i}$ are defined as the following fragments of $\NFP$.  Note that the choice functional in $\A_{i}$ need not be continuous, in contrast to $\NFP$.
\bdefi[$\textsf{A}_{i}$]
For any $C_{i}$-formula $A(\sigma^{0^{*}})$, we have
\[
(\forall f^{1})(\exists n^{0})A(\overline{f}n)\di (\exists \Phi^{2})(\forall f^{1})A(\overline{f}\Phi(f))
\]
\edefi
Besides its fruitful consequences listed below, there are good conceptual motivations for the previous axioms, as discussed in the next section.
One `trivial' argument is that (second-order) RM gauges the strength of theorems in terms of set existence axioms; to this end, the base theory $\RCA_{0}$ contains a 
weak set existence axiom.  Thus, if we are to develop RM based on $\NFP$, it stands to reason that our base theory should include \emph{some} fragment of $\NFP$.  
\subsubsection{Motivation for the $\A_{i}$ axioms}\label{baketheory}
We discuss some of the arguments in favour of a base theory that includes the new axioms $\A_{i}$.

\smallskip

First of all, the equivalence $[\BOOT]_{\ECF}\asa \ACA_{0}$ clearly suggests that one existential quantifier over $\N^{\N}$ in $\BOOT$ gives rise to a numerical quantifier under $\ECF$.  Hence, one existential quantifier over $2^{\N}$ should amount to (almost) the same as quantifier-free under $\ECF$, as also suggested by Theorem \ref{coref}.  
In this light, the axioms $\A_{i}$ yield an inconsequential extension of $\QFAC^{1,0}$, included in $\RCAo$.  

\smallskip

Secondly, $\HBU$ is formulated with a rather `effective' kind of cover, namely where each $x\in I$ is covered by $I_{x}^{\Psi}$ for $\Psi:I\di \R^{+}$, which is exactly the definition used by Cousin and Lindel\"of (\cites{cousin1,blindeloef}).  
A generalisation of $\HBU$ to (more) general covers, is studied in \cite{sahotop} as follows: the principle $\textsf{HBT}$ deals with covers in which only $(\forall x\in I)(\exists y\in I)(x\in I^{\psi}_{y})$ is assumed for $\psi:I\di \R$, i.e.\ $I_{x}^{\psi}$ can be empty.  One can prove $\HBU\asa \textsf{HBT}$ over $\RCAo+\IND+\A_{0}$ by \cite{sahotop}*{\S3.5}.  
Similar results hold for the Lindel\"of lemma and other basic topological theorems, i.e.\ $\A_{0}$ seems essential for an elegant development of the RM of topology. 

\smallskip

Thirdly, $\A_{1}$ readily implies the following `coding principle': any $Y^{2}$ that is continuous on $2^{\N}$, has a continuous modulus of continuity on $2^{\N}$, and hence an RM-code by \cite{kohlenbach4}*{Prop.\ 4.4}.  Indeed, consider $(\exists^{2})\vee \neg(\exists^{2})$ and note that in the former case, \cite{kohlenbach4}*{Prop.\ 4.4 and 4.7} provides the required modulus (and RM code).  
In the latter case, apply $\A_{1}$ to \eqref{shorty}, 
where the underlined formula is a $C_{1}$-formula:
\be\label{shorty}
(\forall f\in 2^{\N})(\exists N^{0})\underline{(\forall g\in 2^{\N})(\overline{f}N=\overline{g}N\di Y(\overline{f}N*00\dots)=Y(g))},
\ee
and note that the resulting $\Phi^{2}$ is continuous by \cite{kohlenbach2}*{\S3}.  
The study of the aforementioned coding principle in \cite{kohlenbach4}*{\S4} suggests that the RM of $\WKL$ does not change upon the 
replacement of continuous functions by RM-codes; we show in \cite{samex} that the RM of \emph{Tietze's extension theorem} and \emph{Ekeland's variational principle} (which involves $\WKL_{0}$) does greatly depend upon coding.

\smallskip

Fourth, Pincherle's theorem states that a locally bounded functional on $2^{\N}$ is bounded; consider the following version, called $\PIT_{o}'$ in \cite{dagsamV}: 
\[
(\forall F^{2}) \big[ {(\forall f \in C)(\exists n^{0})(\forall g\in C)\big[ g\in [\overline{f}n] \di F(g)\leq n    \big]  } \di (\exists m^{0})(\forall h \in C)(F(h)\leq m)  \big].
\]
It seems that the only way to prove $\HBU\di \PIT_{o}'$ is to apply $\A_{1}$ to the antecedent and apply $\HBU$ to the canonical cover associated to the resulting $\Phi^{2}$.
In general, moduli are an important part of constructive and computational approaches to mathematics, and $\A_{1}$ conveniently always seems to provide those.  

\smallskip

Fifth, recall $\Delta$-comprehension from Section \ref{liften}, which plays an important role in lifting proofs from second- to higher-order arithmetic.
Indeed, the recursive counterexample involving Specker sequences can be lifted to higher-order arithmetic by Theorem \ref{proofofconcept}, \emph{assuming $\Delta$-comprehension}; the latter 
plus $\WKL$ is also equivalent to the separation of ranges of non-overlapping type two functionals (see \cite{samrecount, samFLO2}). 
Theorem \ref{tirlydiddy} shows that $\RCAo+\IND+\A_{0}$ proves $\Delta$-comprehension. 

\smallskip

Sixth, Kohlenbach studies generalisations of $\WKL$ to certain formula classes in \cite{kohlenbach4}*{\S3}. 
Since $[\HBU]_{\ECF}$ is just $\WKL$, it is a natural question whether there is a generalisation of $\WKL$ that is equivalent
to $\HBU$.  By Corollary \ref{secondofmany}, $\A_{0}$ suffices to obtain an elegant such equivalence.  


\subsection{Some consequences of the $\A_{i}$ axioms}\label{XxX}
We use the new axioms $\A_{i}$ to obtain some elegant equivalences involving $\BOOT$, $\HBU$, and related principles on one hand, and fragments of $\NFP$ on the other hand.  

\smallskip

First of all, we introduce the new formula class `$\Sigma{\vee}\Pi$'.  
Now, the formula class `$\Sigma_{1}^{0}{\wedge}\Pi_{1}^{0}$' is used in RM (see \cite{simpson2}*{VI.5}) to study fragments of the axiom of determinacy from set theory.  
The formula class `$\Pi_{1}^{0}\vee \Sigma_{1}^{0}$' is mentioned in the title of \cite{araikarai}.
A formula of the form `$(\exists f^{1})(Y(f, n)=0)$' as in $\BOOT$ is called a `$\Sigma$-formula', while its negation is called a `$\Pi$-formula'.  
The formula class `$\Sigma{\vee}\Pi$' consists of disjunctions `$S\vee P$' with $S\in \Sigma$ and $P\in\Pi $.

\smallskip

Now let  $\Sigma{\vee} \Pi$-$\NFP$ be $\NFP$ restricted to $\Sigma{\vee}\Pi$-formulas and let $\NFP_{0}$ be $\NFP$ with `$(\exists \gamma\in K_{0})A(\overline{f}\gamma(f))$' replaced by `$(\exists \Phi^{2})A(\overline{f}\Phi(f))$', and the same for fragments.  The following theorem should be contrasted with the fact that comprehension 
does not capture $\HBU$ or $\BOOT$ well\footnote{The system $\Z_{2}^{\omega}+\QFAC^{0,1}$ cannot prove $\BOOT$ or $\HBU$, while $\Z_{2}^{\Omega}$ can (\cites{dagsamIII, dagsamV}).  However, the latter has the first-order strength of $\Z_{2}$, while $\RCAo+\BOOT$ is conservative over $\ACA_{0}$.} at all.   
\begin{thm}\label{PI}
$\RCAo+\IND$ proves $\Sigma{\vee} \Pi$-$\NFP\asa \BOOT\asa[\Sigma\vee \Pi$-$\NFP_{0}+\HBU]$. 
\end{thm}
\begin{proof}
A proof of $\BOOT\di \Sigma{\vee} \Pi$-$\NFP$ is as follows: $\BOOT$ replaces $\Sigma{\vee} \Pi$-formulas by equivalent quantifier-free ones.  
Then $\QFAC^{1,0}$ yields a (continuous) functional $G^{2}$ such that $G(f)$ is the least $n$ as in  $(\forall f^{1})(\exists n^{0})A(\overline{f}n)$.  
An RM-code for $G^{2}$ is then found as in \cite{kohlenbach4}*{\S4} using $\ACA_{0}$.  

\smallskip

To prove the first forward implication, $\IND$ implies that for any $n^{0}$, there is a finite binary sequence $\sigma$ such that 
\be\label{largene}
(\forall m\leq n)(\sigma(m)=1\asa (\exists f^{1})(Y(f, m)=0)), 
\ee
i.e.\ a kind of `finite comprehension' principle.  Suppose $\BOOT$ is false for $Y_{0}^{2}$, i.e.\
\[
(\forall X \subset \N)(\exists n\in \N)\big[(n\in X \wedge (\forall g^{1})(Y_{0}(g, n)>0)  ) \vee (n\not\in X \wedge (\exists h^{1})(Y_{0}(h, n)=0))  \big].  
\]
Observe that the content of $X$ beyond the number $n$ is irrelevant for the previous formula in square brackets.  Now define the $ \Sigma{\vee} \Pi$-formula $A(\sigma)$ as follows: 
\begin{gather*}
\sigma(|\sigma|-1)=1\wedge (\forall g^{1})(Y_{0}(g, |\sigma|-1)>0) \\\
\vee\\
 \sigma(|\sigma|-1)=0\wedge (\exists h^{1})(Y_{0}(h, |\sigma|-1)=0) .
\end{gather*}
Let $B(\sigma)$ be $A(\tilde{\sigma})$, where $\tilde{\sigma}(n)=1$ if $\sigma(n)>0$, and zero otherwise, for $n<|\sigma|$.  By assumption, we have $(\forall f^{1})(\exists n^{0})B(\overline{f}n)$.  
Apply $\Sigma{\vee} \Pi$-$\NFP$ and obtain an upper bound for the resulting RM-code on Cantor space (using $\WKL_{0}$ by \cite{simpson2}*{IV.2}).  However, this upper bound contradicts \eqref{largene} for large enough $n$.

\smallskip

To prove the second forward implication, proceed as in the previous part of the proof:  apply $\Sigma{\vee} \Pi$-$\NFP_{0}$ to $(\forall f^{1})(\exists n^{0})B(\overline{f}n)$ and let $\Phi$ be the resulting functional. 
Obtain a finite sub-cover for the associated canonical cover $\cup_{f\in 2^{\N}}[\overline{f}\Phi(f)]$ using $\HBU$.  
This provides an upper-bound that contradicts \eqref{largene} for large enough $n$.
\end{proof}
The theorem is not an isolated case:  inspired by \cite{simpson2}*{V.5.2}, $\STR$ is equivalent to comprehension 
for $\Sigma\wedge\Pi$-formulas $\varphi(i, X)$ with continuous second-order parameter as in $\STR$ and satisfying $(\forall i\in \N)(\exists \textup{ at most one } X\subseteq \N)\varphi(i, X)$.
The related statement for trees in \cite{simpson2}*{V.5.2} can also be generalised, similar to $C_{i}$-$\WKL$ below. 

\smallskip

Secondly, we obtain an equivalence result for the Lindel\"of lemma for $\N^{\N}$ and $\NFP$ restricted to $\Sigma$-formulas.  
Note that the Lindel\"of lemma is studied in detail in \cite{dagsamIII, dagsamV}, including a version that implies countable choice as in $\QFAC^{0,1}$. 
The final part of $\LIND(\N^{\N})$ indeed invites the application of the latter.  
\bdefi[$\LIND(\N^{\N})$] 
For every $G^{2}$, there is a sequence $\sigma_{n}^{0\di 0^{*}}$ covering $\N^\N$ such that $(\forall n \in\N)(\exists f \in \N^{\N})(\sigma_{n}=_{0^{*}}\overline{f}G(f))$.  
\edefi
Let $\Sigma$-$\NFP$ be $\NFP$ restricted to $\Sigma$-formulas.  The following theorem should be contrasted with the fact that comprehension 
does not capture $\LIND(\N^{\N})$ well\footnote{The system $\Z_{2}^{\omega}+\QFAC^{0,1}$ cannot prove $\LIND(\N^{\N})$, while $\Z_{2}^{\Omega}+\QFAC^{0,1}$ can.  However, the latter has the first-order strength of $\Z_{2}$, while $\RCAo+\LIND(\N^{\N})$ is conservative over $\RCA_{0}$.}.  
\begin{thm}\label{hulix}
The system $\RCAo+\A_{0}$ proves  $\LIN(\N^{\N})\asa \Sigma\textup{-}\NFP$.
\end{thm}
\begin{proof}
For the reverse implication, consider $A(\sigma)$ defined as follows:
\be\label{bongka3}
(\exists g^{1})\big[G(g)=_{0} |\sigma|-1 \wedge \sigma=_{0^{*}} \overline{g}G(g)\big],
\ee
We have $(\forall f\in \N^{\N})(\exists n^{0})A(\overline{f}n)$ since we may take $g=f$ and $n=G(f)$.
Apply $\Sigma\textup{-}\NFP$ to obtain $(\forall f\in \N^{\N})A(\overline{f}\gamma(f))$ for some $\gamma\in K_{0}$.
The sequence required by $\LIND(\N^{\N})$ is given by $\overline{\sigma}\gamma(\sigma*00\dots)$ for all $\sigma^{0^{*}}$ such that $\gamma(\sigma*00\dots)\geq |\sigma|-1$, which can be formed in $\RCAo$

\smallskip

For the forward direction, in case $\neg(\exists^{2})$, $\Sigma$-$\NFP$ and $\LIND(\N^{\N})$ are outright provable as all functions on Baire space are continuous by \cite{kohlenbach2}*{\S3}.  
In case $(\exists^{2})$, we may replace quantification over $2^{\N}$ by quantification over $\N^{\N}$ as in the proof of Theorem~\ref{keslich}.  
Hence, the antecedent of $\Sigma$-$\NFP$ reduces to $(\forall f^{1})(\exists n^{0}, g\in 2^{\N})(Y(g, \overline{f}n)=0)$.
Now apply $\A_{0}$ to obtain $\Phi^{2}$ such that $(\forall f^{1})(\exists g\in 2^{\N})(Y(g, \overline{f}\Phi(f))=0)$.
Let $\sigma_{n}^{0\di 0^{*}}$ be the sequence obtained from applying $\LIND(\N^{\N})$ to $\cup_{f\in \N^{\N}}[\overline{f}\Phi(f)]$.
Then apply $\QFAC^{1,0}$ to $(\forall f^{1})(\exists n^{0})(f\in [\sigma_{n}]  )$ and obtain (continuous by definition) $\Psi^{2}$ which produces the least such $n^{0}$.  
Finally define $Z^{2}$ as follows: $Z(f):=|\sigma_{\Psi(f)}|$ and note that by \cite{kohlenbach4}*{\S4}, this continuous function has an associate $\gamma\in K_{0}$.  
The latter is as required by $\Sigma$-$\NFP$, and we are done. 
\end{proof}
Let $\Sigma$-$\NFP_{\upharpoonright C}$ be $\Sigma$-$\NFP$ with all quantifiers over $\N^{\N}$ restricted to $2^{\N}$.  
One then proves the following corollary in the same way 
(also with $\IND$ replaced by $\QFAC^{0,1}$).  
\begin{cor}\label{durft}
The system $\RCAo+\IND+\A_{0}$ proves  $\HBU\asa [\Sigma\textup{-}\NFP_{\upharpoonright C}+\WKL]$.
\end{cor}
Thirdly, let $\BOOT_{\w}$ be $\BOOT$ with the quantifier over $\N^{\N}$ restricted to $2^{\N}$.  We have 
the following nice splitting for $\BOOT$, while the same result for $\HBU$ does not seem to follow without additional axioms; this was the initial motivation for $\A_{1}$, which yields         
an equivalence in the final part by Corollary \ref{firstofmany}.
\begin{thm}\label{hoerlepiep}
The system $\RCAo$ proves $[\ACA_{0}+\BOOT_{\w}]\asa \BOOT$; $\RCAo+\IND$ and $\RCAo+\QFAC^{0,1}$ both prove $[\WKL+\BOOT_{\w}]\di \HBU$.
\end{thm}
\begin{proof}
The first reverse implication is immediate.  The first forward implication is immediate in case $\neg(\exists^{2})$, as $\BOOT$ reduces to $\ACA_{0}$.
In case $(\exists^{2})$, $(\exists f^{1})(Y(f,n)=0)$ can be equivalently written as $(\exists X\subset \N^{2})(Y(F(X), n)=0)$ where $F(X)(k):=\lambda k.(\mu m^{0})((k,m)\in X)$.  Clearly, $\BOOT_{\w}$ applies to this equivalent formula. 

\smallskip

For the final implication, we prove $\HBU_{\c}$.  
Fix $G^{2}$ and let $A(\sigma)$ be the following:
\be\label{bongka2}
(\exists g\in C)\big[G(g)\leq |\sigma| \wedge \sigma*00\dots \in [\overline{g}G(g)]\big],
\ee
where $\sigma^{0^{*}}$ is a finite sequence of natural numbers.  Note that the formula in \eqref{bongka2} in square brackets is quantifier-free.  Thus, $\BOOT_{\w}$ provides a set $X\subseteq \N$ such that $(\forall \sigma^{0^{*}})(\sigma  \in X\asa A(\sigma) )$, with minimal coding.  Now, we have $(\forall f\in C)(\exists n^{0})A(\overline{f}n)$ since we may take $g=f$ and $n=G(f)$.  Hence, we have $(\forall f\in C)(\exists n^{0})(\overline{f}n\in X)$ and applying $\QFAC^{1,0}$, there is $H^{2}$ such that $(\forall f\in C)(\overline{f}H(f)\in X)$ and $H(f)$ is the least such number.  Obviously $H^{2}$ is continuous on $C$ and hence bounded above on $C$ by \cite{kohlenbach4}*{\S4}.  Hence, there is $N_{0}^{0}$ such that $(\forall f\in C)(\exists n\leq N_{0})A(\overline{f}n)$.
Now obtain the finite sub-cover as in the proof of Theorem \ref{honor}.
\end{proof}
One readily adapts the final part of the proof to $[\WWKL+\BOOT_{\w}]\di \WHBU$, where the latter captures (the essence of) the Vitali covering theorem for uncountable covers, as studied at length in \cite{dagsamVI}.

\smallskip

Clearly, $\BOOT_{\w}$ readily generalises to more general formulas only involving quantifiers over Cantor space.   
The following theorem implies that such formulas can `almost' be treated as quantifier-free.
\begin{thm}\label{coref}
The system $\WKL_{0}$ proves $[\BOOT_{\w}]_{\ECF}$ and $[\A_{2}]_{\ECF}$.
\end{thm}
\begin{proof}
Let $\MUC$ be the intuitionistic fan functional from \cite{kohlenbach2}*{\S3} as defined in Remark \ref{memmen2}.
The system $\RCAo+\MUC$ readily proves $\BOOT_{\w}$ and $\A_{1}$.  
By \cite{longmann}*{p.~497}, $\ECF$ converts $\MUC$ into $\WKL$.
\end{proof}
As promised, $\A_{2}$ yields an equivalence in the final part of Theorem \ref{hoerlepiep}.
\begin{cor}\label{firstofmany}
$\RCAo+\IND$ proves $[\WKL+\BOOT_{\w}]\asa [\HBU+\A_{2}]$.
\end{cor}
\begin{proof}
We first prove the reverse implication.  To this end, suppose $\neg\BOOT_{\w}$, i.e.\ there is $Y_{0}^{2}$ such that for all $X\subseteq\N$, there is $n\in \N$ such that
\be\label{castiel}
\big[ [n\in X\wedge (\forall f\in C)(Y_{0}(f,n)\ne 0)]\vee  [ (\exists g\in C)(Y_{0}(g,n)= 0)\wedge n\not\in X]  \big].
\ee
Let $A(\overline{X}n)$ be the formula in \eqref{castiel} (modulo minimal modification).  
Clearly, $A$ is a $C_{2}$-formula and $\A_{2}$ yields $\Phi^{2}$ such that $ (\forall X\subseteq\N)A(\overline{X}\Phi(X))$.
Apply $\HBU_{\c}$ for the canonical cover associated to $\Phi$.  The resulting finite sub-cover provides an upper bound $k_{0}$ such that $ (\forall X\subseteq\N)(\exists n\leq k_{0})A(\overline{X}n)$.
However, $\IND$ proves `finite comprehension' for any $C_{2}$-formula, and $\HBU\di \BOOT_{\w}$ follows.
For the forward implication, note that for any $C_{2}$-formula $A(\sigma^{0^{*}})$, there is $X\subset\N$ such that $\sigma \in X\asa A(\sigma)$ by $\BOOT_{\w}$.  Now apply $\QFAC^{1,0}$ to $(\forall f^{1})(\exists n^{0})A(\overline{f}n)$.
\end{proof} 
The previous splitting provides a nice motivation for $\A_{2}$, but there are other arguments in favour of the latter:  
we now use this axiom to show that $\HBU$ is equivalent to weak K\"onig's lemma \emph{generalised} to binary trees where elementhood in the tree is given by a $C_{1}$-formula.  
As to prior art, Kohlenbach studies similar generalisations of weak K\"onig's lemma in \cite{kohlenbach4}*{\S3}.  
\bdefi[$C_{i}$-tree]
We say that a $C_{i}$-formula $A(\sigma^{0^{*}})$,  is (or: represents) a `\emph{$C_{i}$-tree} $T$' if the formula $\sigma^{0^{*}} \in T\equiv \neg A(\sigma)$ satisfies the usual tree property, i.e.\ $\sigma \in T\di \tau \in T$ for any initial segment $\tau $ of $\sigma$.  A $C_{i}$-tree $T$ is \emph{infinite} if $(\forall n^{0})(\exists \sigma^{0^{*}})(|\sigma|=n\wedge \sigma \in T)$ 
and $f^{1}$ is \emph{a path in a $C_{i}$-tree} $T$ is $(\forall n^{0})(\overline{f}n\in T)$.
\edefi
\bdefi[$C_{i}$-$\WKL$]
Any infinite binary $C_{i}$-tree has a path.  
\edefi
The $\ECF$-translation of $C_{i}$-$\WKL$ is $\WKL$ by the following and Remark \ref{unbeliever}.
\begin{thm}
For $i=0,1,2$, $\RCA_{0}$ proves $\WKL\asa [C_{i}\text{-}\WKL]_{\ECF}$.   
\end{thm}
\begin{proof}
By \cite{longmann}*{p.\ 497}, $\WKL$ is equivalent to the $\ECF$-translation of the intuitionistic fan functional as in $\MUC$.  
The latter reduces finding $f\in 2^{\N}$ satisfying $(Y(f, \sigma)>_{0}0)$ to a finite search, i.e.\ elementhood in $C_{2}$-trees is decidable, reducing it to usual $\WKL$.
Alternatively, $\MUC$ readily implies comprehension for $C_{2}$-formulas, and the $\ECF$-interpretation of the former is just $\WKL$.
\end{proof}
At the risk of pedantry (and repetition by Remark \ref{unbeliever}), identifying continuous functions and their codes is second nature in RM.  
In this light, there is no difference between $C_{1}$-trees (under $\ECF$) and `normal' trees in second-order arithmetic (assuming $\WKL$), i.e.\ $[C_{1}\text{-}\WKL]_{\ECF}$ \emph{is} just $\WKL$ if 
we are identifying continuous functions and their codes; our identification however goes in the `reverse' direction. 
\begin{thm}\label{the100}
The systems $\RCAo+\QFAC^{0,1}$ and $\RCAo+\IND$ both prove the implication $C_{1}$-$\WKL\di \HBU$.
\end{thm}
\begin{proof}
Fix $G^{2}$ and define the formula $A(\sigma)$ as \eqref{bongka2}.
Clearly, $ A(\sigma)\di  A(\tau)$ for finite binary sequences $\sigma, \tau$ where $\sigma$ is an initial segment of $\tau$.  
In this light, the formula $  \neg A(\sigma)$ defines a $C_{1}$-tree $T$.  
Note that $(\forall f\in 2^{\N})(\exists n^{0})A(\overline{f}n)$ by considering $g=f$ and $n=G(f)$.
Apply $C_{1}$-$\WKL$ to $(\forall f\in 2^{\N})(\exists n^{0})(\overline{f}n\not\in T)$ to conclude $(\forall f\in 2^{\N})(\exists n^{0}\leq n_{0})(\overline{f}n\not\in T)$ for some $n_{0}$.  
Hence, $\HBU$ for $G^{2}$ follows as in the proof of Theorem \ref{hoerlepiep}.
\end{proof}
\begin{cor}\label{secondofmany}
The system $\RCAo+\IND+\A_{0}$ proves $\HBU\asa C_{1}\text{-}\WKL$.
\end{cor}
\begin{proof}
Apply $\A_{0}$ to $(\forall f\in C)(\exists n^{0})(\overline{f}n\not\in T)$ and apply $\HBU$ to the (canonical cover for the) resulting functional $\Phi^{2}$.
The resulting finite sub-cover readily provides the bound required by $C_{1}$-$\WKL$.
\end{proof}
The contraposition of $C_{1}$-$\WKL$ can be interpreted as a version of the Heine-Borel theorem for countable covers of closed sets as in the previous section.  
We establish the following where $C_{1}$-$\NFP$ is $\NFP$ restricted to $C_{1}$-formulas.  
\begin{thm}
The system $\RCAo+\IND$ proves the following: 
\begin{gather}\label{nojumps}
[\WKL+\BOOT_{\w}]\asa [\HBU+\A_{2}]\asa C_{2}\text{-}\WKL\asa [\WKL+C_{1}\textup{-}\NFP].
\end{gather}
\end{thm}
\begin{proof}
The first equivalence is given by Corolary \ref{firstofmany}.
For the second equivalence, we only need to prove $C_{2}$-$\WKL\di \BOOT_{\w}$.  Now assume $(\exists^{2})$ and define:
\[
A(\sigma^{0^{*}}) \equiv (\forall i<|\sigma|)\big[ \sigma(i)=0 \asa (\exists f\in 2^{\N})(Y(f, i)=0)   \big].
\]
With minor modification (using $\exists^{2}$), this formula yields a binary $C_{2}$-tree called $T$.  Using $\IND$ to establish `finite comprehension' for $C_{2}$-formulas, the $C_{2}$-tree $T$ is infinite.
A path through $T$ then immediately yields the required instance of $\BOOT_{\w}$.  In case $\neg(\exists^{2})$, one uses (the proof of) \cite{kohlenbach4}*{Prop.\ 4.10} to replace all (continuous)
functionals by RM-codes on $2^{\N}$.  The equivalence between $\WKL$ and $[\MUC]_{\ECF}$ from \cite{longmann}*{p.~497} and the implication $\MUC\di \BOOT_{\w}$ then finish this part.

\smallskip

The implication $[\WKL+\BOOT_{\w}]\di C_{1}\textup{-}\NFP$ follows from the final part of the proof of Corollary~\ref{firstofmany} by noting that the function $\Phi^{2}$ from $\QFAC^{1,0}$ can be taken to be continuous in this case. 
The associated associate has a trivial definition (thanks to $\BOOT_{\w}$).
Finally, the implication $[\WKL+C_{1}\textup{-}\NFP]\di \HBU$ readily follows as in the proof of Theorem \ref{hoerlepiep} by noting that \eqref{bongka2} is a $C_{1}$-formula.
\end{proof}
As to similar results, we could obtain equivalences involving $\MCT_{\net}^{-}$ from Remark \ref{memmen2}.
Moreover, $\ACA_{0}$ is equivalent to \emph{K\"onig's lemma} (see \cite{simpson2}*{III.7}), and one can also obtain an equivalent between the latter for $\Sigma{\vee}\Pi$-formulas and $\BOOT$. 
It goes without saying that certain implications from \eqref{nojumps} can also be obtained for $\WWKL$ and $\WHBU$ (see \cite{dagsamVI} and \cite{simpson2}*{X.1}).

\smallskip

Next, we prove that $\Delta$-comprehension indeed follows from $\A_{0}$.
\begin{thm}\label{tirlydiddy}
The system $\RCAo+\IND+\A_{0}$ proves $\Delta$-comprehension. 
\end{thm}
\begin{proof}
In case $\neg(\exists^{2})$, all functions on $\N^{\N}$ are continuous by \cite{kohlenbach2}*{Prop.\ 3.4}, and $\Delta$-comprehension reduces to $\Delta_{1}^{0}$-comprehension.
In case $(\exists^{2})$, we may replace in $\Delta$-comprehension the quantifiers over $\N^{\N}$ by quantifiers over $C$ as in the proof of Theorem \ref{keslich}.  
Now suppose there are $Y_{0}^{2}, Z_{0}^{2}$ satisfying the antecedent of $\Delta$-comprehension such that for all $X\subseteq\N$, there is $n\in \N$ such that
\be\label{castiel2}
\big[ [n\in X\wedge (\forall f\in C)(Y_{0}(f,n)\ne 0)]\vee  [ (\exists g\in C)(Y_{0}(g,n)= 0)\wedge n\not\in X]  \big], 
\ee
and denote by $A(\overline{X}n)$ the formula \eqref{castiel2} (modulo the usual modification).  At first glance, $A(\sigma^{0^{*}})$ is a $C_{2}$-formula, but since $Y_{0}, Z_{0}$ satisfy the 
antecedent of $\Delta$-comprehension, \eqref{castiel2} is in fact in $C_{0}$ and so is the formula $B(\sigma)\equiv [A(\sigma)\wedge (\forall i<|\sigma|-1)\neg A(\overline{\sigma}i)]$, which is readily proved using $\IND$.
Again using $\IND$, $(\forall X\subseteq \N)(\exists n^{0})A(\overline{X}n)$ implies $(\forall X\subseteq \N)(\exists n^{0})B(\overline{X}n)$, and applying $\A_{0}$ yields a \emph{continuous} $\Phi^{2}$, by the definition of $B$. 
Now, $\Phi^{2}$ has an upper bound on $C$ by $\WKL$, and this yields a contradiction as $\IND$ proves `finite comprehension' \eqref{largene}.  
\end{proof}
\noindent
We note in passing that $\HBU$ deals with uncountable covers, while $\WKL$ (up to coding as in \cite{simpson2}*{IV.1}) deals with countable covers, \emph{and never the twain shall meet}: the logical hardness of 
the former is dwarfed by the latter (see \cite{dagsamIII, dagsamV}).  Despite this huge difference, a slight generalisation of the scope of $\A_{2}$, namely closure under $\wedge, \neg, \di$, yields an axiom
that establishes an equivalence between $\WKL$ and $\HBU$, based on the previous proof. 

\smallskip

Finally, as promised, we obtain improved versions of Theorems \ref{klato} and \ref{sosimple}.
\begin{thm}\label{sosimple2}
The system $\RCAo+\A_{0}$ proves $\HBU\asa \CIT$.
\end{thm}
\begin{proof}
The equivalence amounts to the associated second-order result (see \cite{ishi1}) in case $\neg(\exists^{2})$.
In case $(\exists^{2})$, we may replace the use of $\QFAC^{1,1}$ in the proof of Theorem \ref{sosimple} by $\A_{0}$, 
as $\Sigma$-formulas can now be written as $C_{0}$-formulas as in the first part of the proof of Theorem \ref{hoerlepiep}.  Note that since $C_{n}$ is a sequence of closed $\Pi$-sets, 
the formula $(\forall x\in I)(\exists n^{0}){(x\not \in C_{n})}$ implies $(\forall x\in I)(\exists m^{0}){([x](m)\not \in C_{m})}$ as the complement is open and $C_{n+1}\subseteq C_{n}$. 
\end{proof}
\begin{cor}\label{klato2}
The system $\RCAo+\IND+\A_{0}$ proves $\HBU\di\Sigma$-$\SEP$. 
\end{cor}
\begin{proof}
As in the proof of the theorem, the use of $\QFAC^{1,1}$ in the proof of Theorem~\ref{klato} is replaced by $\A_{1}$.
\end{proof}
In light of the results obtained in this section, as well as the attendant discussion, a base theory for higher-order RM as in the Plato hierarchy should include at least $\A_{0}$ in addition to $\QFAC^{1,0}$ to be found in $\RCAo$.
\emph{However}, $\A_{2}$ also makes $\HBU$ `much more explosive': while $\ACAo+\HBU$ seems to prove no second-order theorem beyond $\ATR_{0}$ (see \cite{dagsamVI, dagsamIII}), 
adding $\A_{1}$ immediately results in $\FIVE$ by \eqref{nojumps}.  This is however unproblematic as the Plato hierarchy is intended to yield the G\"odel hierarchy under $\ECF$, i.e.\ 
no fragment of the  former hierarchy implies the existence of discontinuous functions, as $\ECF$ translates such fragments to `$0=1$'. 

\begin{ack}\rm
We thank Dag Normann, Adrian Mathias, Thomas Streicher, Pat Muldowney, and Anil Nerode for their valuable advice.
Our research was supported by the John Templeton Foundation via the grant \emph{a new dawn of intuitionism} with ID 60842 and by the \emph{Deutsche Forschungsgemeinschaft} via the DFG grant SA3418/1-1.
Opinions expressed in this paper do not necessarily reflect those of the John Templeton Foundation.    
The results in Section~\ref{main2} were completed during the stimulating BIRS workshop (19w5111) on Reverse Mathematics at CMO, Oaxaca, Mexico in Sept.\ 2019.  
We express our gratitude towards the aforementioned institutions.  Results in Appendix \ref{AAA} are due to myself and Dag Normann. 
\end{ack}
\appendix
\section{Other reflections of the Big Five}\label{AAA}
After the completion of this paper and \cite{dagsamVII, dagsamX}, it was noticed that the results in the latter two papers
also give rise to `reflections' similar to (but different from) Figure \ref{kk}.  We sketch these results in this section for completeness.

\smallskip

For this section, we stress that the concept of `open set' used in \cite{dagsamVII} is different from the one in this paper.  
Open sets are namely represented in \cite{dagsamVII} via \emph{characteristic functions}, yielding Figure \ref{OFG} as below. 
We first discuss some definitions as follows. 
\begin{enumerate}
\item Open sets in $\R$ are represented in \cite{dagsamVII} by $Y:\R\di \R$ where `$x\in Y$' is $|Y(x)|>_{\R}0$ and satisfies $(\forall x\in Y)(\exists r>_{\R} 0)(B(x, r)\subset Y)$.  Closed sets are the complement of open sets. \label{zelf}
\item $\HBC$ expresses that countable open covers of closed sets (as in item \eqref{zelf}) in the unit interval have finite sub-covers.
\item $\CLO$ expresses that a closed set (as in item \eqref{zelf}) in the unit interval is located.
\item $\CLO_{\RM}$ expresses that an RM-closed set in the unit interval is located.
\item $\open$ expresses that an open set (as in item \eqref{zelf}) has an RM-code, and is equivalent to the Urysohn lemma for closed sets (as in item \eqref{zelf}).    
\end{enumerate}
We chose item \eqref{zelf} as the definition of open set in \cite{dagsamVII} as it reduces to the usual RM-definition under $\ECF$ and sequential compactness behaves as for the RM-definition; elementhood in such open sets is also $\Sigma_{1}^{0}$ \emph{with parameters}. 
With the above in place, the following picture emerges from \cite{dagsamVII}.
We note that $\ECF$ converts the equivalences on the right to those on the left.  
\\
\vspace{3.2cm}
\begin{figure}[h]
\setlength{\unitlength}{5pt}
\begin{picture}(60,0)

\put(0,-15){\vector(0,1){36}}

\put(-0.5,-13.5){{\line(1,0){1}}}
\put(1,-14){{$\RCA_{0}$}}

\put(-0.5,-3.5){{\line(1,0){1}}}
\put(1,-4){{$\WKL_{0}$}}

\put(-0.5,6.5){{\line(1,0){1}}}
\put(1,6){{$\textup{\CLO}_{\RM}$}}

\put(-0.5,15){{\line(1,0){1}}}
\put(1,14.5){{$\textup{\ATR}_{0}$}}

\put(-0.5,19.5){{\line(1,0){1}}}
\put(1,19){$\FIVE$}

\put(34.5,19.5){{\line(1,0){1}}}
\put(35,19){$\begin{array}{c}\FIVE  +\open \end{array} $}

\put(7,-14){\small proves {$\Delta_{1}^{0}$-comprehension}}

\put(7, -4){\small $\asa$ countable Heine-Borel}
\put(7, -6){\small $\asa$ a continuous function  }
\put(9, -8){\small  on $2^{\N}$ is bounded}

\put(7, 6){\small $\asa\ACA_{0}$ }

\put(7,14.5){\small $\asa$ perfect set theorem for}
\put(7,12.5){\small closed sets as countable unions}

\put(8,19){\small $\asa $ Cantor-Bendixson for}
\put(6.5,17){\small closed sets as countable unions}

\put(35,-15){\vector(0,1){36}}


\put(34.5,-13.5){{\line(1,0){1}}}
\put(37,-14){{$\RCA_{0}^{\omega}$}}

\put(34.5,-3.5){{\line(1,0){1}}}
\put(36.5,-4){{$\WKL_{0}$}}
\put(34.5,6.5){{\line(1,0){1}}}
\put(36,6){{${\CLO}$}}
%
\put(36,13.5){{$\ATR_{0}+\open$}}
\put(34.5,14.5){{\line(1,0){1}}}
%
\put(35, -15){{\multiput(-2.5,1)(0,1){36}{\line(0,1){0.5}}}}

\put(3,-20){\small \textbf{second-order arithmetic}}

\put(40,-20){\small \textbf{higher-order arithmetic}}

\put(43,-14){\small plus $\QFAC^{0,1}$}

\put(42,-4){$\asa\HBC$ }
\put(42, -6){\small $\asa$ Pincherle's theorem $\PIT_{o}$}

\put(42,6){\small $\asa[\ACA_{0}+\open]$ }


\put(48,13.5){\small $\asa$ perfect set theorem for}
\put(49,11.5){\small closed sets as in item \eqref{zelf}}

\put(49.5,19){\small $\asa $ Cantor-Bendixson for}
\put(50,17){\small closed sets as in item \eqref{zelf}}

\put(28,-20.5){\small {\Huge $ \longleftarrow$}}
\put(28,-18){\small {\huge $\ECF$}}
\put(28,4){\small {\Huge $ \longleftarrow$}}
\put(28,7){\small {\huge $\ECF$}}
\end{picture}
\vspace{3.5cm}
\caption{Another higher-order hierarchy mapping to the Big Five}
\label{OFG}
\end{figure}
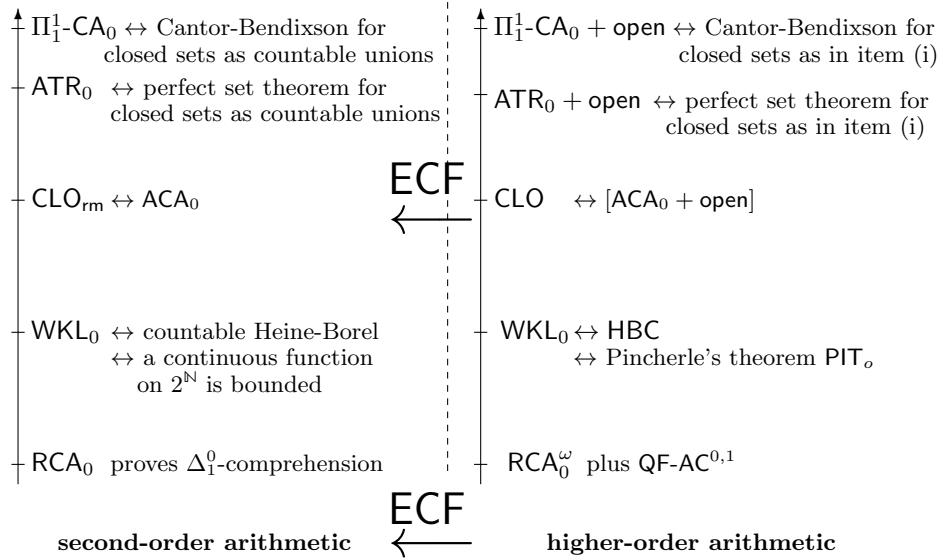~\\
Clearly, Figure \ref{OFG} is not as well-developed as Figure \ref{kk}, 
but then the motivation underlying \cite{dagsamVI} was never to obtain
a hierarchy that yields the Big Five and equivalences under $\ECF$.

\bye